\documentclass[11pt,letterpaper,twoside,reqno,final]{amsart}

\usepackage{fixltx2e}
\usepackage[usenames,dvipsnames]{xcolor}
\usepackage{fancyhdr}
\usepackage{amsmath,amsfonts,amsbsy,amsgen,amscd,mathrsfs,amssymb}
\usepackage{amsthm}
\usepackage{url}
\usepackage{eurosym}
\usepackage{tikz}
\usetikzlibrary{matrix,arrows,shapes}
\usepackage{subfig}
\usepackage{microtype}
\usepackage{enumitem}

\definecolor{dark-gray}{gray}{0.3}
\definecolor{dkgray}{rgb}{.4,.4,.4}
\definecolor{dkblue}{rgb}{0,0,.5}
\definecolor{medblue}{rgb}{0,0,.75}
\definecolor{rust}{rgb}{0.5,0.1,0.1}

\usepackage[colorlinks=true]{hyperref}
\hypersetup{urlcolor=Blue}
\hypersetup{citecolor=Black}
\hypersetup{linkcolor=dark-gray}

\usepackage{graphicx}
\usepackage{booktabs,longtable,tabu} 
\usepackage{multirow} 
\usepackage{float}
\usepackage[T1]{fontenc}

\usepackage{fourier}
\usepackage{charter}
\usepackage{bm} 

\graphicspath{{figures/}}

\numberwithin{equation}{section} 

\providecommand{\mathbold}[1]{\bm{#1}}  

\newtheorem{theorem}{Theorem}[section]
\newtheorem{lemma}[theorem]{Lemma}

\newtheorem{proposition}[theorem]{Proposition}

\newtheorem{corollary}[theorem]{Corollary}

\theoremstyle{definition}

\newtheorem{definition}[theorem]{Definition}
\newtheorem{example}[theorem]{Example}
\newtheorem{remark}[theorem]{Remark}

\renewcommand{\phi}{\varphi}

\newcommand{\e}{\varepsilon}

\renewcommand{\mid}{\mathrel{\mathop{:}}} 

\newcommand{\zerovct}{\vct{0}} 

\newcommand{\Id}{\mathbf{I}}

\newcommand{\ball}[1]{B^{#1}}

\providecommand{\mathbbm}{\mathbb} 

\newcommand{\R}{\mathbbm{R}}

\newcommand{\polar}{\circ}

\newcommand{\prtensor}{\,\hat{\otimes}\,}

\newcommand{\minimize}{\text{minimize}\quad}
\newcommand{\subjto}{\quad\text{subject to}\quad}

\newcommand{\argmin}{\operatorname*{arg\; min}}

\newcommand{\Prob}{\mathbbm{P}}

\newcommand{\Expect}{\operatorname{\mathbb{E}}}

\newcommand{\St}{\operatorname{St}}

\newcommand{\normal}{\textsc{Normal}}

\newcommand{\diag}{\operatorname{diag}}

\newcommand{\vct}[1]{\mathbold{#1}}
\newcommand{\mtx}[1]{\mathbold{#1}}

\newcommand{\transp}[1]{#1^{T}}

\renewcommand{\vec}{\operatorname{vec}}

\newcommand{\Proj}{\ensuremath{\mtx{\Pi}}} 

\newcommand{\ip}[2]{\langle {#1}, {#2} \rangle}

\newcommand{\norm}[1]{\Vert {#1}\Vert}

\newcommand{\lone}[1]{\norm{#1}_{\ell_1}}

\DeclareMathOperator{\dist}{dist}

\newcommand{\enorm}[1]{\norm{#1}}

\newcommand{\Desc}{\mathcal{D}}

\newcommand{\sdim}{\delta}
\newcommand{\sdimw}{\delta^*}

\DeclareMathOperator{\Circ}{Circ}

\newcommand{\IR}{\mathbbm{R}}
\newcommand{\veps}{\varepsilon}

\newcommand{\mC}{\mathcal{C}}

\newcommand{\mE}{\mathcal{E}}

\newcommand{\inter}{{\operatorname{int}}}
\newcommand{\conv}{\operatorname{conv}}
\newcommand{\clconv}{\overline{\conv}}

\newcommand{\GL}{\operatorname{GL}}

\newcommand{\tdash}{\overset{a}{\to}}

\newcommand{\K}{\mathcal{K}}

\newcommand{\D}{\mathcal{D}}
\renewcommand{\P}{\mathcal{P}}
\DeclareMathOperator{\id}{id}
\DeclareMathOperator{\gr}{gr}
\newcommand{\vp}{\varphi}

\DeclareMathOperator{\spa}{span}
\DeclareMathOperator{\vol}{vol}
\DeclareMathOperator{\ima}{im}

\newcommand{\Ren}{\mathcal{R}}
\newcommand{\T}{\mathcal{T}}
\newcommand{\resTMP}[2]{#1\to#2}
\newcommand{\res}[3]{\vct{#1}_{\resTMP{#2}{#3}}}
\newcommand{\nres}[3]{\|\vct{#1}\|_{\resTMP{#2}{#3}}}
\newcommand{\sres}[3]{\sigma_{\resTMP{#2}{#3}}(\vct{#1})}

\newcommand{\nrest}[3]{\|\vct{#1}^T\|_{\resTMP{#2}{#3}}}
\newcommand{\srest}[3]{\sigma_{\resTMP{#2}{#3}}(\vct{#1}^T)}

\newcommand{\srestm}[3]{\sigma_{\resTMP{#2}{#3}}(-\vct{#1}^T)}

\newcommand{\nresdag}[3]{\|\vct{#1}^\dagger\|_{\resTMP{#2}{#3}}}

\newcommand{\RCD}[3]{\Ren_{#2,#3}(\vct{#1})}

\newcommand{\DA}{\vct{\Delta A}}

\newcommand{\llangle}{\langle\hspace{-1.5pt}\langle}  
\newcommand{\rrangle}{\rangle\hspace{-1.5pt}\rangle}  

\title[Gordon's inequality and condition numbers in conic optimization]{Gordon's inequality and condition numbers in conic optimization}

\author[D.~Amelunxen]{Dennis Amelunxen}
\author[M.~Lotz]{Martin Lotz}

\date{\today}

\begin{document}

\begin{abstract}
The probabilistic analysis of condition numbers has traditionally been approached from different angles; one is based on Smale's 
program in complexity theory and features integral geometry, while the other is motivated by geometric functional analysis and makes use of the theory of
Gaussian processes. 
In this note we explore connections between the two approaches in the context of the biconic homogeneous feasiblity problem
and the condition numbers motivated by conic optimization theory.
Key tools in the analysis are Slepian's and Gordon's comparision inequalities for Gaussian processes, interpreted as monotonicity properties of moment functionals,
and their interplay with ideas from conic integral geometry.
\end{abstract}

\maketitle

\section{Introduction}\label{sec:intro}
This article deals with certain aspects of non-asymptotic random matrix theory and integral geometry that are motivated by applications in conic optimization theory, compressed sensing, and linear inverse problems. The objects of study are the norm,
the smallest singular value, and condition numbers of Gaussian random operators between convex cones.

Driven by various types of applications, the probabilistic analysis of condition numbers has been of interest to two seemingly different cultures. 
The first culture derives its raison d'\^{e}tre from Steve Smale's program of a condition-based complexity theory~\cite{Smale81,Sm:97,BuergICM,Condition}, with the aim of developing an average-case or smoothed analysis
of condition numbers, and therefore, of the running time of numerical algorithms that depend on them.
The second culture has roots in geometric functional analysis, and seeks to identify bounds on extreme singular values, and hence condition numbers, that are satisfied with overwhelming probability.
This second strand has found its way into the sensitivity analysis for compressed sensing and related problems~\cite{rudelson2010non,Vershynin2012,CRPW:12,FR:13}. 
One purpose of this article is to explore connections between these two points of view, and to develop a common perspective.
Our setting is a generalization of the homogeneous conic feasibility problem: given a closed convex cone $C\subseteq \R^m$ and $\mtx{A}\in \R^{n\times m}$, does there exist a nonzero $\vct{x}\in C$ such that $\mtx{A}\vct{x}=\vct{0}$?
Instances and variations of this problem include the linear, quadratic, and semidefinite programming feasibility problems, as well as versions of the nullspace condition in compressed sensing and generalizations thereof.
The associated quantities of interest are the norm, singular value, and Renegar's condition number~\cite{rene:95a} of $\mtx{A}$ restricted to $C$.
An important tool in the study of the norms and singular values of Gaussian operators are the inequalities of Slepian and Gordon~\cite{G:85}. We define a class of orthogonal invariant valuations
on the set of convex bodies, the {\em moment functionals}, and interpret Slepian's inequality, as well as a generalization to higher moments, as monotonicity property of these functionals. As a consequence we get new
bounds on higher moments of the norm of cone restricted random operators involving Renegar's condition number, as well as bounds on invariants such as the Gaussian width and the statistical dimension of linear images of convex cones. By generalizing these valuations to convex bundles we obtain a similar geometric interpretation of Gordon's comparison theorem. Finally, we discuss the benefits and limitations of Gaussian comparison inequalities as a tool for the classic probabilistic analysis of condition numbers, and explore relations to the geometric theory based on conic integral geometry and spherical tube formulas.

Before presenting a more detailed description in Section~\ref{sec:intro-Gordon-conrelinop}--\ref{sec:intro-conic-integr-geom}, 
we start out with an illustration of the above-mentioned two different practices in analyzing the probabilistic behavior of condition numbers. For a real-world metaphor, imagine an overdone pizza; to judge its quality you can either focus on the burnt rim, or on the 
mean color of the topping.

\subsection{Two viewpoints on condition numbers}\label{sec:2-viewpoints}

Condition numbers measure the computational difficulty (in a general sense) of an input for a given numerical problem: the higher the condition of an input the harder it is for the machine to solve it. It is often the case that a specific set, the so-called \emph{ill-posed inputs}, is particularly hard, if not numerically impossible to solve. This is usually a zero-volume set, and in many cases of interest the condition number is inversely proportional to the (normalized) distance to the set of ill-posed inputs. To estimate the likeliness of an input being badly conditioned one may either put the emphasis on the set of ill-posed inputs and show that the volume of the tube around it does not grow very fast, or one may try to show that most inputs are ``reasonably well conditioned''. This second approach is a bit delicate, as the condition of a random input may indeed have infinite expectation; only its inverse, the (normalized) distance, which is $1$-Lipschitz by the reverse triangle inequality, always has a first (and 
higher) moment(s). But using concentration of measure arguments one may argue that most inputs lie at a certain distance from the set of ill-posed inputs and deduce from that tail estimates for the condition number.

To make matters a bit more concrete, assume that the input space is~$\IR^d$, the ill-posed inputs form a scale invariant hypersurface $\Sigma\subset \IR^d$, i.e., $\Sigma$ has codimension one in~$\IR^d$
and $\lambda\Sigma=\Sigma$ for all $\lambda>0$, and the condition number of an input $\vct x\in \IR^d$ is given by
\begin{equation}\label{eq:conic-cond}
  \mC (\mtx{x}) = \frac{\norm{\vct{x}}}{\dist(\vct{x},\Sigma)} ,
\end{equation}
where $\dist$ denotes the Euclidean distance.
Denoting by $\vct{g}\sim\normal(\vct 0,\Id_d)$ a Gaussian vector in~$\IR^d$, we are interested in upper bounds for the tail
$\Prob\{\mC(\vct{g})\geq \lambda\}$. Note that
by scale invariance we can assume $\norm{\vct x} = 1$,
so denoting by $\vct \theta\in S^{d-1}$ a uniformly random point on the sphere, we have $\Prob\{\mC(\vct{g})\geq \lambda\}=\Prob\{\mC(\vct{\theta})\geq \lambda\}$.

{\bf (1)} A typical analysis of~$\mC$ motivated by complexity theoretic questions lays its focus on capturing the tail behavior of~$\mC$ at infinity, or, equivalently, the volume (with respect to the probability measure) of an $\veps$-tube around~$\Sigma$ for small~$\veps>0$. In this regime we have
\begin{align*}
   \Prob\{\mC(\vct{\theta})\geq t \} \stackrel{[\veps=1/t]}{=} \Prob\{\dist(\vct{\theta},\Sigma)\leq \e \} = \e\,\vol\Sigma + (\text{higher order terms in } \e) \approx t^{-1}\vol\Sigma ,
\end{align*}
where $\vol\Sigma$ denotes the lower-dimensional volume of~$\Sigma$ induced by the Gaussian measure. 
Typically, it is the logarithm of the condition number that drives complexity bounds for numerical algorithms.
A bound of the form $\Prob\{\mC(\vct{\theta})\geq t \} \leq \tfrac{c_1}{t}$ for some constant $c_1>0$
allows to estimate the expected value of the logarithm of $\mC$ (see~\cite[Prop.~2.4]{BCL:06a}):
\begin{equation}\label{eq:E[log(C(x))]<...}
  \Expect[\log \mC(\vct{\theta})] \leq 1+\log c_1 .
\end{equation}
Such an estimate can be used to derive bounds for the expected running time of numerical algorithms~\cite{Condition}.

{\bf (2)} The second approach focuses on the distribution of the distance to ill-posedness. We denote its expectation by $\mu:=\Expect[\dist(\vct\theta,\Sigma)]$. Keep in mind that the condition number itself~$\mC(\vct\theta)$ has infinite expectation. The concentration of measure phenomenon~\cite{ledo:01} yields
\begin{equation*}
  \Prob\{ \dist(\vct{\theta},\Sigma)\leq \mu-\lambda \}\leq c_2 \exp(-c_3 d \lambda^2) ,
\end{equation*}
for some constants $c_2,c_3>0$, or equivalently, setting $t:=(\mu-\lambda)^{-1}$,
\begin{equation}\label{eq:simple-concentration}
  \Prob\{ \mC(\vct{\theta})\geq t \}\leq c_2 \exp\big(-c_3 d \big(\mu-\tfrac{1}{t}\big)^2\big) .
\end{equation}
An estimate of this form allows precise statements about the order of magnitude of $\mC(\vct{\theta})$ for growing~$d$.
However, observe that for fixed~$d$ (no matter how large) the right-hand side in~\eqref{eq:simple-concentration} does not converge to zero as $t\rightarrow \infty$. This deficiency, among other things, rules out estimates of the logarithm of the condition number as achieved in~\eqref{eq:E[log(C(x))]<...} by the direct method.

\subsubsection{The upshot}

We have chosen the spherical setting in the above discussion as this provides the clearest connection between the two described viewpoints. In particular, in view of the second approach it is often more convenient to work in the Gaussian setting; also because not all condition numbers are exactly of the form~\eqref{eq:conic-cond}, some are of the form
\begin{equation}\label{eq:gen-cond}
  \kappa (\mtx{x}) = \frac{\norm{\vct{x}}_*}{\dist(\vct{x},\Sigma)} ,
\end{equation}
where $\norm{\vct x}_*$ is a more general norm (or Minkowski functional). In this case the spherical and the Gaussian setting are still closely related but not the same. A prominent example of a non-conic condition number is the classical matrix condition: here $\norm{\vct x}$ corresponds to the Frobenius norm, while $\norm{\vct x}_*$ corresponds to the operator norm, which is the more common choice. 
From the above discussion, in particular from~\eqref{eq:simple-concentration} it should have become clear that a central problem is to have \emph{lower bounds} for the expected distance to ill-posedness~$\Expect[\dist(\vct g,\Sigma)]$, and in the non-conic setting one additionally needs \emph{upper bounds} for the expected norm~$\Expect[\norm{\vct g}_*]$. Important tools to achieve such bounds are the inequalities of Slepian and Gordon, which we discuss next.

\subsection{Conically restricted operators and Gordon's inequalities}\label{sec:intro-Gordon-conrelinop}

The classical condition number of a matrix\footnote{Contrary to tradition, we denote the number of rows by $n$ and the number of columns by $m$, so that, as an operator, $\vct A\colon\IR^m\to\IR^n$.} $\vct A\in\IR^{n\times m}$
 is the ratio of the 
operator norm and the smallest singular value. Using the notation
\begin{align*}
   \norm{\vct A} & := \max_{\vct x\in S^{m-1}} \norm{\vct{Ax}} ,\hspace{2cm} \sigma(\vct A) := \min_{\vct x\in S^{m-1}} \norm{\vct{Ax}} ,
\end{align*}
the smallest singular value of~$\vct A$ is given by $\max\big\{\sigma(\vct A),\sigma(\vct A^T)\big\}$, so that the classical condition number is given by
  \[ \kappa(\vct A) = \min\Bigg\{ \frac{\norm{\vct A}}{\sigma(\vct A)},\frac{\norm{\vct A}}{\sigma(\vct A^T)}\Bigg\} . \]
With a view towards the convex feasibility problem, cf.~Section~\ref{sec:intro-conv-feas}/\ref{sec:gen-feas-prob}, we introduce the following generalization: Let $C\subseteq\IR^m$, $D\subseteq\IR^n$ be closed convex cones, and let $\vct A\in\IR^{n\times m}$. We define the \emph{restriction} of the linear operator~$\vct A$ to $C$ and $D$ by
\begin{equation}\label{eq:def-res-lin-op}
  \res{A}{C}{D}\colon C\to D ,\qquad \res{A}{C}{D}(\vct x) := \Proj_D(\vct{Ax}) ,
\end{equation}
where $\Proj_D\colon\IR^n\to D$ denotes the orthogonal projection, i.e., $\Proj_D(\vct y) = \argmin\{ \|\vct y-\vct z\|\mid \vct z\in D\}$. Accordingly, we define restricted versions of the 
norm and the singular value:
\begin{align}\label{eq:def-nres,sres}
   \nres{A}{C}{D} & := \max_{\vct x\in C\cap S^{m-1}} \|\res{A}{C}{D}(\vct x)\| , &
   \sres{A}{C}{D} & := \min_{\vct x\in C\cap S^{m-1}} \|\res{A}{C}{D}(\vct x)\| .
\end{align}
In Section~\ref{sec:con-res-lin-op} we will give a geometric interpretation of these quantities and describe how they appear in applications.

A standard way to bound these quantities for Gaussian operators is by means of Slepian's and Gordon's inequalities~\cite{G:85,LT:91,davidson2001local}, as in the following theorem. As the first of the stated bounds does not seem available in the literature, we provide a proof in Appendix~\ref{sec:proof-comp-thm}.

\begin{theorem}\label{thm:moments-restr-singvals}
Let $C\subseteq\IR^m$, $D\subseteq\IR^n$ closed convex cones, and let $\vct G\in\IR^{n\times m}$ be a Gaussian matrix and $\vct g\in\IR^m$, $\vct g'\in\IR^n$ independent Gaussian vectors. Then for $f\colon\IR\to\IR$ monotonically increasing and convex,
\begin{align}
   \Expect\big[ f(\nres{G}{C}{D})\big] & \leq \Expect\big[ f\big(\|\Proj_D(\vct g')\| + \|\Proj_C(\vct g)\|\big)\big] .
\label{eq:moments-restr-singvals-claim-norm-conv}
\intertext{If $\gamma\in\IR$ denotes a standard Gaussian variable, which is independent of~$\vct G$, then for every monotonically increasing~$f$,}
   \Expect\big[ f(\nres{G}{C}{D}+\gamma)\big] & \leq \Expect\big[ f\big(\|\Proj_D(\vct g')\| + \|\Proj_C(\vct g)\|\big)\big] ,
\label{eq:moments-restr-singvals-claim-norm}
\\ \Expect\big[ f(\sres{G}{C}{D})+\gamma\big] & \geq \Expect\big[ f\big(\|\Proj_D(\vct g')\|-\|\Proj_C(\vct g)\|\big)\big] .
\label{eq:moments-restr-singvals-claim-sval}
\end{align}
\end{theorem}

As a corollary from~\eqref{eq:moments-restr-singvals-claim-norm-conv} we obtain another inequality that turns out to be useful in the area of conic integral geometry that we describe in Section~\ref{sec:intro-conic-integr-geom} below.

\begin{corollary}
Let $C\subseteq\IR^m$, $D\subseteq\IR^n$ closed convex cones, and let $\vct T\in\IR^{\ell\times m}$ and $\vct U\in\IR^{p\times n}$. Then for $r\geq1$,
\begin{equation}\label{eq:bod-lin-img}
   \Expect\big[ \|\tilde{\vct G}\|_{\vct TC\to\vct UD}^r\big] \leq \kappa(\vct T)^r\,\kappa(\vct U)^r\,\Expect\big[ \nres{G}{C}{D}^r\big] ,
\end{equation}
where $\tilde{\vct G}\in\IR^{p\times \ell}$ and $\vct G\in\IR^{n\times m}$ are Gaussian matrices.
\end{corollary}

In fact, the above corollary holds with Renegar's condition number instead of the usual matrix one, cf.~Proposition~\ref{prop:moments-res.norm}.
As special cases we obtain condition based estimates of moment functionals such as the Gaussian width and the statistical dimension (see Section~\ref{sec:mom-func}). The idea of using Slepian's lemma to obtain condition number estimates for the Gaussian width of linear images of convex cones was suggested to us by Michael~B.~McCoy.

We will see that~\eqref{eq:bod-lin-img} may fail for~$r<1$. This will also show that~\eqref{eq:moments-restr-singvals-claim-norm-conv} may fail if~$f$ is not convex.

\subsection{The convex feasibility problem}\label{sec:intro-conv-feas}

The primal and dual (homogeneous) feasibility problems with reference cone~$C\subseteq\IR^m$ are the decision problems
\\\def\tmpX{3mm}
\begin{minipage}{0.46\textwidth}
\begin{align}
   \exists \vct x & \in C\setminus\{\vct0\} \quad\text{s.t.} \hspace{\tmpX} \vct{Ax}=\vct0 ,
\tag{P}
\end{align}
\end{minipage}
\rule{0.07\textwidth}{0mm}
\begin{minipage}{0.46\textwidth}
\begin{align}
   \exists \vct y & \in\IR^n\setminus\{\vct0\} \quad\text{s.t.} \hspace{\tmpX} -\vct A^T\vct y\in C^\polar ,
\tag{D}
\end{align}
\end{minipage}
\\[2mm] where $\vct A\in\IR^{n\times m}$ and $C^\polar = \{\vct z\in\IR^m\mid \langle \vct x,\vct z\rangle \leq 0 \text{ for all } \vct x\in C\}$ denotes the polar cone of~$C$.
Special cases of interest in conic optimization are when~$C$ is the non-negative orthant, the second-order cone, or the cone of positive semidefinite matrices~\cite{BV:04}. Other cases of interest, that include compressed sensing, are when~$C$ is the descent cone of a convex regularizer, in which case (the negation of)~(P) is sometimes referred to as a \emph{nullspace condition}.

In the spirit of~\cite{rene:95a} we consider the following convex feasibility problem in the setting with two nonzero closed convex cones $C\subseteq\IR^m$, $D\subseteq\IR^n$:
\\\def\tmpX{3mm}
\begin{minipage}{0.46\textwidth}
\begin{align}
   \exists \vct x & \in C\setminus\{\vct0\} \quad\text{s.t.} \hspace{\tmpX} \vct{Ax} \in D^\polar ,
\tag{P}
\label{eq:(Pintro)}
\end{align}
\end{minipage}
\rule{0.07\textwidth}{0mm}
\begin{minipage}{0.46\textwidth}
\begin{align}
   \exists \vct y & \in D\setminus\{\vct0\} \quad\text{s.t.} \hspace{\tmpX} -\vct A^T\vct y\in C^\polar .
\tag{D}
\label{eq:(Dintro)}
\end{align}
\end{minipage}
\\[2mm] 
Note that we have a complete symmetry between~(P) and~(D) via the exchange of $\vct A$ by $-\vct A^T$.
We denote the set of primal feasible instances and the set of dual feasible instances by
\begin{align*}
   \P(C,D) & := \{\vct A\in\IR^{n\times m}\mid \text{(P) is feasible}\} , & \D(C,D) & := \{\vct A\in\IR^{n\times m}\mid \text{(D) is feasible}\},
\end{align*}
and we call
  \[ \Sigma(C,D) := \P(C,D)\cap \D(C,D) \]
the set of \emph{ill-posed} inputs. Indeed, we will see that $\P(C,D)$ and $\D(C,D)$ are both closed, the union of $\P(C,D)$ and $\D(C,D)$ is the whole input space~$\IR^{n\times m}$ unless $C=D=\IR^m$, and the probability that a Gaussian matrix lies in $\Sigma(C,D)$ is zero (Section~\ref{sec:gen-feas-prob}).

As for the relation of this generalized feasibility problem to conically restricted linear operators, observe that $\vct A\in\P(C,D)$ if and only if $\sres{A}{C}{D}=0$, and $\vct A\in\D(C,D)$ if and only if $\srestm{A}{D}{C}=0$. Moreover, we will see in Section~\ref{sec:gen-feas-prob} that
\begin{align*}
  \dist(\vct A,\P) & = \sres{A}{C}{D} , & \dist(\vct A,\D) & = \srestm{A}{D}{C} ,
\end{align*}
where $\dist$ denotes the Euclidean distance.

The restricted singular value has also found applications in the context of linear inverse problems~\cite{CRPW:12}.
More precisely,
consider the problem of recovering an unknown signal $\vct{x}_0\in \R^m$ from noisy observations $\vct{z}_0=\mtx{A}\vct{x}_0+\vct{w}$, with $\mtx{A}\in \R^{n\times m}$ and $\norm{\vct{w}}\leq \e$, by solving the
optimization problem
\begin{equation*}
 \minimize \|\vct{x}\|_* \subjto \norm{\mtx{A}\vct{x}-\vct{z}_0}\leq \e,
\end{equation*}
where $\|.\|_*$ is a convex function, usually a suitably chosen norm (a typical example is when $\vct{x}_0$ is sparse and $\|\vct{x}\|_*=\lone{\vct{x}}$). If $\hat{\vct{x}}$ is a solution of the above problem and $C=\Desc(\|.\|_*,\vct{x}_0)$ denotes the cone of descent
directions of $\|.\|_*$, then the error~$\norm{\hat{\vct{x}}-\vct{x}_0}$ is bounded by~\cite[Prop.~2.2]{CRPW:12}
\begin{equation*}
 \norm{\hat{\vct{x}}-\vct{x}_0} \leq \frac{2\e}{\sres{A}{C}{\R^n}}.
\end{equation*}
The more natural setting where the norm of the perturbation is proportional to the norm of the measurement matrix, 
$\norm{\vct{w}}\leq \e\norm{\mtx{A}}$, leads directly to the consideration of a common condition number for conic optimization,
to be introduced next.

Renegar's condition number is given by
\begin{equation*}
  \RCD{A}{C}{D} := \frac{\|\vct A\|}{\dist(\vct A,\Sigma(C,D))} = \min\Bigg\{\frac{\|\vct A\|}{\sres{A}{C}{D}},\frac{\|\vct A\|}{\srestm{A}{D}{C}}\Bigg\} . 
\end{equation*}
Besides the above mentioned application, Renegar's condition number has originally been used to estimate the running time of interior point algorithms that solve the convex feasibility problem (in the case $D=\IR^n$). An average-case analysis of this condition number has been given in~\cite{AB:13} (for previous analyses in the context of linear programming see~\cite{Condition} and the references therein).
This analysis does not use Gordon's inequalities, but relies on certain quantities stemming from the domain of integral geometry.

In the following section we provide a short overview of the main ideas of conic integral geometry. One aim of this paper is to highlight connections between the exact theory of integral geometry and the approximate methods centering around the concentration of measure phenomenon, and to provide some first applications of Gordon's inequalities in this neighboring area.

\subsection{Conic integral geometry}\label{sec:intro-conic-integr-geom}

The theory of conic integral geometry centers around the \emph{intrinsic volumes} $v_0(C),\ldots,v_m(C)$, which are assigned to every closed convex cone $C\subseteq\IR^m$. They form a discrete probability distribution on $\{0,\ldots,m\}$ that captures statistical properties of the cone~$C$. 
For example, the (conic) Steiner formula~\eqref{eq:Steiner-conic} describes the Gaussian measure of a neighborhood of $C$, while the kinematic formulas~\cite{edge} describe the exact intersection probabilities of randomly oriented cones.
Instead of giving the exact definition of these quantities (see Section~\ref{sec:rel-intr-vols}), we remark that the moment generating function of this discrete probability distribution coincides, after a simple variable transformation, with the moment generating function of~$\|\Pi_C(\vct g)\|$, where~$\vct g$ denotes as usual a standard Gaussian vector in~$\IR^m$ and~$\Pi_C$ denotes the orthogonal projection on~$C$, as shown by McCoy and Tropp~\cite{MT:13}.
Moreover, the expectation of the discrete probability distribution given by the intrinsic volumes, which is called the~\emph{statistical dimension} of~$C$, coincides with the 
expectation of the squared projected length $\Expect\big[\|\Pi_C(\vct{g})\|^2\big]$. This evinces a close relation between the statistical dimension of~$C$ and the Gaussian width of~$C\cap S^{m-1}$, see Section~\ref{sec:gwidth-edim-sdim} for a discussion.

In Section~\ref{sec:rel-intr-vols} we present an alternative perspective on conic integral geometry through \emph{moment functionals}. 
The moment functionals evaluated in a convex body~$K$ describe the moments of $h_K(\vct g)$, where $h_K(\vct{x})=\max_{\vct{z}\in K}\ip{\vct{z}}{\vct{x}}$ is the support function of~$K$ (see Section~\ref{sec:conv-geom-restr-norm-sval}) and $\vct g$ is a Gaussian vector,
\begin{equation*}
 \mu_f(K) = \Expect\big[f(h_K(\vct{g}))\big] .
\end{equation*}
Slepian's lemma and its extensions can be interpretated as monotonicity properties of moment functionals with respect to contractions (Proposition~\ref{prop:slep-ext}); 
direct consequences of these monotonicity properties are the bounds~\eqref{eq:moments-restr-singvals-claim-norm-conv},~\eqref{eq:moments-restr-singvals-claim-norm}, and~\eqref{eq:bod-lin-img}.
The moment functionals generalize to support functions of {\em convex bundles} (see Section~\ref{sec:contr-ineq-bdls}), which leads to the inequality~\eqref{eq:moments-restr-singvals-claim-sval}.

The main motivation for introducing moment functionals was to give a geometric interpretation of Slepian's and Gordon's inequalities and to simplify the use of these inequalities. Moreover, this concept highlights the difference between the largest and the smallest singular value of a (conically restricted) operator and it helps in establishing connections to conic and Euclidean integral geometry.

\subsection{Context and contributions}

We finish this introduction by putting the work into a broader context and describe the specific contributions made.

\subsubsection{Condition numbers and random matrices}\label{sec:condnumbs-randommatr}
Motivated by problems from nuclear physics, classical random matrix theory has traditionally been concerned with the limiting distributions 
of the spectrum of a matrix as the dimension grows to infinity.
For a square Gaussian $(m\times m)$-matrix~$\mtx{G}$ one has the asymptotics $\|\vct G\|\sim\sqrt{m}$ and $\sigma(\vct G)\sim\frac{1}{\sqrt{m}}$,
whereas for a rectangular $(n\times m)$-matrix, $n>m$, one has asymptotically (with $\frac{m}{n}$ converging to a limit in~$(0,1)$) $\|\vct G\|\sim\sqrt{n}+\sqrt{m}$ and $\sigma(\vct G)\sim\sqrt{n}-\sqrt{m}$~\cite{bai1993,RV:09rect}.
Other fields of applications such as statistics, signal processing and numerical analysis, however, have led to an interest in properties of random matrices that hold for finite values of $n$ and $m$. 
The typical kind of questions asked within each of the two cultures mentioned in the beginning are:
\begin{enumerate}
\item How big is the expected logarithm of the condition number? How small is the probability that the condition number of a random matrix exceeds, say, $10^2$?
\item What is the probability that the extremal singular values of a random matrix are close to their asymptotic limiting values?
\end{enumerate}

For square matrices $n=m$, the precise distribution of the matrix condition number 
$\kappa(\mtx{G})=\norm{\mtx{G}}/\sigma(\mtx{G})$ of has been derived by Alan Edelman~\cite{Edelman88}. From his analysis, it follows that the expected logarithm of the condition number of $\mtx{G}$ satisfies the bounds
\begin{equation}\label{eq:log-cond}
  \Expect[\log \kappa(\mtx{G})] \leq \log m + 1.537.
\end{equation}
Moreover, the following type of tail estimates are known~\cite{azais2004upper,CD:05,Edelman2005tails},
\begin{equation}\label{eq:tail-cond}
  c t^{-(n-m+1)}\leq \Prob\{\kappa(\mtx{G})\geq t\} \leq C t^{-(n-m+1)},
\end{equation}
with constants that 
depend polynomially on $m,n$.
As already remarked in~\eqref{eq:E[log(C(x))]<...}, a standard argument shows that an estimate of the form $\Expect[\log \kappa(\mtx{G})]=O(\log n)$ follows from~(\ref{eq:tail-cond}). 

By the Eckart-Young Theorem, the smallest singular value is the distance of a matrix to the set of singular matrices. 
Moreover, other condition numbers can be interpreted as (normalized) inverse distance to a set of ill-posed inputs as well~\cite{Demmel87}, if not defined in this way~\cite{rene:95b,rene:95a}.
Bounds of the type~(\ref{eq:tail-cond}) can thus be derived conveniently using {\em tubular neighborhoods}, as was pointed out by Demmel~\cite{Demmel87,Demmel88}.
This method works well
for estimating the condition number of square matrices and generalizations to the condition number of other problems.
A common generalization of average-case analysis is smoothed analysis~\cite{ST:02,Wsch:04,BCL:08,BuergICM,BC:10}, which studies the behavior of condition numbers under small perturbations.

From the point of view of the second culture, the focus of what is known as non-asymptotic random matrix theory (NRMT) does not
lie on condition numbers per se, but in deriving effective versions of the asymptotic limiting results for the spectra of random matrices,
in the sense that classical measure concentration results for the sums of random variables provide effective versions of the central limit theorem~\cite{rudelson2010non,Vershynin2012}.
A method of choice in this type of analysis when dealing with Gaussian matrices is Gordon's comparison inequality. 
Applied to the smallest and largest singular values, Gordon's theorem implies (e.g.,~\cite[Theorem 5.32]{Vershynin2012}),
\begin{equation*}
  \sqrt{n}-\sqrt{m} \leq \Expect[\sigma(\mtx{G})]\leq \Expect[\norm{\mtx{G}}]\leq \sqrt{n}+\sqrt{m}.
\end{equation*}
In combination with standard measure concentration, one obtains sharp bounds for the probability that the condition number is close to $C\cdot (\sqrt{n}+\sqrt{m})/(\sqrt{n}-\sqrt{m})$.

While smoothed analysis was the ``generalization of choice'' when focusing on tail estimates and expectation of condition numbers in the context of computational complexity theory, 
the results in 
NRMT tend to hold for {\em sub-Gaussian} random matrix ensembles. This is natural, given the analogy to concentration results for
sums of random variables as effective interpretations of the central limit theorem, which is obviously not restricted to Gaussian distributions. For an overview of results in this 
stream of developments, see~\cite{rudelson2010non}. We remark that also a hybrid of these generalizations, a smoothed analysis of the matrix condition number under rather general conditions on the distribution is available~\cite{TV:07}.

\subsubsection{Contributions}

The main motivation for this work was to show and clarify connections between non-asymptotic random matrix theory (NRMT), the theory of condition numbers, and integral geometry in the setting of convex cones
and operators between them.
To establish this we introduce the notion of \emph{conically restricted operators} and show that the methods from NRMT based on the concentration of measure phenomenon and the inequalities of Slepian and Gordon, which were previously used to analyze the spectrum of Gaussian matrices, naturally extend to this more general setting. On the other hand we introduce the \emph{biconic feasibility problem}, and show that the distance to primal/dual feasibility coincides with the singular value of the conically restricted linear operator. Having established this connection we clarify what kind of condition number estimates result from the above-mentioned methods from NRMT.

On the theoretical side, we introduce the concept of \emph{moment functionals} on convex bodies and convex bundles, and show that Slepian's lemma and Gordon's inequality can be interpreted as monotonicity properties of these moment functionals. This simplifies the use of these inequalities in finite dimensional settings. Furthermore, we establish a close connection between moment functionals and conic integral geometry. This connection, on the one hand, lets us apply Slepian's lemma to obtain some novel estimates of important integral-geometric quantities. On the other hand, we use integral-geometric arguments to obtain exact formulas for certain events for which NRMT arguments only provide bounds. We illustrate the usefulness of these exact formulas by providing an example that shows the necessity of an assumption in an extension of Gordon's inequality that we provide in this work, which also seems to be novel.

\subsection{Acknowledgments}
We thank Michael B.~McCoy for suggesting the use of Slepian's inequality 
in relation to linear images of cones, and for pointing out to us Maurer's article~\cite{M:11}, and Joel Tropp for useful comments.

\section{Conically restricted linear operators}\label{sec:con-res-lin-op}

In this section we discuss the restriction of a linear operator to closed convex cones. Our focus will not be on the restriction itself~\eqref{eq:def-res-lin-op}, but rather on the restricted norm and the restricted (smallest) singular value~\eqref{eq:def-nres,sres}.
In Section~\ref{sec:restr-norm-sval} we derive general properties of these quantities and compare them to the unrestricted versions, in Section~\ref{sec:gen-feas-prob} we establish a relation to the generalized homogeneous feasibility problem, and in Section~\ref{sec:conv-geom-restr-norm-sval} we derive a convex geometric perspective.

\subsection{Restricted norm and restricted singular value}\label{sec:restr-norm-sval}

Before discussing conically restricted operators, we record the following simple but useful lemma, which generalizes the relation $\ker\vct A=(\ima \vct A^T)^\bot$.

\begin{lemma}\label{lem:A^(-1)(D^*)}
Let $D\subseteq\IR^n$ closed convex cone. Then the polar cone is the inverse image of the origin under the projection map, $D^\polar := \{\vct z\in\IR^n\mid \langle \vct y,\vct z\rangle \leq 0 \text{ for all } \vct y\in D\} = \Proj_D^{-1}(\vct0)$. Furthermore, if $\vct A\in\IR^{n\times m}$, then
\begin{equation}\label{eq:A^(-1)(D^*)}
  \vct A^{-1}(D^\polar) = \big( \vct A^T D\big)^\polar ,
\end{equation}
where $\vct A^{-1}(D^\polar)=\{\vct x\in\IR^m\mid \vct{Ax}\in D^\polar\}$ denotes the inverse image of~$D^\polar$ under~$\vct A$.
\end{lemma}

\begin{proof}
For the first claim, note that $\|\Proj_D(\vct{z})\|=\max_{\vct{y}\in D\cap B^{n}}\ip{\vct{z}}{\vct{y}}$, and $\max_{\vct{y}\in D\cap B^{n}}\ip{\vct{z}}{\vct{y}}=0$ is equivalent to $\ip{\vct{z}}{\vct{y}}\leq 0$ for all $\vct y\in D$, i.e.,~$\vct{z}\in D^{\polar}$.

For~\eqref{eq:A^(-1)(D^*)}, let $\vct x\in\vct A^{-1}(D^\polar)$ and $\vct y\in D$. 
Then $\langle \vct x,\vct A^T\vct y\rangle = \langle \vct{Ax},\vct y\rangle\leq 0$, as $\vct{Ax}\in D^\polar$. Therefore, $\vct A^{-1}(D^\polar)\subseteq (\vct A^TD)^\polar$. On the other hand, if $\vct v\in (\vct A^TD)^\polar$ and $\vct{y}\in D$, then $\langle\vct{Av},\vct y\rangle=\langle\vct v,\vct A^T\vct y\rangle\leq 0$, so that $\vct{Av}\in D^\polar$ and hence, $(\vct A^TD)^\polar\subseteq\vct A^{-1}(D^\polar)$.
\end{proof}

Recall from~\eqref{eq:def-nres,sres} that for $\vct A\in\IR^{n\times m}$, $C\subseteq\IR^m$ and $D\subseteq\IR^n$ closed convex cones, the restricted norm and singular value of~$\vct A$ are defined by
$\nres{A}{C}{D} := \max\{\|\res{A}{C}{D}(\vct x)\|\mid \vct x\in C\cap S^{m-1}\}$ and $\sres{A}{C}{D} := \min\{\|\res{A}{C}{D}(\vct x)\|\mid \vct x\in C\cap S^{m-1}\}$, respectively, where $\res{A}{C}{D}(\vct x) = \Proj_D(\vct{Ax})$. The following proposition provides geometric conditions for the vanishing of the restricted norm or singular value.

\begin{proposition}\label{prop:nres=0,sres=0}
Let $\vct A\in\IR^{n\times m}$, $C\subseteq\IR^m$ and $D\subseteq\IR^n$ closed convex cones. Then the restricted norm vanishes, $\nres{A}{C}{D} = 0$, if and only if $C\subseteq (\vct A^TD)^\polar$. Furthermore, the restricted singular value vanishes, $\sres{A}{C}{D} = 0$, if and only if $C\cap (\vct A^TD)^\polar\neq\{\vct0\}$, which is equivalent to $\vct AC\cap D^\polar\neq\{\vct0\}$ or $\ker\vct A\cap C\neq\{\vct0\}$.
\end{proposition}

\begin{proof}
Using Lemma~\ref{lem:A^(-1)(D^*)} we have $\Proj_D(\vct{Ax})=\vct0$ if and only if $\vct{Ax}\in D^\polar$. This shows $\nres{A}{C}{D} = 0$ if and only if $\vct{Ax}\in D^\polar$ for all $\vct x\in C\cap S^{m-1}$, or equivalently, $C\subseteq \vct A^{-1}(D^\polar)=(\vct A^TD)^\polar$ by~\eqref{eq:A^(-1)(D^*)}. The claim about the restricted singular value follows similarly: $\sres{A}{C}{D} = 0$ if and only if $\vct{Ax}\in D^\polar$ for some $\vct x\in C\cap S^{m-1}$, or equivalently, $C\cap \vct A^{-1}(D^\polar)\neq\{\vct0\}$. If $\vct x\in C\cap \vct A^{-1}(D^\polar)\setminus\{\vct0\}$, then either $\vct{Ax}$ is nonzero or $\vct x$ lies in the kernel of~$\vct A$, which shows the second characterization.
\end{proof}

It is easily seen that the restricted norm is symmetric $\nres{A}{C}{D} = \nrest{A}{D}{C}$,
\begin{equation}\label{eq:symm-nres}
  \nres{A}{C}{D} = \max_{\vct x\in C\cap B^m}\max_{\vct y\in D\cap B^n} \langle \vct{Ax},\vct y\rangle = \max_{\vct y\in D\cap B^n}\max_{\vct x\in C\cap B^m} \langle \vct A^T\vct y,\vct x\rangle = \nrest{A}{D}{C} .
\end{equation}
Such a relation does not hold in general for the restricted singular value. 
In fact, in Section~\ref{sec:gen-feas-prob} we will see that, unless $C=D=\IR^m$, the minimum of~$\sres{A}{C}{D}$ and~$\srestm{A}{D}{C}$ is always zero, if~$C$ and~$D$ have nonempty interior, cf.~\eqref{eq:min(sres(A,C,D),srestm(A,D,C))=0}. And if~$C$ or~$D$ is a linear subspace then $\srestm{A}{D}{C}=\srest{A}{D}{C}$.

\begin{remark}
In the case $C=\IR^m$, $D=\IR^n$, with $n\geq m$, one can characterize the smallest singular value of~$\vct A$ as the inverse of the norm of the (Moore-Penrose) 
\emph{pseudoinverse} of~$\vct A$:
  \[ \sigma(\vct A) = \|\vct A^\dagger\|^{-1} . \]
Such a characterization does \emph{not} hold in general for the restricted singular value, i.e., in general one cannot write $\sres{A}{C}{D}$ as $\nresdag{A}{D}{C}^{-1}$. 
Consider for example the case $D=\IR^n$ and $C$ a circular cone of angle $\alpha$ around some center $\vct p\in S^{m-1}$. Both cones have nonempty interior, but 
letting $\alpha$ go to zero, it is readily seen that $\sres{A}{C}{D}$ tends to $\|\vct{Ap}\|$, while $\nresdag{A}{D}{C}$ tends to $\|\vct p^T\mtx{A}^\dagger\|$, which is in general 
not equal to~$\|\vct{Ap}\|^{-1}$, unless $\transp{\mtx{A}}\mtx{A}=\Id_m$.
\end{remark}

\subsubsection{Cone angles}

One can express the restricted singular value in terms of a restricted norm, if the matrix~$\vct A$ is \emph{well-conditioned} and ``tall and skinny''. We show this in the following proposition for which we need to prepare some notation: for nonzero cones $C,D\subseteq\IR^m$ denote the smallest angle between these cones by
\begin{align*}
   d(C,D) & := \min\big\{ \arccos\langle \vct x,\vct y\rangle\mid \vct x\in C\cap S^{m-1}, \vct y\in D\cap S^{m-1}\big\} ,
\intertext{and denote the capped angle by}
   \bar d(C,D) & := \min\big\{ \arccos\langle \vct x,\vct y\rangle\mid \vct x\in C\cap B^{m}, \vct y\in D\cap B^m\big\} = \min\big\{ d(C,D),\tfrac{\pi}{2}\big\} .
\end{align*}
The quantity $\cos \bar d(C,D)$ coincides with the quantity $\llangle C,D\rrangle$ introduced in~\cite{MT:13b} in the context of convex demixing.

\begin{lemma}\label{lem:nres,sres-I_m}
Let $C,D\subseteq\IR^m$ closed convex cones. Then
\begin{align*}
  \norm{\Id_m}_{C\to D} & = \cos \bar d(C,D) , & \sigma_{C\to D}(\Id_m) & = \sin \bar d(C,D^\polar) .
\end{align*}
\end{lemma}

\begin{proof}
We have
\begin{align*}
  \norm{\Id_m}_{C\to D} & = \max_{\vct{x}\in  C\cap S^{m-1}} \norm{\Proj_D(\vct{x})} = \max_{\vct{x}\in  C\cap B^{m}} \max_{\vct{y}\in D\cap B^m}\ip{\vct{x}}{\vct{y}} = \cos \bar d(C,D) ,
\intertext{and, using the Pythagorean identity $\|\vct x\|^2 = \|\Proj_D(\vct x)\|^2 + \|\Proj_{D^\polar}(\vct x)\|^2$,}
  \sigma_{C\to D}(\Id_m)^2 & = \min_{\vct{x}\in  C\cap S^{m-1}} \norm{\Proj_D(\vct{x})}^2 = \min_{\vct{x}\in  C\cap S^{m-1}} (1-\norm{\Proj_{D^\polar}(\vct{x})}^2)
\\ & = 1-\max_{\vct{x}\in  C\cap S^{m-1}} \norm{\Proj_{D^\polar}(\vct{x})}^2 = 1-\norm{\Id_n}_{C\to D^{\polar}}^2 = \sin^2 \bar d(C,D^{\polar}). \qedhere
\end{align*}
\end{proof}

\begin{proposition}\label{prop:angles}
Let $C\subseteq\IR^m$ and $D\subseteq\IR^n$ nonzero closed convex cones. If $\vct A\in\IR^{n\times m}$ satisfies 
$\vct A^T\vct A=\Id_m$ or $\vct A\vct A^T=\Id_n$, then
\begin{align}\label{eq:restr-norm-sval-well-cdt}
   \nres{A}{C}{D} & = \begin{cases} \cos \bar d(\vct AC,D) & \text{if } n\geq m \\ \cos \bar d(C,\vct A^TD) & \text{if } n\leq m , \end{cases} & \sres{A}{C}{D} & = \begin{cases} \sin \bar d(\vct AC,D^\polar) & \text{if } n\geq m \\ \sin \bar d(C,(\mtx{A}^{T}D)^\polar) & \text{if } n\leq m . \end{cases}
\end{align}
In particular, if $n\geq m$, then
  \[ \sres{A}{C}{D}^2 = 1 - \nres{A}{C}{D^\polar}^2 \; . \]
\end{proposition}

\begin{proof}
If $n\geq m$ then $\|\vct{Ax}\|=\|\vct x\|$ and hence, $\vct AC\cap S^{n-1}=\vct A(C\cap S^{m-1})$. It follows that
\begin{align*}
  \norm{\mtx{A}}_{C\to D} & = \max_{\vct{x}\in  C\cap S^{m-1}} \norm{\Proj_D(\vct{Ax})} =\max_{\vct{z}\in  \mtx{A}C\cap S^{n-1}} \norm{\Proj_D(\vct{z})} = \norm{\Id_n}_{\vct AC\to D} = \cos \bar d(\mtx{A}C,D)
\end{align*}
by Lemma~\ref{lem:nres,sres-I_m}. Analogous arguments show $\sres{A}{C}{D} = \sin \bar d(\vct AC,D^\polar)$.

If $n\leq m$ then $\nres{A}{C}{D} = \cos \bar d(C,\vct A^TD)$ follows directly from symmetry~\eqref{eq:symm-nres}. For $\sres{A}{C}{D}$ we compute, using $\vct A^T(D\cap B^n)=\vct A^TD\cap B^m$
\begin{align*}
  \sres{A}{C}{D} & = \min_{\vct x\in  C\cap S^{m-1}} \max_{\vct y\in D\cap B^n} \langle\vct{Ax},\vct y\rangle = \min_{\vct x\in  C\cap S^{m-1}} \max_{\vct y\in D\cap B^n} \langle \vct x,\vct A^T\vct y\rangle = \min_{\vct x\in  C\cap S^{m-1}} \max_{\vct z\in (\vct A^TD)\cap B^m} \langle \vct x,\vct z\rangle
\\ & = \min_{\vct{x}\in  C\cap S^{m-1}} \norm{\Proj_{\vct A^TD}(\vct{x})} = \norm{\Id_m}_{C\to \vct A^TD} = \sin \bar d(C,(\mtx{A}^TD)^\polar) . \qedhere
\end{align*}
\end{proof}

In the case $n\leq m$ the value of $\sres{A}{C}{D}$ depends on the position of $C$ relative to the kernel of~$\vct A$. For example, if $D=\IR^n$ and $\vct A\vct A^T=\Id_m$, then
  \[ \sres{A}{C}{\IR^n} = \sin \bar d(C,\ker\vct A) , \]
as $(\mtx{A}^{T}\IR^n)^\polar=\mtx{A}^{-1}(\{\vct0\})=\ker\vct A$ by~\eqref{eq:A^(-1)(D^*)}. This equality and further relations, in particular in connection to convex programming, is the focus of the following section.

\subsection{The biconic feasibility problem}\label{sec:gen-feas-prob}

Recall from Section~\ref{sec:intro-conv-feas} the convex feasibility problem in the setting with two nonzero closed convex cones $C\subseteq\IR^m$, $D\subseteq\IR^n$:
\\\def\tmpX{3mm}
\begin{minipage}{0.46\textwidth}
\begin{align}
   \exists \vct x & \in C\setminus\{\vct0\} \quad\text{s.t.} \hspace{\tmpX} \vct{Ax} \in D^\polar ,
\tag{P}
\label{eq:(P)}
\end{align}
\end{minipage}
\rule{0.07\textwidth}{0mm}
\begin{minipage}{0.46\textwidth}
\begin{align}
   \exists \vct y & \in D\setminus\{\vct0\} \quad\text{s.t.} \hspace{\tmpX} -\vct A^T\vct y\in C^\polar .
\tag{D}
\label{eq:(D)}
\end{align}
\end{minipage}
\\[2mm] Using Lemma~\ref{lem:A^(-1)(D^*)} and Proposition~\ref{prop:nres=0,sres=0} we obtain the following characterizations of the primal feasible matrices $\P(C,D) := \{\vct A\in\IR^{n\times m}\mid \text{(P) is feasible}\}$,
\begin{align}\label{eq:P(C,D)}
   \P(C,D) & \stackrel{\eqref{eq:A^(-1)(D^*)}}{=} \big\{\vct A\in\IR^{n\times m}\mid C\cap \big(\vct A^T D\big)^\polar\neq\{\vct0\}\big\} \stackrel{\text{[Prop.~\ref{prop:nres=0,sres=0}]}}{=} \{\vct A\in\IR^{n\times m}\mid \sres{A}{C}{D}=0 \} .
\end{align}
By symmetry, we obtain for the dual feasible matrices $\D(C,D) := \{\vct A\in\IR^{n\times m}\mid \text{(D) is feasible}\}$,
\begin{align}\label{eq:D(C,D)}
   \D(C,D) & = \{\vct A\in\IR^{n\times m}\mid D\cap (-\vct A C)^\polar\neq\{\vct0\}\} =  \{\vct A\in\IR^{n\times m}\mid \srestm{A}{D}{C}=0 \} .
\end{align}
In fact, we will see that $\sres{A}{C}{D}$ and $\srestm{A}{D}{C}$ can be characterized as the distances to $\P(C,D)$ and $\D(C,D)$, respectively. 
We defer the proofs for this section to Appendix~\ref{sec:appendix-convex}.

In the following proposition we collect some general properties of $\P(C,D)$ and $\D(C,D)$.

\begin{proposition}\label{prop:primaldual}
Let $C\subseteq\IR^m$, $D\subseteq\IR^n$ closed convex cones with nonempty interior. Then
\begin{enumerate}
  \item $\P(C,D)$ and $\D(C,D)$ are closed;
  \item the union of these sets is given by
     \[ \P(C,D)\cup\D(C,D) = \begin{cases} \{\vct A\in\IR^{m\times m}\mid \det\vct A=0\} & \text{if } C=D=\IR^m \\ \IR^{n\times m} & \text{else} ; \end{cases} \]
  \item the intersection of these sets is nonempty but has zero (Lebesgue) volume, i.e.,
     \[ \Prob\big\{\vct G\in\P(C,D)\cap\D(C,D)\big\} = 0 , \]
        where $\vct G\in\IR^{n\times m}$ Gaussian.
\end{enumerate}
\end{proposition}

Note that from~(2) and the characterizations~\eqref{eq:P(C,D)} and~\eqref{eq:D(C,D)} of $\P(C,D)$ and $\D(C,D)$, respectively, we obtain for every $\vct A\in\IR^{n\times m}$: $\min\{ \sres{A}{C}{D}, \srestm{A}{D}{C}\} = 0$ or, equivalently,
\begin{equation}\label{eq:min(sres(A,C,D),srestm(A,D,C))=0}
  \max\big\{ \sres{A}{C}{D}, \srestm{A}{D}{C}\big\} = \sres{A}{C}{D} + \srestm{A}{D}{C} ,
\end{equation}
unless $C=D=\IR^m$.

In the following we simplify the notation by writing $\P,\D$ instead of $\P(C,D),\D(C,D)$. For the announced interpretation of the restricted singular value as distance to $\P,\D$ we introduce the following notation: for $\vct A\in\IR^{n\times m}$ define
\begin{align*}
   \dist(\vct A,\P) & := \min\{\|\vct\Delta\|\mid \vct A+\vct\Delta\in\P\} , & \dist(\vct A,\D) & := \min\{\|\vct\Delta\|\mid \vct A+\vct\Delta\in\D\},
\end{align*}
where as usual, the norm considered is the operator norm. The proof of the following proposition, given in Appendix~\ref{sec:appendix-convex}, follows along the lines
of similar derivations in the case with a cone and a linear subspace~\cite{BF:09}.

\begin{proposition}\label{prop:sing-dist}
Let $C\subseteq\IR^m$, $D\subseteq\IR^n$ nonzero closed convex cones with nonempty interior. 
Then
\begin{align*}
  \dist(\vct A,\P) & =\sres{A}{C}{D} , & \dist(\vct A,\D) & =\srestm{A}{D}{C} .
\end{align*}
\end{proposition}

We finish this section by considering the intersection of~$\P$ and~$\D$, which we denote by
  \[ \Sigma(C,D) := \P(C,D)\cap \D(C,D) , \]
or simply $\Sigma$ when the cones are clear from context.
This set is usually referred to as the set of \emph{ill-posed inputs}.
As shown in Proposition~\ref{prop:primaldual}, the set of ill-posed inputs, assuming~$C\subseteq\IR^m$ and~$D\subseteq\IR^n$ each have nonempty interior, is a nonempty zero volume set. In the special case $C=\IR^m$, $D=\IR^n$,
  \[ \Sigma(\IR^m,\IR^n) = \{\text{rank deficient matrices in } \IR^{n\times m}\} . \]
From~\eqref{eq:min(sres(A,C,D),srestm(A,D,C))=0} and Proposition~\ref{prop:sing-dist} we obtain, if $(C,D)\neq(\IR^m,\IR^m)$,
  \[ \dist(\vct A,\Sigma) = \max\big\{\dist(\vct A,\P) , \dist(\vct A,\D)\big\} = \dist(\vct A,\P) + \dist(\vct A,\D) . \]

The inverse distance to ill-posedness forms the heart of Renegar's condition number~\cite{rene:94,rene:95b}. We denote
\begin{equation}\label{eq:renegars}
  \RCD{A}{C}{D} := \frac{\|\vct A\|}{\dist(\vct A,\Sigma(C,D))} = \min\Bigg\{ \frac{\|\vct A\|}{\sres{A}{C}{D}} , \frac{\|\vct A\|}{\srestm{A}{D}{C}} \Bigg\} . 
\end{equation}
Furthermore, we abbreviate the special case $D=\IR^n$, which corresponds to the classical feasibility problem, cf.~Section~\ref{sec:intro-conv-feas}, by the notation
\begin{equation}\label{eq:renegars-single_cone}
  \Ren_C(\vct A) := \RCD{A}{C}{\IR^n} .
\end{equation}
Note that the usual matrix condition number is recovered in the case $C=\IR^m$, $D=\IR^n$,
  \[ \Ren_{\IR^m}(\vct A) = \Ren_{\IR^m,\IR^n}(\vct A) = \kappa(\vct A) . \]
Another simple but useful property is the symmetry $\Ren_{C,D}(\vct A) = \Ren_{D,C}(-\vct A^T)$. Finally, note that the restricted singular value has the following monotonicity properties
\begin{align*}
   C\subseteq C' & \Rightarrow \sres{A}{C}{D}\geq \sres{A}{C'}{D} , & D\subseteq D' & \Rightarrow \sres{A}{C}{D}\leq \sres{A}{C}{D'} .
\end{align*}
This indicates that not necessarily $\Ren_C(\vct A)\leq\Ren_{C'}(\vct A)$ if $C\subseteq C'$. But in the case $C'=\IR^m$ and $n\geq m$ this inequality does hold, which we formulate in the following lemma.

\begin{lemma}
Let $C\subseteq\IR^m$ closed convex cone with nonempty interior and $\vct A\in\IR^{n\times m}$ with $n\geq m$. Then
\begin{equation}\label{eq:R_C(A)<=kappa(A)}
  \Ren_C(\vct A) \leq \kappa(\vct A) .
\end{equation}
\end{lemma}

\begin{proof}
In the case $C=\IR^m$ we have $\Ren_{\IR^m}(\vct A) = \kappa(\vct A)$. If $C\neq\IR^m$ then $\vct AC\neq\IR^n$, as $n\geq m$. It follows that $\IR^n\cap (-\vct AC)^\polar\neq\{\vct0\}$, and thus $\srestm{A}{\IR^n}{C}=0$, cf.~\eqref{eq:D(C,D)}. Hence,
  \[ \Ren_C(\vct A)=\frac{\|\vct A\|}{\sres{A}{C}{\IR^n}}\leq \frac{\|\vct A\|}{\sres{A}{\IR^m}{\IR^n}} = \kappa(\vct A) . \qedhere \]
\end{proof}

The interesting case for convex optimizations is in fact where $C$ is some self-dual cone like the nonnegative orthant or the cone of nonnegative definite matrices, and $n<m$. For these cases Renegar's condition number has found applications in the complexity analysis of convex optimization, see~\cite{Condition} for a discussion and references. For example,~\cite{VRP:07} provides an analysis of the running time of an interior-point algorithm for the convex feasibility problem in terms of this condition number. In Section~\ref{sec:lin_imag} we use $\Ren_C(\vct A)$ for upper bounds of some important moment functionals. 
Section~\ref{sec:comp-bounds} contains further details on the mentioned complexity analysis in convex optimization.

\subsection{Convex geometric interpretation}\label{sec:conv-geom-restr-norm-sval}
The restricted norm and the restricted singular value can be interpreted in terms of the support function of convex bodies. 
Recall that a set $K\subset\IR^m$ is a convex body if $K$ is nonempty, compact, and convex. The support function of a convex body $K\subset\IR^m$ is given by
  \[ h_K\colon\IR^m\to\IR ,\qquad h_K(\vct x) := \max_{\vct z\in K}\langle \vct x,\vct z\rangle . \]
If $C\subseteq\IR^m$ is a closed convex cone, and if $K=C\cap B^m$, where $\ball{m} := \big\{ \vct{x} \in \R^m : \enorm{\vct{x}} \leq 1 \big\}$, denotes the corresponding cone stub, then one readily verifies that
\begin{equation}\label{eq:|Pi_C(x)|=h_K(x)}
  \|\Pi_C(\vct x)\| = h_K(\vct x) ,\qquad \text{for all } \vct x\in\IR^m .
\end{equation}
If $K=\conv(C\cap S^{m-1})$ then one still has $\|\Pi_C(\vct x)\| = h_K(\vct x)$ for all $\vct x\not\in \inter(C^\polar)$, but in general one only gets an inequality:
\begin{equation}\label{eq:|Pi_C(x)|>=h_K(x)}
  \|\Pi_C(\vct x)\| \geq h_K(\vct x) ,\qquad \text{for all } \vct x\in\IR^m .
\end{equation}

\begin{remark}\label{rem:normprojC-h_K}
Interpreting $\vct{x}\in\IR^m$ as the linear map $\R\to \R^m$, $\lambda\mapsto \lambda\vct x$, we have the characterization $\norm{\Proj_C(\vct{x})} = \sigma_{\R_+\to C}(\vct{x})$.
Moreover, we have
  \[ \srestm{x}{C}{\IR_+} = \min_{\vct z\in C\cap S^{m-1}} \max\{-\langle\vct x,\vct z\rangle,0\} = \begin{cases} 0 & \text{if } \vct x\not\in C^\polar \\ -\max_{\vct z\in C\cap S^{m-1}} \langle\vct x,\vct z\rangle & \text{if } \vct x\in C^\polar . \end{cases} \]
Using this, we can write the support function of $K=\conv(C\cap S^{m-1})$ in the form
\begin{equation}\label{eq:h_kproj-error}
 h_K(\vct{x}) = \sres{x}{\R_+}{C}-\sigma_{C\to \R_+}(-\transp{\vct{x}}) .
\end{equation}
In particular, one can interpret the summands in~\eqref{eq:h_kproj-error} as positive and negative parts of $h_K(\vct{x})$, and write
\begin{equation*}
  \dist(\vct{x},\Sigma(\IR_+,C)) = \sres{x}{\R_+}{C}+\sigma_{C\to \R_+}(-\transp{\vct{x}}) = |h_K(\vct{x})|.
\end{equation*}
\end{remark}

The restricted norm is related to the following construction: for convex bodies $K\subset\IR^m$, $K'\subset\IR^n$ define the \emph{(convex) tensor product}
\begin{equation*}\label{ex:kron-prod} 
   K\prtensor K' := \conv \{ \vct x\otimes\vct y\mid \vct x\in K,\vct y\in K'\} \subset\IR^{mn} ,
\end{equation*}
where $\otimes$ denotes the Kronecker product $\vct x\otimes \vct y=(x_1y_1,\ldots,x_1y_n,x_2y_2,\ldots, x_my_n)$, which is a concrete model for the classical tensor product. An application of Carath\'eodory's theorem~\cite[(2.4)]{Barv} shows that $K\prtensor K'$ is again a convex body. If $\vec\colon\IR^{n\times m}\to\IR^{nm}$ denotes the function that concatenates the columns of the matrices, then
\begin{equation}\label{eq:<Ax,y>=<vec(A),xy>}
  \langle \vec(\vct A), \vct x\otimes\vct y\rangle = \langle \vct{Ax},\vct y\rangle .
\end{equation}
Hence, if $C\subseteq\IR^m$ and $D\subseteq\IR^n$ are closed convex cones, and if $K=C\cap B^m$ and 
$K' = D\cap B^n$ denote the corresponding cone stubs, then
\begin{align}\label{eq:nres(A,C,D)-h_(KxK')}
   \nres{A}{C}{D} & = \max_{\vct x\in C\cap S^{m-1}} \|\Pi_D(\vct{Ax})\| = \max_{\vct x\in K} \|\Pi_D(\vct{Ax})\| \stackrel{\eqref{eq:|Pi_C(x)|=h_K(x)}}{=} \max_{\vct x\in K,\vct y\in K'} \langle \vct{Ax},\vct y\rangle \stackrel{\eqref{eq:<Ax,y>=<vec(A),xy>}}{=} \max_{\vct x\in K,\vct y\in K'} \langle \vec(\vct A), \vct x\otimes\vct y\rangle
\nonumber
\\ & = h_{K\prtensor K'}(\vec(\vct A)) .
\end{align}
Note that instead of the cone stub $C\cap B^m$ we could have taken $K = \conv(C\cap S^{m-1})$.

\begin{example}
Choosing $C=\IR^m$, $D=\IR^n$, shows that the operator norm is given by the support function of $B^m\prtensor B^n$. 
As the dual of the operator norm is given by the Schatten-1 matrix norm $\|.\|_*$, which returns the sum of the singular values of a matrix, we obtain
\begin{equation}\label{eq:Kronprod-unit-balls}
  B^m\prtensor B^n = \{\vec(\vct A)\mid \vct A\in\IR^{m\times n} , \|\vct A\|_*\leq 1\}.
\end{equation}
In particular, the convex tensor product of two cone stubs is not necessarily a cone stub.
\end{example}

The restricted singular value has a similar description as the restricted norm in~\eqref{eq:nres(A,C,D)-h_(KxK')}. For this we need to introduce the concept of a {\em convex bundle}. In the following let the set of convex bodies in~$\IR^N$ be denoted by $\K(\IR^N) := \{K\subset\IR^N\mid \text{convex body}\}$.

\begin{definition}
A \emph{convex bundle} with compact base set $M\subset\IR^m$ is a set-valued map $F\colon M\to \K(\IR^N)$ such that the graph
  \[ \gr(F) = \bigcup_{\vct x\in M} \{\vct x\}\times F(\vct x) \subset \IR^m\times\IR^N \]
is compact as well. The \emph{support function} of $F$ is defined by
\begin{equation*}
  h_F\colon\IR^N\to\IR ,\qquad h_F(\vct v) := \min_{\vct x\in M} h_{F(\vct x)}(\vct v) = \min_{\vct x\in M} \, \max_{\vct z\in F(\vct x)} \langle\vct v,\vct z\rangle .
\end{equation*}
\end{definition}

The restricted singular value is related to the following construction: for compact sets $M\subset\IR^m$, $M'\subset\IR^n$ define the \emph{(convex) tensor bundle}
\begin{equation*}
  F\colon M\to \K(\IR^{mn}) ,\qquad \vct x \mapsto \{\vct x\} \otimes K',
\end{equation*}
where $K':=\conv(M')$ denotes the convex hull of~$M'$. Using~\eqref{eq:<Ax,y>=<vec(A),xy>}, we can write the support function in the form
  \[ h_F(\vec(\vct A)) = \min_{\vct x\in M}\, \max_{\vct y\in M'} \langle \vct{Ax},\vct y \rangle , \]
where $\vct A\in\IR^{n\times m}$. We denote this bundle by $M\to M\prtensor M'$.
Setting $M:=C\cap S^{m-1}$ and $M':=D\cap B^n$, we obtain for the bundle $F=M\to M\prtensor M'$, cp.~\eqref{eq:nres(A,C,D)-h_(KxK')},
\begin{equation}\label{eq:sres(A,C,D)-h_(KK')}
  \sres{A}{C}{D} = h_F(\vec(\vct A)) .
\end{equation}

\section{Moment functionals of convex bodies}\label{sec:mom-func}

In this section we consider the moments of the random variable $h_K(\vct g)$, where~$K$ is a convex body and~$\vct g$ is a standard Gaussian vector of appropriate dimension. As shown in Section~\ref{sec:conv-geom-restr-norm-sval} this includes as a special case the moments of the restricted norm $\nres{G}{C}{D}$, where~$\vct G$ is a Gaussian matrix. In Section~\ref{sec:intro-mom-func} we introduce the concept of moment functionals, whose relation to Euclidean and conic intrinsic volumes will be discussed in Section~\ref{sec:rel-intr-vols}. 
In Section~\ref{sec:contr-inequ} we present Slepian's Lemma and an extension to higher moments
as monotonicity properties of moment functionals. Section~\ref{subse:moments} and Section~\ref{sec:lin_imag} describe applications of the extended Slepian's Lemma.

\subsection{Introduction of moment functionals}\label{sec:intro-mom-func}

Recall that $\K(\IR^m)$ denotes the set of convex bodies in~$\IR^m$; additionally, we define $\K:=\bigcup_m \K(\IR^m)$. In the following let $\vct g$ denote a standard Gaussian vector of appropriate dimension.

\begin{definition}\label{def:mom-func}
Let $f\colon\IR\to\IR$ Borel measurable. The \emph{$f$-moment functional} is defined by
\begin{equation}\label{eq:def-mu_f(K)}
  \mu_f\colon\K\to \IR ,\qquad \mu_f(K) := \Expect\big[f(h_K(\vct g))\big] .
\end{equation}
\end{definition}

An important special case of a moment functional is the \emph{Gaussian width} obtained by choosing $f=\id$. We denote this special functional by
\begin{equation}\label{eq:def-w(K)}
  w(K) := \mu_{\id}(K) = \Expect\big[h_K(\vct g)\big] .
\end{equation}
Another special case is the constant function $f\equiv1$, which is in fact an emergence of the \emph{Euler characteristic} $\mu_1(K)=\chi(K)=1$.

The following proposition lists some general properties of the moment functionals.

\begin{proposition}\label{prop:props-mu_f}
Let $f\colon\IR\to\IR$ Borel measurable.
\begin{enumerate}
  \setlength{\itemsep}{3pt}
  \item $\mu_f$ is intrinsic: $\mu_f(K)=\mu_f(K\times\{\vct0\})$.
  \item $\mu_f$ is continuous: $\mu_f(K_i)\to\mu_f(K)$ if $K_i\to K$.
  \item $\mu_f$ is orthogonal invariant: $\mu_f(\vct QK)=\mu_f(K)$ if $K\in\K(\IR^m)$, $\vct Q\in O(m)$.
  \item $\mu_f$ is additive: $\mu_f(K\cup K')+\mu_f(K\cap K')=\mu_f(K)+\mu_f(K')$, if $K,K',K\cup K'\in\K$.
  \item If $f$ is monotonically increasing, $f(x)\geq f(y)$ for all $x\geq y$, then so is $\mu_f$: $\mu_f(K)\geq \mu_f(K')$ for all $K\supseteq K'$.
  \item If $f(x)=x^r$, then $\mu_f$ is not translation invariant unless $r\in\{0,1\}$.
\end{enumerate}
\end{proposition}

\begin{proof}
(1) This follows from $h_{K\times\{\vct0\}}(\vct v,\vct v')=h_K(\vct v)$.

(2) This follows from the continuity of $h_K$, see for example~\cite[Lem.~1.8.10]{Schn:book}.

(3) This follows from the orthogonal invariance of the normal distribution.

(4) Note that $h_{K\cup K'}(\vct v)=\max\{h_K(\vct v),h_{K'}(\vct v)\}$ and $h_{K\cap K'}(\vct v)\leq\min\{h_K(\vct v),h_{K'}(\vct v)\}$. The convexity of $K\cup K'$ implies that $h_{K\cap K'}(\vct v)=\min\{h_K(\vct v),h_{K'}(\vct v)\}$. Hence, we obtain
\begin{align*}
   \mu_f(K\cup K') & +\mu_f(K\cap K')-\mu_f(K)-\mu_f(K') = \Expect\big[ f(h_{K\cup K'}(\vct g)) + f(h_{K\cap K'}(\vct g)) - f(h_K(\vct g)) - f(h_{K'}(\vct g)) \big]
\\ & = \Expect\big[ f\big(\max\{h_K(\vct g),h_{K'}(\vct g)\}\big) + f\big(\min\{h_K(\vct g),h_{K'}(\vct g)\}\big) - f(h_K(\vct g)) - f(h_{K'}(\vct g)) \big] = 0 .
\end{align*}

(5) This follows from the monotonicity of~$h_K$.

(6) This is easily verified in dimension $m=1$.
\end{proof}

A functional satisfying the additivity property~(4) is called a \emph{valuation}. Note that the translation invariance distinguishes the Gaussian width and the Euler characteristic from all other functionals. Somewhat surprisingly, restricted to cone stubs, i.e., convex bodies of the form $C\cap B^m$ for some cone $C\subseteq\IR^m$, \emph{every} moment functional can be written as a linear combination of the translation invariant (Euclidean) intrinsic volumes. We explain this in the following section.

\subsection{Relation to Euclidean and conic intrinsic volumes}\label{sec:rel-intr-vols}

The Euclidean intrinsic volumes of a convex body $K\in\K(\IR^m)$ appear in the formula, named after Jakob Steiner, for the volume of its tubular neighborhood $\T_m(K,r)=\{\vct z\in\IR^m\mid \|\vct x-\vct z\|\leq r \text{ for some } \vct x\in K\}$:
\begin{equation}\label{eq:Steiner-eucl}
  \vol_m \T_m(K,r) = \sum_{i=0}^m V_i(K)\, \vol_{m-i}(r B^{m-i}),
\end{equation}
where the volume of the Euclidean $k$-ball is given by $\vol_{k}(\ball{k}) = \frac{\pi^{k/2}}{\Gamma(1+k/2)}$.
The intrinsic volumes $V_0(K),\ldots,V_m(K)$ can be defined by formula~\eqref{eq:Steiner-eucl}. Special cases include the usual ($m$-dimensional) volume $V_m(K)=\vol_m(K)$, the Euler characteristic $V_0(K)=\chi(K)=1$, and the mean width~$V_1$, which is a multiple of the Gaussian width $w(K)=w(B^m)\,V_1(K)$. Hadwiger's characterization theorem, see for example~\cite{KR:97}, states that $V_0,\ldots,V_m$ form a basis for the Euclidean motion invariant continuous valuations on~$\K(\IR^m)$.

The spherical or conic intrinsic volumes can be characterized by a conic analog of the Steiner formula~\eqref{eq:Steiner-eucl}: if $C\subseteq\IR^m$ is a closed convex cone and $\vct g\in\IR^m$ Gaussian, then
\begin{equation}\label{eq:Steiner-conic}
  \Prob\{ \|\Pi_C(\vct g)\|\geq r \} = \sum_{i=0}^m v_i(C)\, \Prob\{ \chi_i\geq r \} ,
\end{equation}
where $\chi_0=0$ and $\chi_1,\ldots,\chi_m$ are independent chi-distributed random variables with $\chi_i$ having~$i$ degrees of freedom.
A powerful generalization of the Steiner formula~\eqref{eq:Steiner-conic} was derived in~\cite{MT:13}, which we state for completeness and later reference: if $f\colon \R_+^2\to \R$ is a Borel function and $C\subseteq\IR^m$ a closed convex cone, then
\begin{equation}\label{eq:steiner-mike}
  \Expect\big[f(\norm{\Proj_C(\vct{g})},\norm{\Proj_{C^{\polar}}(\vct{g})})\big] = \sum_{i=0}^m v_i(C) \Expect\big[f(\chi_i,\chi_{m-i}')\big],
\end{equation}
where $\chi_0=\chi_0'=0$ and $\chi_1,\ldots,\chi_m,\chi_1',\ldots,\chi_m'$ are independent chi-distributed random variables with $\chi_i$ and $\chi_i'$ having~$i$ degrees of freedom.

A close relation between the Euclidean and the conic intrinsic volumes is given by cone stubs, intersections of cones with the unit ball: if $C\subseteq\IR^m$ a closed convex cone and $K=C\cap B^m$, then for $0\leq j\leq m$,
\begin{equation}\label{eq:V_i(cone-stub)=...v_j(C)}
  V_i(K) = \sum_{j=i}^m V_i(B^j)\, v_j(C) = \sum_{j=i}^m \bigg(\begin{matrix}j\\i\end{matrix}\bigg) \frac{\vol_j B^j}{\vol_{j-i} B^{j-i}}\, v_j(C) ,
\end{equation}
cf.~\cite[Prop.~4.4.18]{am:thesis}.

As seen in Proposition~\ref{prop:props-mu_f}, the moment functionals are rarely translation invariant, and can therefore in general not be expressed in terms of the Euclidean intrinsic volumes, which are translation invariant. However, for $K=C\cap B^m$ with $C\subseteq\IR^m$ a closed convex cone, one can express \emph{every} $f$-moment functional in terms of the Euclidean intrinsic 
volumes $V_0(K),\ldots,V_m(K)$ as well as in the conic intrinsic volumes $v_0(C),\ldots,v_m(C)$.

\begin{proposition}\label{prop:expr-intr-vol}
Let $f\colon\IR\to\IR$ Borel measurable. For every $K=C\cap B^m$ with $C\subseteq\IR^m$ a closed convex cone, we have
\begin{equation}\label{eq:mu_f(cone-stub)=...v_j(C)}
  \mu_f(K) = \sum_{j=0}^m v_j(C) \mu_f(B^j) = \sum_{j=0}^m V_j(K) \sum_{i=0}^j a_{ij}\, \mu_f(B^i)  ,
\end{equation}
where the matrix $\vct A=(a_{ij})\in\IR^{(m+1)\times(m+1)}$ is given by
  \[ \vct A^{-1} = \Bigg( \bigg(\begin{matrix}j\\i\end{matrix}\bigg)\, \frac{\vol B^j}{\vol B^{j-i}}\Bigg)_{ij} . \]
\end{proposition}

\begin{proof}
Using~\eqref{eq:|Pi_C(x)|=h_K(x)} we can write $\mu_f(K) = \Expect\big[ f(\|\Pi_C(\vct g)\|)\big]$.
The generalized Steiner formula~\eqref{eq:steiner-mike} implies
  \[ \mu_f(K) = \sum_{j=0}^m \Expect\big[ f(\|\Pi_{L_j}(\vct g)\|)\big]\,v_j(C) = \sum_{j=0}^m \mu_f(B^j)\,v_j(C) , \]
where $L_j\subseteq\IR^m$ denotes a $j$-dimensional linear subspace. The second equation in~\eqref{eq:mu_f(cone-stub)=...v_j(C)} follows from the relation~\eqref{eq:V_i(cone-stub)=...v_j(C)} between the conic intrinsic volumes of a cone and the Euclidean intrinsic volumes of the corresponding cone stub.
\end{proof}

In particular, we obtain from~\eqref{eq:nres(A,C,D)-h_(KxK')} for monotonically increasing~$f$,
  \[ \Expect\big[ f(\nres{G}{C}{D})\big] = \mu_f(K\prtensor K') \leq \mu_f\big((C\prtensor D)\cap B^{mn}\big)= \sum_{j=0}^{mn} v_j(C\prtensor D) \mu_f(B^j) , \]
where $C\subseteq\IR^m$ and $D\subseteq\IR^n$ closed convex cones, $K=C\cap B^m$, $K'=D\cap B^n$, and $\vct G\in\IR^{n\times m}$ a standard Gaussian matrix. 
Specializing further by setting $f=\id$, we obtain
  \[ \Expect[ \nres{G}{C}{D} ] \leq w((C\prtensor D)\cap B^{mn}) . \]
This estimate can be sharp, as seen from the case $C=\IR_+$, where we have $\Expect[ \nres{G}{C}{D} ] = w(D\cap B^n) = w((C\prtensor D)\cap B^{mn})$, or trivial, as seen from the case $C=\IR^m,D=\IR^n$, where we have $\IR^m\otimes\IR^n=\IR^{mn}$ and $\Expect[ \|\vct G\| ] \leq w(B^{mn}) = \Expect[ \|\vct G\|_F ]$, where $\|\cdot\|_F$ is the Frobenius norm. 
Furthermore, this upper bound is not very explicit. A more useful estimate is achieved by an extension of Slepian's Lemma, which we present next.

\subsection{Contraction inequalities}\label{sec:contr-inequ}

Recall from Proposition~\ref{prop:props-mu_f}(5) that for monotonically increasing $f\colon\IR\to\IR$, the functional~$\mu_f$ is monotonically increasing under inclusion. Slepian's Lemma generalizes this monotonicity by weakening the inclusion assumption.

\begin{definition}\label{def:contr-body}
For a convex body $K\in\K$ we say that $M\subseteq K$ generates~$K$ if $K=\clconv(M)$. For $K_1,K_2\in \K$ we say that $K_2$ is a \emph{contraction} of $K_1$ if there exists a $1$-Lipschitz surjection $\vp\colon M_1\to M_2$ between generating sets~$M_1,M_2$ of $K_1,K_2$. If additionally $\|\vp(\vct x)\|=\|\vct x\|$ for all $\vct x\in M_1$ then $K_2$ is a \emph{norm-preserved contraction} of $K_1$.

If $\vct0\in K_1\cap K_2$ we say that $K_2$ is a \emph{$\vct 0$-contraction} of $K_1$ if there exists a $1$-Lipschitz surjection $\vp\colon M_1\to M_2$ such that additionally $\|\vp(\vct x)\|\leq\|\vct x\|$ for all $\vct x\in M_1$.
\end{definition}

The following proposition is a convex geometric formulation of (the generalized) Slepian's Inequality.
The proof is deferred to Section~\ref{sec:mom-functls-bdls}, where a more general version will be proved in Theorem~\ref{thm:Gordon-geom}.

\begin{proposition}\label{prop:slep-ext}
Let $K_1,K_2\in\K$.
\begin{enumerate}
  \item If $K_2$ is a contraction of $K_1$, then $w(K_2)\leq w(K_1)$.
  \item If $K_2$ is a norm-preserved contraction of $K_1$ and $f\colon\IR\to\IR$ monotonically increasing, then $\mu_f(K_2)\leq \mu_f(K_1)$.
  \item If $\vct 0\in K_1\cap K_2$ and $K_2$ is a $\vct0$-contraction of~$K_1$ and $f\colon\IR_+\to\IR$ monotonically increasing and convex, then $\mu_f(K_2)\leq \mu_f(K_1)$.
\end{enumerate}
\end{proposition}

Statements~(1) and~(2) follow from known versions of Slepian's Inequality, while~(3) seems to be novel. We apply~(3) in the following two sections. In particular, we will see that the convexity assumption on~$f$ may not be dropped.

\subsection{Moments of the restricted norm of a Gaussian matrix}\label{subse:moments}

In this section we compare the moment functionals of the convex tensor product $K\prtensor K'$ with those of the direct product $K\times K'$. As a corollary we obtain the upper bounds~\eqref{eq:moments-restr-singvals-claim-norm-conv} and~\eqref{eq:moments-restr-singvals-claim-norm} in Theorem~\ref{thm:moments-restr-singvals}.

Recall from~\eqref{eq:nres(A,C,D)-h_(KxK')} that the norm restricted to cones $C\subseteq\IR^m$, $D\subseteq\IR^n$ can be expressed through the support function of the tensor product $K\prtensor K'$, where $K=C\cap B^m$ (or $K=\conv(C\cap S^{m-1})$) and $K'=D\cap B^n$, via $\nres{A}{C}{D} = h_{K\prtensor K'}(\vec(\vct A))$. As already seen in Section~\ref{sec:rel-intr-vols}, this implies that the moments of the restricted norm of a Gaussian matrix are given by the moment functional of the tensor product $K\prtensor K'$,
\begin{equation}\label{eq:E[f(||G||_(C,D))]=mu_f(KK')}
  \Expect\big[ f(\nres{G}{C}{D})\big] = \mu_f(K\prtensor K') ,
\end{equation}
where $\vct G\in\IR^{m\times n}$ is a (standard) Gaussian matrix.

To compare the tensor product with the direct product we consider the function $(\vct x,\vct y)\mapsto \vct x\otimes\vct y$.
This map is not necessarily a contraction:
\begin{align*}
  \|(\vct x_1,\vct y_1) - (\vct x_2,\vct y_2)\|^2 & = \|\vct x_1\|^2 + \|\vct y_1\|^2 + \|\vct x_2\|^2 + \|\vct y_2\|^2 - 2\big( \langle \vct x_1,\vct x_2\rangle + \langle \vct y_1,\vct y_2\rangle\big) ,
\\ \|\vct x_1\otimes\vct y_1 - \vct x_2\otimes\vct y_2\|^2 & = \|\vct x_1\|^2\,\|\vct y_1\|^2 + \|\vct x_2\|^2\,\|\vct y_2\|^2 - 2\langle \vct x_1,\vct x_2\rangle\langle \vct y_1,\vct y_2\rangle ;
\end{align*}
choosing $\vct x_2=\lambda\vct x_1$, $\vct y_2=\lambda\vct y_1$ with $\|\vct x_1\|=\|\vct y_1\|=1$ yields
  \[ \|(\vct x_1,\vct y_1) - (\vct x_2,\vct y_2)\|^2 - \|\vct x_1\otimes\vct y_2 - \vct x_2\otimes\vct y_2\|^2 = -(\lambda^2+2\lambda-1)(\lambda-1)^2 \;\stackrel{\big[\lambda=\frac{\sqrt{5}-1}{2}\big]}{=}\; \frac{11-5\sqrt{5}}{2} \approx -0.09 . \]

In the following proposition we provide sufficient assumptions on $K$ and $K'$, which imply that $K\prtensor K'$ is in fact a contraction of $K\times K'$.

\begin{proposition}\label{prop:tensprod-contract}
Let $K\subseteq B^m$ and $K'\subseteq B^n$ convex bodies, and let $M\subseteq S^{m-1}$ be closed.
\begin{enumerate}
  \item If $K=\conv(\{\vct0\}\cup M)$ and if $\vct0\in K'$, then $K\prtensor K'$ is a $\vct0$-contraction of $K\times K'$.
  \item If $K=\conv(M)$, then $(K\prtensor K')\times \{1\}$ is a norm-preserved contraction of $K\times K'$.  
\end{enumerate}
\end{proposition}

\begin{proof}
(1) The set $(\{\vct0\}\cup M)\times K'$ generates~$K\times K'$, and $(\{\vct0\}\cup M)\otimes K'$ generates~$K\prtensor K'$ by assumption. We define $\phi\colon (\{\vct0\}\cup M)\times K' \to (\{\vct0\}\cup M)\otimes K'$, $(\vct x,\vct y)\mapsto \vct x\otimes \vct y$. This map is surjective, and we need to show that it is also $1$-Lipschitz and satisfies $\|(\vct x,\vct y)\|\geq \|\vct x\otimes \vct y\|$. 
The second claim is obvious, since $\|\vct x\|\in \{0,1\}$. 
If $\|\vct x_1\|=\|\vct x_2\|=1$, then
\begin{equation}\label{eq:dist_dir/Kron-prod_|x|=1}
   \|(\vct x_1,\vct y_1) - (\vct x_2,\vct y_2)\|^2 - \|\vct x_1\otimes\vct y_1 - \vct x_2\otimes\vct y_2\|^2 = 2\,( 1 - \langle \vct x_1,\vct x_2\rangle)\,(1 - \langle \vct y_1,\vct y_2\rangle) \geq 0 .
\end{equation}
Furthermore, if $\vct x_1=\zerovct$, we have 
  \[ \|(\vct 0,\vct y_1) - (\vct x_2,\vct y_2)\|^2 - \|\vct 0\otimes\vct y_1 - \vct x_2\otimes\vct y_2\|^2 = \|\vct x_2\|^2+\|\vct y_1-\vct y_2\|^2-\|\vct x_2\|^2\|\vct y_2\|^2 \geq 0 , \]
the last inequality being a consequence of $\|\vct y_2\|\leq 1$ and $\|\vct x_2\|\in \{0,1\}$.

(2) This follows analogously.
\end{proof}

The following corollary consists of~\eqref{eq:moments-restr-singvals-claim-norm-conv} and~\eqref{eq:moments-restr-singvals-claim-norm} in Theorem~\ref{thm:moments-restr-singvals}. We recall these claims here for convenience.

\begin{corollary}
Let $C\subseteq\IR^m$, $D\subseteq\IR^n$ closed convex cones, and let $\vct G\in\IR^{n\times m}$ Gaussian matrix and $\vct g\in\IR^m$, $\vct g'\in\IR^n$ independent Gaussian vectors. Then for $f\colon\IR\to\IR$ monotonically increasing and convex,
\begin{equation}\label{eq:mom-nres-bd-1}
 \Expect\big[ f(\nres{G}{C}{D})\big] \leq \Expect\big[ f\big(\|\Proj_C(\vct g)\| + \|\Proj_D(\vct g')\|\big)\big] .
\end{equation}
If $\gamma\in\IR$ denotes a standard Gaussian variable, which is independent of~$\vct G$, then for every monotonically increasing~$f$,
\begin{equation}\label{eq:mom-nres-bd-2}
   \Expect\big[ f(\nres{G}{C}{D}+\gamma)\big] \leq \Expect\big[ f\big(\|\Proj_C(\vct g)\| + \|\Proj_D(\vct g')\|\big)\big] .
\end{equation}
\end{corollary}

\begin{proof}
We denote in this proof $K:=\conv(C\cap S^{m-1})$, $K_0 := C\cap B^m$, and $K_0':=D\cap B^n$. Note that $K_0$ is generated by $\{\vct0\}\cup (C\cap S^{m-1})$. 
Combining Proposition~\ref{prop:tensprod-contract}(1) with Proposition~\ref{prop:slep-ext} yields
\begin{equation*}
   \Expect\big[ f(\nres{G}{C}{D})\big] = \mu_f(K_0\prtensor K_0') \leq \mu_f(K_0\times K_0') \stackrel{(*)}{=} \Expect\big[ f(\|\Proj_C(\vct g)\| + \|\Proj_D(\vct g')\|)\big] ,
\end{equation*}
where $(*)$ follows from $h_{K_0\times K_0'}(\vct v,\vct v')=h_{K_0}(\vct v)+h_{K_0'}(\vct v')$. The second claim follows analogously by applying Proposition~\ref{prop:tensprod-contract}(2) and Proposition~\ref{prop:slep-ext}:
\begin{equation}\label{eq:estim-E[f(|G|_(CD))]<=...}
   \Expect\big[ f(\nres{G}{C}{D}+\gamma)\big] = \mu_f\big((K\prtensor K_0')\times\{1\}\big) \leq \mu_f(K\times K_0') = \Expect\big[ f(h_K(\vct g) + \|\Proj_D(\vct g')\|)\big] ,
\end{equation}
and $h_K(\vct g)\leq\|\Proj_C(\vct g)\|$.
\end{proof}

The above proof actually shows a slightly stronger bound: from~\eqref{eq:estim-E[f(|G|_(CD))]<=...} and from the symmetry $\nres{G}{C}{D}=\nrest{G}{C}{D}$, it follows
\begin{equation*}
  \Expect\big[ f(\nres{G}{C}{D}+\gamma)\big] \leq \min\Big\{ \Expect\big[ f\big(h_K(\vct g) + \|\Proj_D(\vct g')\|\big)\big],\Expect\big[ f\big( \|\Proj_C(\vct g)\| + h_{K'}(\vct g')\big)\big]\Big\} ,
\end{equation*}
where $K=\conv(C\cap S^{m-1})$, $K'=\conv(D\cap S^{n-1})$.

\begin{remark}\label{rem:rhs-intrvol-norm}
Using the generalized Steiner formula~\eqref{eq:steiner-mike}, the right-hand sides in~\eqref{eq:mom-nres-bd-1} and \eqref{eq:mom-nres-bd-2} can be written in terms of the intrinsic volumes of~$C$ and~$D$:
\begin{equation}\label{eq:exp-rhs-upp-bd}
  \Expect\big[ f\big(\|\Proj_C(\vct g)\| + \|\Proj_D(\vct g')\|\big)\big] = \sum_{i=0}^m \sum_{j=0}^n v_i(C) \, v_j(D) \, \Expect\big[ f\big(\chi_i + \chi_j'\big)\big] ,
\end{equation}
where $\chi_0=\chi_0'=0$ and $\chi_1,\ldots,\chi_m,\chi_1',\ldots,\chi_n'$ denote independent chi-distributed random variables with $\chi_i$ and $\chi_j'$ having~$i$ and~$j$ degrees of freedom, respectively. We will make use of this expression in Section~\ref{sec:comp-bounds}.
\end{remark}

\subsection{Linear images of cones}\label{sec:lin_imag}

In conic integral geometry the random variable $\|\Proj_C(\vct g)\|$,
where $C\subseteq\IR^m$ a closed convex cone
and $\vct g\in\IR^m$ a standard Gaussian vector, plays an important role, as indicated in Section~\ref{sec:rel-intr-vols}.
In fact, the norm of the projection is a special case of a cone-restricted norm:
\begin{equation}\label{eq:normproj-resnorm}
 \|\Proj_C(\vct g)\| = \norm{\mtx{g}}_{\R_+\to C},
\end{equation}
where on the right-hand side we interpret $\mtx{g}\in \R^{m\times 1}$ as linear map. 
Using Proposition~\ref{prop:slep-ext} we will derive estimates for the moments of $\norm{\tilde{\mtx{G}}}_{\vct TC\to \mtx{U}D}$, where $C\subseteq\IR^m$ and $D\subseteq\IR^n$ closed convex cones, $\vct T\in \IR^{\ell\times m}$ and $\vct U\in\IR^{p\times n}$, and $\tilde{\vct G}$ is a Gaussian $(p\times \ell)$-matrix.

\begin{proposition}\label{prop:moments-res.norm}
Let $C\subseteq\IR^m$ and $D\subseteq\IR^n$ closed convex cones, $\vct T\in \IR^{\ell\times m}$ and $\vct U\in\IR^{p\times n}$. Then for $r\geq 1$,
\begin{equation*}
  \Expect\big[\norm{\tilde{\mtx{G}}}^r_{\vct TC\to \mtx{U}D}\big]\leq \Ren_C(\vct T)^r \; \Ren_D(\vct U)^r \; \Expect\big[\norm{\mtx{G}}_{C\to D}^r\big] ,
\end{equation*}
where $\tilde{\vct G}\in\IR^{p\times\ell}$ and $\vct G\in\IR^{n\times m}$ Gaussian matrices.
\end{proposition}

\begin{lemma}\label{lem:TC-1}
Let $D\subseteq \IR^n$ closed convex cone and $\mtx{U}\in \R^{p\times n}$. Then
\begin{equation}
  \vct UD\cap B^p \,\subseteq\, \tfrac{1}{\lambda}\, \vct U(D\cap B^n) ,
\end{equation}
with $\lambda := \max\big\{\sres{U}{D}{\IR^p},\srestm{U}{\IR^p}{D}\big\}$.
\end{lemma}

\begin{proof}
Let $\lambda_1:=\sres{U}{D}{\IR^p}$, $\lambda_2:=\srestm{U}{\IR^p}{D}$. We will show in two steps that $\vct UD\cap B^p \,\subseteq\, \tfrac{1}{\lambda_1}\, \vct U(D\cap B^n)$ and $\vct UD\cap B^p \,\subseteq\, \tfrac{1}{\lambda_2}\, \vct U(D\cap B^n)$.

(1) Since $\vct UD\cap B^p$ as well as $\vct U(D\cap B^n)$ contain the origin, it suffices to show that $\vct UD\cap S^{p-1}\subseteq \tfrac{1}{\lambda_1}\, \vct U(D\cap B^n)$. Every element in $\vct UD\cap S^{p-1}$ can be written as $\frac{\vct{Uy}_0}{\|\vct{Uy}_0\|}$ for some $\vct y_0\in D\cap S^{n-1}$, and since $\sres{U}{D}{\IR^p}=\min_{\vct y\in D\cap S^{n-1}}\|\vct{Uy}\|\leq \|\vct{Uy}_0\|$, we obtain $\sres{U}{D}{\IR^p}\frac{\vct{Uy}_0}{\|\vct{Uy}_0\|}\in \conv\{\vct0,\vct{Uy}_0\}\subseteq\vct U(D\cap B^n)$. This shows $\vct UD\cap S^{p-1}\subseteq \tfrac{1}{\lambda_1}\, \vct U(D\cap B^n)$.

(2) Recall from~\eqref{eq:D(C,D)} that $\srestm{U}{\IR^p}{D}>0$ only if $(\vct UD)^\polar=\{\vct0\}$, i.e., $\vct UD=\IR^p$. Observe that
\begin{align*}
   \srestm{U}{\IR^p}{D} & = \min_{\vct z\in\IR^p} \max_{\vct y\in D\cap B^n} \langle \vct{Uy},\vct z\rangle = \max\big\{r\geq0\mid r B^p\subseteq \vct U(D\cap B^n)\big\} .
\end{align*}
This shows $B^p \subseteq \tfrac{1}{\lambda_2}\, \vct U(D\cap B^n)$ and thus finishes the proof.
\end{proof}

\begin{lemma}\label{lem:TC-2}
Let $K,K'$ be convex bodies such that $K=\conv(M)$ for some closed set $M\subseteq S^{m-1}$ and $\vct0\in K'$, and let $L:=\spa(K')$ the linear hull of~$K'$. If $\vct T$ denotes a linear transformation on~$L$, then $K\prtensor \mtx{T}K'$ is a $\zerovct$-contraction of $K\prtensor \norm{\mtx{T}} K'$. 
\end{lemma}

\begin{proof}
Note that the norm of the difference of two rank one matrices can be written as
\begin{align*}
   \|\vct x_1\otimes \vct y_1 - \vct x_2\otimes\vct y_2\|^2 & = \|\vct x_1\|^2\|\vct y_1\|^2 - 2\langle\vct x_1,\vct x_2\rangle \langle\vct y_1,\vct y_2\rangle + \|\vct x_2\|^2\|\vct y_2\|^2 .
\end{align*}
So for $\|\vct x_1\|=\|\vct x_2\|=1$, 
\begin{align*}
   & \|\vct x_1\otimes \vct y_1 - \vct x_2\otimes\vct y_2\|^2 - \|\vct x_1\otimes \vct z_1 - \vct x_2\otimes\vct z_2\|^2
\\ & = \|\vct y_1\|^2 + \|\vct y_2\|^2 - \|\vct z_1\|^2 - \|\vct z_2\|^2 + 2\langle\vct x_1,\vct x_2\rangle \big(\langle\vct z_1,\vct z_2\rangle-\langle\vct y_1,\vct y_2\rangle\big)
\\ & \geq \begin{cases}
             \|\vct y_1\|^2 + \|\vct y_2\|^2 - 2\langle\vct y_1,\vct y_2\rangle - \|\vct z_1\|^2 - \|\vct z_2\|^2 + 2\langle\vct z_1,\vct z_2\rangle
          \\ = \|\vct y_1-\vct y_2\|^2 - \|\vct z_1-\vct z_2\|^2 & \text{if } \langle\vct z_1,\vct z_2\rangle\leq \langle\vct y_1,\vct y_2\rangle
          \\[2mm] \|\vct y_1\|^2 + \|\vct y_2\|^2 + 2\langle\vct y_1,\vct y_2\rangle - \|\vct z_1\|^2 - \|\vct z_2\|^2 - 2\langle\vct z_1,\vct z_2\rangle
          \\ = \|\vct y_1+\vct y_2\|^2 - \|\vct z_1+\vct z_2\|^2 & \text{if } \langle\vct z_1,\vct z_2\rangle\geq \langle\vct y_1,\vct y_2\rangle .
          \end{cases}
\end{align*}
Setting $\phi\colon M\otimes \|\vct T\|K'\to M\otimes \vct TK'$, $\phi(\vct x\otimes \|\vct T\|\vct x'):=\vct x\otimes \vct{Tx}'$, we have $\|\phi(\vct x\otimes \|\vct T\|\vct x')\|=\|\vct{Tx}'\|\leq\|\vct T\|\,\|\vct x'\|=\big\|\vct x\otimes \|\vct T\|\vct x'\big\|$, and from the above computation, with $\vct y_i=\|\vct T\|\vct x'_i$ and $\vct z_i=\vct{Tx}'_i$, $i=1,2$, we obtain either
  \[ \big\|\vct x_1\otimes \|\vct T\|\vct x'_1 - \vct x_2\otimes\|\vct T\|\vct x'_2\big\|^2 - \big\|\vct x_1\otimes \vct{Tx}'_1 - \vct x_2\otimes \vct{Tx}'_2\big\|^2 \geq \|\vct T\|^2\,\|\vct x_1'-\vct x_2'\|^2 - \|\vct T(\vct x_1'-\vct x_2')\|^2 \geq 0 , \]
or
  \[ \big\|\vct x_1\otimes \|\vct T\|\vct x'_1 - \vct x_2\otimes\|\vct T\|\vct x'_2\big\|^2 - \big\|\vct x_1\otimes \vct{Tx}'_1 - \vct x_2\otimes \vct{Tx}'_2\big\|^2 \geq \|\vct T\|^2\,\|\vct x_1'+\vct x_2'\|^2 - \|\vct T(\vct x_1'+\vct x_2')\|^2 \geq 0 . \]
This shows that $K\prtensor \mtx{T}K'$ is a $\zerovct$-contraction of $K\prtensor \norm{\mtx{T}} K'$.
\end{proof}

\begin{proof}[Proof of Proposition~\ref{prop:moments-res.norm}]
Assume first that $\ell=m$ and $\vct T=\Id_m$. As in Lemma~\ref{lem:TC-1}, let $\lambda := \max\big\{\sres{U}{D}{\IR^p},\srestm{U}{\IR^p}{D}\big\}$, so that
\begin{align*}
  \Expect\big[\norm{\mtx{G}}_{C\to\mtx{U}D}^r\big] & = \Expect\big[\big(\max_{\vct{x}\in C\cap S^{m-1}}\max_{\vct{y}\in \mtx{U}D\cap B^p}\ip{\mtx{G}\vct{x}}{\vct{y}}\big)^r\big]
\\ & \leq \lambda^{-r}\Expect\big[\big(\max_{\vct{x}\in C\cap B^m}\max_{\vct{y}\in \mtx{U}(D\cap B^n)}\ip{\mtx{G}\vct{x}}{\vct{y}}\big)^r\big] = \lambda^{-r}\mu_f(K\prtensor \mtx{U}K') ,
\end{align*}
where $f(t):=t^r$, $K:=C\cap B^m$, and $K':=D\cap B^n$. From Lemma~\ref{lem:TC-2} and Slepian's Inequality~(3) in Proposition~\ref{prop:slep-ext} we obtain $\mu_f(K\prtensor \mtx{U}K')\leq \norm{\mtx{U}}^{r} \mu_f(K\prtensor K') = \norm{\mtx{U}}^{r}\Expect\big[\norm{\mtx{G}}_{C\to D}^r\big]$, so that
  \[ \Expect\big[\norm{\mtx{G}}_{C\to\mtx{U}D}^r\big] \leq \Bigg(\frac{\norm{\mtx{U}}}{\max\big\{\sres{U}{D}{\IR^p},\srestm{U}{\IR^p}{D}\big\}}\Bigg)^{r} \Expect\big[\norm{\mtx{G}}_{C\to D}^r\big] = \Ren_D(\vct U)^r \; \Expect\big[\norm{\mtx{G}}_{C\to D}^r\big] . \]
This shows the claim for $\ell=m$ and $\vct T=\Id_m$. For the general case we use the symmetry of the restricted norm,
\begin{align*}
   \Expect\big[\norm{\mtx{G}}^r_{\vct TC\to \mtx{U}D}\big] & \leq \Ren_D(\vct U)^r \; \Expect\big[\norm{\mtx{G}}_{\vct TC\to D}^r\big] = \Ren_D(\vct U)^r \; \Expect\big[\norm{-\mtx{G}^T}_{D\to\vct TC}^r\big]
\\ & \leq \Ren_C(\vct T)^r \; \Ren_D(\vct U)^r \; \Expect\big[\norm{-\mtx{G}^T}_{D\to C}^r\big] = \Ren_C(\vct T)^r \; \Ren_D(\vct U)^r \; \Expect\big[\norm{\mtx{G}}_{C\to D}^r\big] . \qedhere
\end{align*}
\end{proof}

From the simple observation~\eqref{eq:normproj-resnorm} we obtain the following immediate corollary.

\begin{corollary}
Let $C\subseteq\IR^m$ closed convex cone, and
\begin{equation*}
 \nu_r(C) := \Expect\big[\norm{\Proj_C(\vct{g})}^r\big] ,
\end{equation*}
where $\vct{g}\in\IR^m$ Gaussian. Then for $\vct T\in\IR^{\ell\times m}$, and $r\geq1$,
\begin{equation}\label{eq:mu_r(TC)<=...}
  \nu_r(\mtx{T}C) \leq \Ren_C(\mtx{T})^r\nu_r(C).
\end{equation}
In particular, if $\ell=m$ then
\begin{align}\label{eq:mu_r(TC)<=...-kappa}
  \frac{\nu_r(C)}{\kappa(\vct T)^r} & \leq \nu_r(\mtx{T}C) \leq \kappa(\mtx{T})^r\nu_r(C) & \Bigg( r=2:\quad \frac{\sdim(C)}{\kappa(\vct T)^2} & \leq \sdim(\vct TC) \leq \kappa(\vct T)^2\, \sdim(C)\Bigg) .
\end{align}
\end{corollary}

Here,~\eqref{eq:mu_r(TC)<=...-kappa} follows from the inequality $\Ren_C(\mtx{T})\leq \kappa(\mtx{T})$, cf.~\eqref{eq:R_C(A)<=kappa(A)}, and by considering $C=\vct T^{-1}\vct TC$ and using $\kappa(\vct T)=\kappa(\vct T^{-1})$ to obtain the lower bound.

\begin{example}[circular cones]
Let $\Circ_m(\alpha)=\{\vct x\in\IR^m\mid x_1\geq\|\vct x\|\cos\alpha\}$ denote the circular cone of radius~$\alpha$ around the first coordinate vector. For our purposes it is more convenient to use $\tan\alpha$ instead of~$\alpha$, so we define $C_m(t):=\Circ_m(\arctan(t))$. Consider the linear map $\vct T:=\diag(1,s,\ldots,s)$ with $s\geq1$, whose condition number is $\kappa(\vct T)=s$. Then $\vct TC_m(t)=C_m(s t)$, and by~\eqref{eq:mu_r(TC)<=...-kappa} we have
\begin{equation}\label{eq:quot-kappa^rmu_r(C_d(t))/...}
  \frac{s^r\nu_r(C_m(t))}{\nu_r(C_m(s t))}\geq 1
\end{equation}
for $r\geq1$. Using Proposition~\ref{prop:expr-intr-vol} we can express~$\nu_r(C)$ in terms of the intrinsic volumes of~$C$: for $r>0$
  \[ \nu_r(C) = \sum_{j=1}^m v_j(C)\Expect[\|\vct g_j\|^r] = \sum_{j=1}^m v_j(C) \,\frac{2^{r/2}\,\Gamma(\frac{j+r}{2})}{\Gamma(\frac{j}{2})} , \]
where $\vct g_j\in\IR^j$ denotes a standard Gaussian vector. The intrinsic volumes of the circular cones are given by, cf.~\cite[Ex.~4.4.8]{am:thesis}
\begin{align*}
   v_j(C_m(t)) & = \frac{\Gamma(\frac{m}{2})\,t^j}{2\,\Gamma(\frac{j+1}{2})\Gamma(\frac{m-j+1}{2})\,(1+t^2)^{(m-2)/2}}, \quad \text{for } j=1,\ldots,m-1,
\\ v_m(C_m(t)) & = \frac{\Gamma(\frac{m}{2})}{\sqrt{\pi}\,\Gamma(\frac{m-1}{2})} \int_0^t \frac{\tau^{m-2}}{(1+\tau^2)^{n/2}}\,d\tau .
\end{align*}
Using these formulas we can compute $\nu_r(D_n(t))$. 

\begin{figure}[ht]
   \begin{center}
   \def\myXsc{4}
   \def\myYsc{1.1}
   \def\myEpsy{0.02}
   \def\myEpsx{\myEpsy*\myXsc/\myYsc}
   \subfloat[$m=50$]{\begin{tikzpicture}[xscale=\myXsc, yscale=\myYsc]
      \draw (0,0) -- (0,4) (0,0) -- (1,0);
      \foreach \y in {0,1,2,3,4}
        \draw (0,\y) -- ++(-\myEpsy,0) node[left=-1mm]{$\scriptscriptstyle\y$};
      \foreach \x in {0,0.2,0.4,0.6,0.8,1}
        \draw (\x,0) -- ++(0,-\myEpsx) node[below=-1mm]{$\scriptscriptstyle\x$};
      \draw[color=blue] plot file {tables/quot_kappa=2_r=0.5_d=50.table};
      \draw[color=blue] plot file {tables/quot_kappa=2_r=1_d=50.table};
      \draw[color=blue] plot file {tables/quot_kappa=2_r=2_d=50.table};
      \draw[color=red] (0,1) -- (1,1);
      \path (0.5,0) node[below=2pt]{$\scriptstyle t$}
            (1,2.4) node[above left]{$\scriptscriptstyle r=2$}
            (1,1.5) node[above left]{$\scriptscriptstyle r=1$}
            (1,1.17) node[above left]{$\scriptscriptstyle r=0.5$};
   \end{tikzpicture}}
   \hfill
   \subfloat[$m=100$]{\begin{tikzpicture}[xscale=\myXsc, yscale=\myYsc]
      \draw (0,0) -- (0,4) (0,0) -- (1,0);
      \foreach \y in {0,1,2,3,4}
        \draw (0,\y) -- ++(-\myEpsy,0) node[left=-1mm]{$\scriptscriptstyle\y$};
      \foreach \x in {0,0.2,0.4,0.6,0.8,1}
        \draw (\x,0) -- ++(0,-\myEpsx) node[below=-1mm]{$\scriptscriptstyle\x$};
      \draw[color=blue] plot file {tables/quot_kappa=2_r=0.5_d=100.table};
      \draw[color=blue] plot file {tables/quot_kappa=2_r=1_d=100.table};
      \draw[color=blue] plot file {tables/quot_kappa=2_r=2_d=100.table};
      \draw[color=red] (0,1) -- (1,1);
      \path (0.5,0) node[below=2pt]{$\scriptstyle t$}
            (1,2.4) node[above left]{$\scriptscriptstyle r=2$}
            (1,1.5) node[above left]{$\scriptscriptstyle r=1$}
            (1,1.17) node[above left]{$\scriptscriptstyle r=0.5$};
   \end{tikzpicture}}
   \hfill
   \subfloat[$m=200$]{\begin{tikzpicture}[xscale=\myXsc, yscale=\myYsc]
      \draw (0,0) -- (0,4) (0,0) -- (1,0);
      \foreach \y in {0,1,2,3,4}
        \draw (0,\y) -- ++(-\myEpsy,0) node[left=-1mm]{$\scriptscriptstyle\y$};
      \foreach \x in {0,0.2,0.4,0.6,0.8,1}
        \draw (\x,0) -- ++(0,-\myEpsx) node[below=-1mm]{$\scriptscriptstyle\x$};
      \draw[color=blue] plot file {tables/quot_kappa=2_r=0.5_d=200.table};
      \draw[color=blue] plot file {tables/quot_kappa=2_r=1_d=200.table};
      \draw[color=blue] plot file {tables/quot_kappa=2_r=2_d=200.table};
      \draw[color=red] (0,1) -- (1,1);
      \path (0.5,0) node[below=2pt]{$\scriptstyle t$}
            (1,2.4) node[above left]{$\scriptscriptstyle r=2$}
            (1,1.5) node[above left]{$\scriptscriptstyle r=1$}
            (1,1.17) node[above left]{$\scriptscriptstyle r=0.5$};
   \end{tikzpicture}}
   \end{center}
   \caption{Plot of the quotient $s^r\nu_r(C_m(t))/\nu_r(C_m(s t))$, cf.~\eqref{eq:quot-kappa^rmu_r(C_d(t))/...}, with $s=2$, $m\in\{50,100,200\}$, $r\in\{0.5,1,2\}$.}
   \label{fig:quot-circ}
\end{figure}
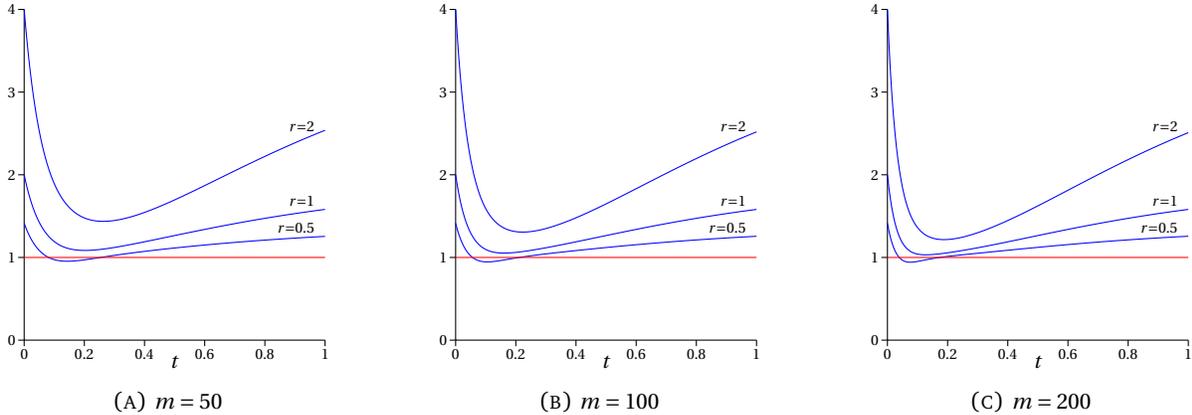

Figure~\ref{fig:quot-circ} shows a plot of the quotient in~\eqref{eq:quot-kappa^rmu_r(C_d(t))/...} for $s=2$, $m\in\{50,100,200\}$, and $r\in\{0.5,1,2\}$. The plot shows that the inequality~\eqref{eq:quot-kappa^rmu_r(C_d(t))/...} may be violated if $r<1$, which ultimately shows that the convexity assumptions in Proposition~\ref{prop:slep-ext} and Theorem~\ref{prop:gordon-with-0} may not be dropped. The plots also indicate that the inequality~\eqref{eq:quot-kappa^rmu_r(C_d(t))/...} is asymptotically sharp for $m\to\infty$. This could be shown with an analysis similar to the one given in~\cite[Sec.~6.3]{MT:13}; we leave the details to the interested reader.
\end{example}

\section{Moment functionals of convex bundles}\label{sec:mom-functls-bdls}

In this section we generalize the moment functionals to the setting of convex bundles, which we introduced in Section~\ref{sec:conv-geom-restr-norm-sval}. 
This setting allows one to study the smallest singular value $\sres{G}{C}{D}$ of a Gaussian matrix $\mtx{G}$ restricted to convex cones.
Following a similar structure as that of Section~\ref{sec:mom-func}, we begin in Section~\ref{sec:intro-mom-func-bundles} by introducing the moment functionals of bundles and describe some typical examples. 
Section~\ref{sec:contr-ineq-bdls} then introduces the notion of bundle contractions and describes the Gordon inequalities as monotonicity properties of moment functionals of convex bundles. 
As an application we get bounds of the moments of restricted singular values of a Gaussian matrix in terms of conic intrinsic volumes.

Throughout this section we will repeatedly refer to sets $M,M',K,K',C,D$.
Unless otherwise stated, we will always assume $M\subset \R^m$, $M'\subset \R^n$ to be compact sets, $K:=\conv(M)$, $K':=\conv(M')$, and $C\subseteq \R^m$, $D\subseteq \R^n$
to be closed convex cones. 

\subsection{Introduction of moment functionals.}\label{sec:intro-mom-func-bundles}
Recall that the support function of a convex bundle $F\colon M\to \K(\IR^N)$ is given by
  \[ h_F\colon\IR^N\to\IR ,\qquad h_F(\vct v) = \min_{\vct x\in M} \, \max_{\vct z\in F(\vct x)} \langle\vct v,\vct z\rangle . \]
We extend the moment functionals from Definition~\ref{def:mom-func} to convex bundles via
  \[ \mu_f(F) := \Expect\big[f(h_F(\vct g)) \big] = \Expect\Big[ f\Big(\min_{\vct x\in M} \, \max_{\vct z\in F(\vct x)} \langle\vct g,\vct z\rangle\Big)\Big] , \]
where $f\colon\IR\to\IR$ Borel measurable and $\vct g\in\IR^N$ Gaussian. Again, we denote the extension of the Gaussian width by
  \[ w(F) := \mu_{\id}(F) = \Expect\Big[ \min_{\vct x\in M} \, \max_{\vct z\in F(\vct x)} \langle\vct g,\vct z\rangle\Big] . \]

\begin{example}[Product bundle]\label{ex:prod-bundle}
The product bundle of~$K'$ over~$M$, denoted $M\to M\times K'$, is the map
$F\colon M\to \K(\IR^{m+n})$, $\vct x \mapsto \{\vct x\} \times K'$.
For $(\vct v,\vct v')\in\IR^{m+n}$ we have
\begin{align*}
   h_F(\vct v,\vct v') & = \min_{\vct x\in M} \, \max_{(\vct x,\vct x')\in M\times K'} \langle(\vct v,\vct v'),(\vct x,\vct x')\rangle = \min_{\vct x\in M} \langle\vct v,\vct x\rangle + \max_{\vct x'\in K'} \langle\vct v',\vct x'\rangle = \max_{\vct x'\in K'} \langle\vct v',\vct x'\rangle - \max_{\vct x\in K} \langle-\vct v,\vct x\rangle
\\ & = h_{K'}(\vct v') - h_K(-\vct v) .
\end{align*}
The moment functionals of this bundle are given by
\begin{equation}\label{eq:momfunc-prodbdl}
  \mu_f(M\to M\times K') = \Expect\big[ f\big(h_{K'}(\vct g') - h_K(\vct g)\big)\big] ,
\end{equation}
in particular, $w(M\to M\times K') = w(K') - w(K)$. 
\end{example}

\begin{example}[Tensor bundle]\label{ex:tensorbundle}
Recall from Section~\ref{sec:conv-geom-restr-norm-sval} that the (convex) tensor bundle of~$K'$ over~$M$, denoted $M\to M\prtensor K'$, is the map
$F\colon M\to \K(\IR^m\otimes\IR^n)$, $\vct x \mapsto \{\vct x\} \otimes K'$.
We can write the support function in the form
  \[ h_F(\vec(\vct A)) = \min_{\vct x\in M}\, \max_{\vct y\in K'} \, \langle \vct{Ax},\vct y \rangle , \]
where $\vct A\in\IR^{n\times m}$.
In the special case where $M=C\cap S^{m-1}$ and $K'=D\cap B^{n}$
we obtain $h_F(\vec(\vct A))=\sres{A}{C}{D}$. For the corresponding moment functionals,
  \[ \mu_f(M\to M\prtensor K') = \Expect\big[ f\big(\sres{G}{C}{D}\big)\big] , \]
where $\vct G\in\IR^{n\times m}$ Gaussian.
\end{example}

\begin{example}[Affine tensor bundle]\label{ex:exttensorbundle}
A variation of the above tensor bundle is the affine tensor bundle of~$K'$ over~$M$, denoted $M\tdash M\prtensor K'$ and defined by means of
$F\colon M\to \K\big((\IR^m\otimes\IR^n)\times\IR\big)$, $\vct x \mapsto (\{\vct x\} \otimes K')\times \{1\}$.
That is, the tensor bundle is embedded in an affine space at height one.
We can write the support function in the form
  \[ h_F(\vec(\vct A),\lambda) = \min_{\vct x\in M}\, \max_{\vct y\in K'} \, \langle \vct{Ax},\vct y \rangle + \lambda , \]
where $\vct A\in\IR^{n\times m}$ and $\lambda\in \R$.
If $M=C\cap S^{m-1}$, $K'=D\cap B^{n}$ then we obtain $h_F(\vec(\vct A),\lambda)=\sres{A}{C}{D}+\lambda$. In particular,
\begin{equation}\label{eq:momfunc-atensbdl}
  \mu_f(M\tdash M\prtensor K') = \Expect\big[ f(\sres{G}{C}{D}+\gamma)\big] ,
\end{equation}
where $\vct G\in\IR^{n\times m}$ a Gaussian matrix and $\gamma$ an independent Gaussian random variable.
\end{example}

\subsection{Contraction inequalities}\label{sec:contr-ineq-bdls}

We next turn to the problem of comparing two bundles over the same base set. The goal of this section is elaborate on the extent to which Proposition~\ref{prop:slep-ext} generalizes to convex bundles.

Let $F\colon M\to\K(\IR^N)$ be a convex bundle over the compact base set $M\subset\IR^m$ and let $G$ be a set-valued map on~$M$ with $G(\vct x)\subset\IR^N$.
We say that $G$ \emph{generates} the convex bundle~$F$, if the fibers of~$F$ are generated by the fibers of~$G$, i.e., $F(\vct x)=\clconv(G(\vct x))$ for all $\vct x\in M$.

The following definition extends the notion of contraction, cf.~Definition~\ref{def:contr-body}, to the setting of convex bundles. Broadly speaking, a bundle contraction is a contraction within the fibers and an expansion across the fibers.

\begin{definition}\label{def:contr-bdls}
Let $F_1,F_2\colon M\to \K(\IR^N)$ be convex bundles over the compact base set~$M\subset\IR^m$. We say that $F_2$ is a \emph{contraction} of $F_1$ if there exist generators $G_1,G_2$ of $F_1,F_2$, respectively, and surjective maps $\vp_{\vct x}\colon G_1(\vct x)\to G_2(\vct x)$,  $\vct x\in M$,
such that
\begin{align}
  \|\vp_{\vct x}(\vct y)-\vp_{\vct x}(\vct y')\| &\leq \|\vct y - \vct y'\| \quad \text{for all $\vct y,\vct y'\in G_1(\vct x)$},\label{eq:lipschitz-prop-bdl}\\
  \|\vp_{\vct x}(\vct y)-\vp_{\vct x'}(\vct y')\| &\geq \|\vct y - \vct y'\| \quad \text{for all $\vct y\in G_1(\vct x),\vct y'\in G_1(\vct x'),\vct x\neq \vct x'$}\label{eq:expans-prop-bdl} .
\end{align}
If additionally $\|\vp_{\vct x}(\vct y)\| = \|\vct y\|$ for all $\vct x\in M$ and for all $\vct y\in G_1(\vct x)$, then we say that $F_2$ is a \emph{norm-preserved contraction} of $F_1$.

If $\vct0\in F_1(\vct x)\cap F_2(\vct x)$ for all $\vct x\in M$, we say that $F_2$ is a \emph{$\vct0$-contraction} of $F_1$ if there exist surjective maps $\vp_{\vct x}\colon G_1(\vct x)\to G_2(\vct x)$, which satisfy~\eqref{eq:lipschitz-prop-bdl} and~\eqref{eq:expans-prop-bdl}, and which additionally satisfy $\|\vp_{\vct x}(\vct y)\|\leq\|\vct y\|$ for all $\vct x\in M$ and all $\vct y\in G_1(\vct x)$.
\end{definition}

The following theorem is based on Gordon's Theorem~\ref{thm:gordon} and the extension in Theorem~\ref{prop:gordon-with-0}.

\begin{theorem}\label{thm:Gordon-geom}
Let $F_1,F_2\colon M\to\in\K(\IR^N)$ be convex bundles over the compact base set $M\subset\IR^m$, and let $f\colon\IR\to\IR$ .
\begin{enumerate}
  \item If $F_2$ is a contraction of $F_1$, then $w(F_2)\leq w(F_1)$.
  \item If $F_2$ is a norm-preserved contraction of $F_1$ and~$f$ monotonically increasing, then $\mu_f(F_2)\leq \mu_f(F_1)$.
  \item If $F_2$ is a $\vct0$-contraction of~$F_1$ and~$f$ monotonically increasing and convex, then $\mu_f(F_2)\leq \mu_f(F_1)$.
\end{enumerate}
\end{theorem}

Before indulging in the proof, we mention that a typical example of a monotonically increasing
function in~(2) would be the step function $f(t)=1_{t>\lambda}$, leading to inequalities between the distributions of the support functions $h_{F_2}(\vct g)$ and $h_{F_1}(\vct g)$.
Typical examples for functions appearing
in the context of claim~(3) are powers $t^r \ (r\geq 1)$ and $e^{\lambda t}$, leading to moment inequalities.

Note that Proposition~\ref{prop:slep-ext} is the special case $|M|=1$ of Theorem~\ref{thm:Gordon-geom}, i.e., the bundles have only one fiber.

\begin{proof}
We first show claim~(3). Let $G_1,G_2$ be generators of $F_1,F_2$, respectively, such that the properties of a $\vct0$-contraction in Definition~\ref{def:contr-bdls} are satisfied with a corresponding set of 
maps $\{\vp_{\vct x}\mid \vct x\in M\}$ satisfying~\eqref{eq:lipschitz-prop-bdl} and~\eqref{eq:expans-prop-bdl} and the additional property $\|\vp_{\vct x}(\vct y)\|\leq\|\vct y\|$ for all $\vct x\in M$ and all $\vct y\in G_1(\vct x)$.

By a standard continuity argument we may assume that~$M$ as well as all fibers $G_1(\vct x),G_2(\vct x)$, $\vct x\in M$, are finite sets. Furthermore, by embedding the sets in a high-dimensional space, we may assume without loss of generality that $M$ has at most $m$ elements and each fiber $G_1(\vct x),G_2(\vct x)$ has at most $N$ elements. Concretely, let $M=\{\vct x_1,\ldots,\vct x_m\}$, where we allow repetitions among the~$\vct x_i$, and let
\begin{align*}
   G_1(\vct x_i) & = \{\vct x_{ij}\mid 1\leq j\leq N\} , & G_2(\vct x_i) & = \{\vct y_{ij}\mid 1\leq j\leq N\} ,
\end{align*}
allowing repetitions on the $\vct x_{ij}$ and $\vct y_{ij}$ as well.
Define centered Gaussian random variables $X_{ij},Y_{ij}$, $1\leq i\leq m$, $1\leq j\leq N$, via
\begin{align*}
   X_{ij} & := \langle \vct x_{ij},\vct g\rangle , & Y_{ij} & := \langle \vct y_{ij},\vct g\rangle ,
\end{align*}
where $\vct g\in\IR^N$ is a standard Gaussian vector. The properties of the maps $\vp_{\vct x}$ imply that
\begin{align*}
   \Expect |X_{ij}-X_{k\ell}|^2 & \leq \Expect |Y_{ij}-Y_{k\ell}|^2 , & & \hspace{-2cm} \text{for all } i\neq k \text{ and } j,\ell ,
\\ \Expect |X_{ij}-X_{i\ell}|^2 & \geq \Expect |Y_{ij}-Y_{i\ell}|^2 , & & \hspace{-2cm} \text{for all } i,j,\ell ,
\\ \Expect |X_{ij}|^2 & \geq \Expect |Y_{ij}|^2 , & & \hspace{-2cm} \text{for all } i,j .
\end{align*}
Applying Theorem~\ref{prop:gordon-with-0} in the degenerate case $X_0:=Y_0:=0$ yields
  \[ \mu_f(F_2) = \Expect \min_i\max_j f_+(Y_{ij}) \leq \Expect \min_i\max_j f_+(X_{ij}) = \mu_f(F_1) .  \]

Claims~(1) and~(2) follow from Gordon's Theorem~\ref{thm:gordon} by analogous arguments as above.
\end{proof}

\subsection{Moments of the restricted singular value of a Gaussian matrix}\label{sec:sres}

In this section we apply the contraction inequalities in Theorem~\ref{thm:Gordon-geom} to the
product bundle and the affine tensor bundle, described in Section~\ref{sec:intro-mom-func-bundles}.
In particular, we get inequalities involving the restricted singular values and differences of support functionals.

The following proposition is based on the same arguments as Proposition~\ref{prop:tensprod-contract}.

\begin{proposition}\label{prop:tensprod-contract-bdl}
Let $K'\subseteq B^n$ convex body, and let $M\subseteq S^{m-1}$ closed and $K:=\conv(M)$. Then the product bundle $M\to M\times K'$ is a norm-preserved contraction of the affine tensor bundle $M\tdash M\prtensor K'$.
\end{proposition}

\begin{proof}
Define $\varphi_{\vct{x}}\colon \{1\}\times (\{\vct{x}\}\otimes K')\to \{\vct{x}\}\times K'$ via $\varphi_{\vct{x}}(1,\vct{x}\otimes\vct{y}) = (\vct{x},\vct{y})$. 
 As in the proof of Proposition~\ref{prop:tensprod-contract} one concludes that
 \begin{equation*}
  \norm{(1,\vct{x}_1\otimes \vct{y}_1)-(1,\vct{x}_2\otimes \vct{y}_2)}^2 - \norm{(\vct{x}_1,\vct{y}_1)-(\vct{x}_2,\vct{y}_2)}^2 = -2(1-\ip{\vct{x}_1}{\vct{x}_2})(1-\ip{\vct{y}_1}{\vct{y}_2}),
 \end{equation*}
which is zero if $\vct{x}_1=\vct{x}_2$, and nonpositive if $\vct{x}_1\neq \vct{x}_2$. Finally,
\begin{equation*}
 \norm{(\vct{x},\vct{y})}^2 = 1+\norm{\vct{y}}^2 = 1+\norm{\vct{x}\otimes \vct{y}}^2 = \norm{\varphi(\vct{x},\vct{y})}^2. \qedhere
\end{equation*}
\end{proof}

As a corollary we obtain the final claim~\eqref{eq:moments-restr-singvals-claim-sval} in Theorem~\ref{thm:moments-restr-singvals}, which we recall here.

\begin{corollary}
Let $C\subseteq\IR^m$, $D\subseteq\IR^n$ closed convex cones, and let $\vct G\in\IR^{n\times m}$, $\vct g\in\IR^m$, $\vct g'\in\IR^n$, and $\gamma\in\IR$ all independent Gaussian. If $f\colon\IR\to\IR$ monotonically increasing, then
\begin{equation}\label{eq:mom-sval-bd}
   \Expect\big[ f(\sres{G}{C}{D}+\gamma)\big] \geq \Expect\big[ f\big(\|\Proj_D(\vct g')\| - \|\Proj_C(\vct g)\| \big)\big] .
\end{equation}
\end{corollary}

\begin{proof}
Let $M:= C\cap S^{m-1}$ and $K':=D\cap B^n$. Then Proposition~\ref{prop:tensprod-contract-bdl} implies that the product bundle $M\to M\times K'$ is a norm-preserved contraction of the affine tensor bundle $M\tdash M\prtensor K'$, and Theorem~\ref{thm:Gordon-geom} implies
  \[ \Expect\big[f\big(\norm{\Proj_D(\vct{g}')}-h_{K}(\vct{g})\big)\big] \stackrel{\eqref{eq:momfunc-prodbdl}}{=} \mu_f(M\to M\times K')\leq \mu_f(M\tdash M\prtensor K') \stackrel{\eqref{eq:momfunc-atensbdl}}{=} \Expect\big[ f(\sres{G}{C}{D}+\gamma)\big] . \]
The claim follows from $h_K(\vct x) \leq \|\Pi_C(\vct x)\|$, cf.~\eqref{eq:|Pi_C(x)|>=h_K(x)} (see also Remark~\ref{rem:normprojC-h_K}).
\end{proof}

\begin{remark}\label{rem:rhs-intrvol-sval}
Analogous to Remark~\ref{rem:rhs-intrvol-norm} we note that the right-hand side in~\eqref{eq:mom-sval-bd} can be written in terms of the intrinsic volumes of~$C$ and~$D$:
\begin{equation}\label{eq:exp-rhs-upp-bd-sval}
  \Expect\big[ f\big(\|\Proj_D(\vct g')\| - \|\Proj_C(\vct g)\|\big)\big] = \sum_{i=0}^m \sum_{j=0}^n v_i(C) \, v_j(D) \, \Expect\big[ f\big(\chi_j' - \chi_i\big)\big] ,
\end{equation}
where $\chi_0=\chi_0'=0$ and $\chi_1,\ldots,\chi_m,\chi_1',\ldots,\chi_n'$ denote independent chi-distributed random variables with $\chi_i$ and $\chi_j'$ having~$i$ and~$j$ degrees of freedom, respectively. Section~\ref{sec:comp-bounds} discusses applications of this expression.
\end{remark}

\section{Comparing bounds}\label{sec:comp-bounds}

In this section we present the bounds for the distributions of~$\nres{G}{C}{D}$ and~$\sres{G}{C}{D}$, which arise from Gordon's comparison theorem and which have found applications in the analysis of convex relaxation methods, cf.~Section~\ref{sec:intro}. These bounds are well known, but using the theory of intrinsic volumes we will present new aspects to elucidate certain interesting relations and to sharpen some bounds. Finally, we will present a short description of bounds stemming from tube formulas, and arrange these into the context of restricted singular values.

In this section we generally assume that $C\subseteq\IR^m$, $D\subseteq\IR^n$ closed convex cones, and $\vct G\in\IR^{n\times m}$, $\vct g\in\IR^m$, $\vct g'\in\IR^n$ independent Gaussians.

\subsection{Gaussian width, intrinsic volumes, statistical dimension}\label{sec:gwidth-edim-sdim}

Before getting to the discussion of the bounds, we clarify the relations between the Gaussian width, the intrinsic volumes, and the statistical dimension of a cone. The intrinsic volumes $v_0(C),\ldots,v_m(C)$, introduced in Section~\ref{sec:rel-intr-vols} through the conic Steiner formula~\eqref{eq:Steiner-conic}, form a discrete probability distribution on~$\{0,\ldots,m\}$. The \emph{statistical dimension} is defined as the mean of this distribution, which happens to coincide with the expected squared norm of the projection of a Gaussian vector onto the cone, sometimes also called ``squared Gaussian complexity''~\cite{chan:12},
\begin{equation}\label{eq:def-sdim}
  \sdim(C) := \sum_{k=0}^m k\, v_k(C) = \Expect\big[ \norm{\Proj_C(\vct g)}^2 \big] .
\end{equation}
The importance of the statistical dimension stems from the fact that the intrinsic volumes of~$C$ concentrate around~$\sdim(C)$, so that this quantity can serve as a summary parameter for the whole distribution. This concentration property has been shown in~\cite{edge,MT:13} alongside with some applications, cf.~also~\cite{MT:13b}.

Besides capturing the moments of the norm of a projected Gaussian vector, as seen through the generalized Steiner formula~\eqref{eq:steiner-mike}, the intrinsic volumes also yield exact formulas, so-called \emph{kinematic formulas}, for certain probabilities like the probability for nontrivial intersection of a cone with a uniformly random linear subspace. One of these formulas covers the probability that the restricted singular value of a Gaussian matrix is zero:
recall from~\eqref{eq:P(C,D)} that $\sres{G}{C}{D}=0$ if and only if $C\cap \big(\vct G^T D\big)^\polar\neq\{\vct0\}$. Clearly, if $C=\IR^m$ and $D=\IR^n$, then almost surely $C\cap \big(\vct G^T D\big)^\polar=\{\vct0\}$ if $n\geq m$ and $C\cap \big(\vct G^T D\big)^\polar\neq\{\vct0\}$ if $n< m$. If at least one of the cones is not a linear subspace, then the kinematic formula yields
\begin{equation}\label{eq:prob{sres(G)=0}}
  \Prob\{\sres{G}{C}{D}=0\} = 2\sum_{\substack{k=1\\ k \text{ odd}}}^m \Bigg(\sum_{\ell=0}^{m-n} v_k(C) \, v_\ell(D) + \sum_{\ell=1}^{m-k} v_{k+\ell}(C) \, v_\ell(D)\Bigg) .
\end{equation}
For a proof of this equation we refer to a forthcoming survey. From this equation and from the concentration of intrinsic volumes around their mean~\cite{edge} one also deduces that for cones basically the same threshold behavior appears as for the linear subspaces: $\Prob\{\sres{G}{C}{D}=0\} \approx 0$ if $\sdim(C)<\sdim(D)$ and $\Prob\{\sres{G}{C}{D}=0\} \approx 1$ if $\sdim(C)>\sdim(D)$.

Closely related to the statistical dimension, and in fact chronologically older, is the Gaussian width of a cone, or more precisely, of its intersection with the unit ball.\footnote{Instead of the intersection with the unit ball one can also take the intersection with the unit sphere. But the difference between these quantities is marginal, cf.~Remark~\ref{rem:normprojC-h_K}.} 
We denote the squared Gaussian width by $\sdimw(C)$, i.e.,
\begin{equation}\label{eq:def-sdimw}
  \sdimw(C) := w(C\cap B^m)^2 = \Expect\big[ \norm{\Proj_C(\vct g)} \big]^2 = \Bigg( \sum_{k=1}^m \frac{\sqrt{2}\,\Gamma(\frac{k+1}{2})}{\Gamma(\frac{k}{2})} \, v_k(C) \Bigg)^2 ,
\end{equation}
where the last equality follows from the generalized Steiner formula~\eqref{eq:steiner-mike}. The quantity~$\sdimw(C)$ is closely related to the statistical dimension, in fact,
  \[ \sdimw(C) \leq \sdim(C) \leq \sdimw(C)+1 \]
where the first inequality follows from Jensen's inequality, and where the second inequality has been shown in~\cite{edge}. This ``approximate statistical dimension''~$\sdimw(C)$, is more convenient to work with in the context of Gordon's comparison inequality and has found numerous applications in the literature~\cite{RV:08,stojnic10,CRPW:12,FR:13,RK:13,OTH:13}. However, the statistical dimension
has much better algebraic properties 
and is more natural in the context of cones (in fact, it can be argued~\cite{edge} that the statistical dimension is the \emph{canonical} extension of the dimension from linear subspaces to convex cones), which is why we interpret $\sdimw(C)$ as an approximate version of the statistical dimension~$\sdim(C)$, and not the other way round.

\subsection{Bounds from comparison theorems}

We deal with the distributions of~$\nres{G}{C}{D}$ and $\sres{G}{C}{D}$ through the tail of~$\nres{G}{C}{D}$, $\Prob\big\{\nres{G}{C}{D}\geq \lambda\big\}$, and through the cumulative distribution function (cdf) of~$\sres{G}{C}{D}$, $\Prob\big\{\sres{G}{C}{D}\leq \lambda\big\}$, where $\lambda\geq0$.
Slepian's and Gordon's comparison inequalities open up basically two ways of estimating these functions. First, Gaussian concentration of measure~\cite[Thm.~1.7.6]{Bog:98} yields
\begin{align*}
   \Prob\big\{\nres{G}{C}{D}\geq \mu_1+\lambda\big\} & \leq \exp\big(-\tfrac{\lambda^2}{2}\big) , & \text{where}\qquad \mu_1 & := \Expect[\nres{G}{C}{D}] ,
\\ \Prob\big\{\sres{G}{C}{D}\leq \mu_2-\lambda\big\} & \leq \exp\big(-\tfrac{\lambda^2}{2}\big) , & \text{where}\qquad \mu_2 & := \Expect[\sres{G}{C}{D}] .
\end{align*}
Slightly reformulated, we have
\begin{align}\label{eq:bounds-conc}
   \Prob\big\{\nres{G}{C}{D}\geq \lambda\big\} & \leq \exp\bigg(-\frac{\max\{0,\lambda-\mu_1\}^2}{2}\bigg) , &
   \Prob\big\{\sres{G}{C}{D}\leq \lambda\big\} & \leq \exp\bigg(-\frac{\max\{0,\mu_2-\lambda\}^2}{2}\bigg) .
\end{align}
Second, using the step function $f(t)=1_{t>\lambda}$ we obtain from~\eqref{eq:moments-restr-singvals-claim-norm}
\begin{align*}
   \Prob\big\{\nres{G}{C}{D}+\gamma \geq\lambda\big\} & \leq \Prob\big\{ \|\Proj_D(\vct g')\| + \|\Proj_C(\vct g)\|\geq\lambda\big\} ,
\intertext{and similarly,~\eqref{eq:moments-restr-singvals-claim-sval} implies $\Prob\big\{\sres{G}{C}{D}+\gamma \geq\lambda\big\} \geq \Prob\big\{ \|\Proj_D(\vct g')\| - \|\Proj_C(\vct g)\|\geq\lambda\big\}$, or equivalently,}
   \Prob\big\{\sres{G}{C}{D}+\gamma \leq\lambda\big\} & \leq \Prob\big\{ \|\Proj_D(\vct g')\| - \|\Proj_C(\vct g)\|\leq\lambda\big\} .
\end{align*}
Using a simple union bound trick, we can get rid of the~$\gamma$: 
\begin{align*}
   \Prob\big\{\nres{G}{C}{D}\geq\lambda\big\} = \Prob\big\{\nres{G}{C}{D}+\gamma-\gamma\geq\lambda\big\} & \leq \Prob\big\{\nres{G}{C}{D}+\gamma\geq\lambda\big\} + \Prob\big\{\nres{G}{C}{D}-\gamma\geq\lambda\big\}
\\ & = 2\Prob\big\{\nres{G}{C}{D}+\gamma\geq\lambda\big\} ,
\end{align*}
and similarly for $\sres{G}{C}{D}$. The resulting bounds are
\begin{align}
   \Prob\big\{\nres{G}{C}{D}\geq \lambda\big\} & \leq 2\Prob\big\{ \|\Proj_D(\vct g')\| + \|\Proj_C(\vct g)\|\geq\lambda\big\} ,
\label{eq:bounds-step_fct-nres}
\\ \Prob\big\{\sres{G}{C}{D}\leq \lambda\big\} & \leq 2\Prob\big\{ \|\Proj_D(\vct g')\| - \|\Proj_C(\vct g)\|\leq\lambda\big\} .
\label{eq:bounds-step_fct-sres}
\end{align}
We refer to~\cite[Appendix C]{OTH:13} for a similar approach.

The bounds in~\eqref{eq:bounds-conc} as well as the bounds~\eqref{eq:bounds-step_fct-nres} and~\eqref{eq:bounds-step_fct-sres} are not quite ``ready'', as the bounds in~\eqref{eq:bounds-conc} rely on the expectations~$\mu_1$ and~$\mu_2$, and the bounds~\eqref{eq:bounds-step_fct-nres} and~\eqref{eq:bounds-step_fct-sres} rely on the distribution of $\|\Proj_C(\vct g)\|$ and $\|\Proj_D(\vct g')\|$. The bounds in~\eqref{eq:bounds-conc} are made ready by using the inequalities~\eqref{eq:moments-restr-singvals-claim-norm} and~\eqref{eq:moments-restr-singvals-claim-sval} with $f(t)=t$,
\begin{align*}
   \mu_1 & \leq \sqrt{\sdimw(D)}+\sqrt{\sdimw(C)} , & \mu_2 & \geq \sqrt{\sdimw(D)}-\sqrt{\sdimw(C)} ,
\end{align*}
which results in
\begin{align}
   \Prob\big\{\nres{G}{C}{D}\geq \lambda\big\} & \leq \exp\Bigg(-\frac{\max\big\{0,\lambda-\sqrt{\sdimw(D)}-\sqrt{\sdimw(C)}\big\}^2}{2}\Bigg) ,
\label{eq:bounds-conc-sdimw-nres}
\\ \Prob\big\{\sres{G}{C}{D}\leq \lambda\big\} & \leq \exp\Bigg(-\frac{\max\big\{0,\sqrt{\sdimw(D)}-\sqrt{\sdimw(C)}-\lambda\big\}^2}{2}\Bigg) .
\label{eq:bounds-conc-sdimw-sres}
\end{align}

As for the bounds~\eqref{eq:bounds-step_fct-nres} and~\eqref{eq:bounds-step_fct-sres}, we use the generalized Steiner formula~\eqref{eq:steiner-mike} to express the bounds in terms of the intrinsic volumes, cp.~Remarks~\ref{rem:normprojC-h_K}/\ref{rem:rhs-intrvol-sval}:
\begin{align}
   \Prob\big\{ \|\Proj_D(\vct g')\| + \|\Proj_C(\vct g)\|\geq\lambda\big\} & = \sum_{i=0}^m \sum_{j=0}^n v_i(C) \, v_j(D) \, \Prob\big\{ \chi_j' + \chi_i \geq\lambda\big\} ,
\label{eq:rhs-intrvol-nres}
\\ \Prob\big\{ \|\Proj_D(\vct g')\| - \|\Proj_C(\vct g)\|\geq\lambda\big\} & = \sum_{i=0}^m \sum_{j=0}^n v_i(C) \, v_j(D) \, \Prob\big\{ \chi_j' - \chi_i \geq\lambda\big\} ,
\label{eq:rhs-intrvol-sres}
\end{align}
where $\chi_0=\chi_0'=0$ and $\chi_1,\ldots,\chi_m,\chi_1',\ldots,\chi_n'$ denote independent chi-distributed random variables with $\chi_i$ and $\chi_j'$ having~$i$ and~$j$ degrees of freedom, respectively.

Numerical experiments suggest that the bounds in~\eqref{eq:bounds-step_fct-nres} and~\eqref{eq:bounds-step_fct-sres} are stronger than the bounds in~\eqref{eq:bounds-conc-sdimw-nres} and~\eqref{eq:bounds-conc-sdimw-sres}; especially for~$\sres{G}{C}{D}$ and~$\lambda$ close to zero, which is the interesting regime in the study of Renegar condition number, cf.~also Section~\ref{sec:tube-forms} below. Instead of having a full discussion about the pros and cons of one bound over the other, which would go beyond the scope of this section, we content ourselves with an example to illustrate the differences.

\begin{example}
We consider cones such that $\sdim(C)=20$ and $\sdim(D)=50$. For this we take linear spaces and Lorentz cones, i.e., circular cones of radius~$\frac{\pi}{4}$. More precisely, for~$C$ we take~$\IR^{20}$ and $\Circ_{40}(\frac{\pi}{4})$, and for~$D$ we take~$\IR^{50}$ and $\Circ_{100}(\frac{\pi}{4})$. The approximate statistical dimensions of these cones are given by
\begin{align*}
   \sdimw(\IR^{20}) & \approx 19.51 , & \sdimw\big(\Circ_{40}\big(\tfrac{\pi}{4}\big)\big) & \approx 19.25 , & \sdimw(\IR^{50}) & \approx 49.50 , & \sdimw\big(\Circ_{100}\big(\tfrac{\pi}{4}\big)\big) & \approx 49.25 .
\end{align*}
Furthermore, for $C=\Circ_{40}(\frac{\pi}{4})$ and $D=\Circ_{100}(\frac{\pi}{4})$ we obtain from~\eqref{eq:prob{sres(G)=0}},
\begin{align*}
   \Prob\big\{\sres{G_1}{\IR^{20}}{\IR^{50}}=0\big\} & = 0 , & \Prob\big\{\sres{G_2}{C}{\IR^{50}}=0\big\} & = 0 ,
\\ \Prob\big\{\sres{G_3}{\IR^{20}}{D}=0\big\} & \approx 5\cdot 10^{-6} , & \Prob\big\{\sres{G_4}{C}{D}=0\big\} & \approx 10^{-4}
\end{align*}
where $\vct G_1\in\IR^{50\times20}$, $\vct G_2\in\IR^{50\times40}$, $\vct G_3\in\IR^{100\times20}$, $\vct G_4\in\IR^{100\times40}$ Gaussian matrices.
To compare the bounds \eqref{eq:bounds-step_fct-nres}--\eqref{eq:bounds-conc-sdimw-sres} we define the functions
\begin{align*}
   f(\lambda) & := \exp\Bigg(-\frac{\max\big\{0,\sqrt{50}-\sqrt{20}-\lambda\big\}^2}{2}\Bigg) , & f(C,D;\lambda) & := \min\Big\{1 , 2\, \sum_{i,j} v_i(C) v_j(D) \Prob\big\{ \chi_j - \chi_i' < \lambda\big\} \Big\} ,
\\ g(\lambda) & := \exp\Bigg(-\frac{\max\big\{0,\lambda-\sqrt{50}-\sqrt{20}\big\}^2}{2}\Bigg) , & g(C,D;\lambda) & := \min\Big\{1 , 2\, \sum_{i,j} v_i(C) v_j(D) \Prob\big\{ \chi_j + \chi_i' > \lambda\big\} \Big\} .
\end{align*}
See Figure~\ref{fig:norm-sval-bounds} for the resulting plots with $C\in\big\{\IR^{20},\Circ_{40}\big(\frac{\pi}{4}\big)\big\}$ and $D\in\big\{\IR^{50},\Circ_{100}\big(\frac{\pi}{4}\big)\big\}$.
\begin{figure}[!ht]
   \begin{center}
   \def\myXsc{1}
   \def\myYsc{4}
   \def\myEpsy{0.07}
   \def\myEpsx{\myEpsy*\myXsc/\myYsc}

   \subfloat[singular value (cdf)]{\begin{tikzpicture}[xscale=\myXsc, yscale=\myYsc,scale=1.1, >=stealth]
      \draw (0,0) -- (6,0) (0,0) -- (0,1);
      \foreach \y in {0,.1,.2,.3,.4,.5,.6,.7,.8,.9,1}
        \draw (0,\y) -- ++(-\myEpsy,0) node[left=-1mm]{$\scriptscriptstyle\y$};
      \foreach \x in {0,1,...,6}
        \draw (\x,0) -- ++(0,-\myEpsx) node[below=-1mm]{$\scriptscriptstyle\x$};
      \draw[color=blue,dotted] plot file {tables/sval_m=20_n=50_mean_empirical.table};
      \draw[color=blue, dashed] plot file {tables/sval_m=20_n=50_mean_estimated.table};
      \draw[color=black] plot file {tables/sval_m=20_n=50_cdf_empirical.table};
      \draw[color=magenta] plot file {tables/sval_m=20_n=50_bound_affbdl.table};
      \draw[color=red] plot file {tables/sval_sdim=20_n=50_bound_affbdl.table};
      \draw[color=green] plot file {tables/sval_m=20_sdim=50_bound_affbdl.table};
      \draw[color=cyan] plot file {tables/sval_sdim=20_sdim=50_bound_affbdl.table};
      \draw[color=black] plot file {tables/sval_m=20_n=50_bound_conc.table};
      \draw[<-] (2.6,0.5) -- (3.22,0.5) node[right=-1mm]{$\scriptstyle \sqrt{50}-\sqrt{20}$};
      \draw[<-] (2.79,0.35)  -- (3.22,0.35) node[right=-1mm]{$\scriptstyle \Prob\{\sigma(\vct G)<\lambda\}$};
      \draw[<-] (2.9,0.2) -- (3.22,0.2) node[right=-1mm]{$\scriptstyle \Expect[\sigma(\vct G)]$};
      \path (1.5,0.9) node[left=-1mm]{$\scriptstyle f(\lambda)$};
      \path (1.5,0.8) node[left=-1mm]{$\scriptstyle f(C,D;\lambda)$};
      \path (1.5,0.7) node[left=-1mm]{$\scriptstyle f(\IR^{20},D;\lambda)$};
      \path (1.5,0.6) node[left=-1mm]{$\scriptstyle f(C,\IR^{50};\lambda)$};
      \path (1.5,0.5) node[left=-1mm]{$\scriptstyle f(\IR^{20},\IR^{50};\lambda)$};
      \draw[->] (1.5,0.9) -- (1.972,0.82);
      \draw[->,cyan] (1.5,0.8) -- (2.12,0.695);
      \draw[->,green] (1.5,0.7) -- (2.01,0.615);
      \draw[->,red] (1.5,0.6) -- (1.924,0.536);
      \draw[->,magenta] (1.5,0.5) -- (1.864,0.44);
   \end{tikzpicture}}
   \hfill
   \subfloat[norm (tail)]{\begin{tikzpicture}[xscale=\myXsc, yscale=\myYsc,scale=1.1, >=stealth]
      \draw (9,0) -- (15,0) (9,0) -- (9,1);
      \foreach \y in {0,.1,.2,.3,.4,.5,.6,.7,.8,.9,1}
        \draw (9,\y) -- ++(-\myEpsy,0) node[left=-1mm]{$\scriptscriptstyle\y$};
      \foreach \x in {9,10,...,15}
        \draw (\x,0) -- ++(0,-\myEpsx) node[below=-1mm]{$\scriptscriptstyle\x$};
      \draw[color=black] plot file {tables/norm_m=20_n=50_tail_empirical.table};
      \draw[color=blue,dotted] plot file {tables/norm_m=20_n=50_mean_empirical.table};
      \draw[color=blue,dashed] plot file {tables/norm_m=20_n=50_mean_estimated.table};
      \draw[color=magenta] plot file {tables/norm_m=20_n=50_bound_affbdl.table};
      \draw[color=green] plot file {tables/norm_m=20_sdim=50_bound_affbdl.table};
      \draw[color=red] plot file {tables/norm_sdim=20_n=50_bound_affbdl.table};
      \draw[color=cyan] plot file {tables/norm_sdim=20_sdim=50_bound_affbdl.table};
      \draw[color=black] plot file {tables/norm_m=20_n=50_bound_conc.table};
      \draw[->] (10.7,0.5) node[left=-1mm]{$\scriptstyle \sqrt{50}+\sqrt{20}$} -- (11.54,0.5);
      \draw[->] (10.7,0.35) node[left=-1mm]{$\scriptstyle \Prob\{\|\vct G\|>\lambda\}$} -- (11.16,0.35);
      \draw[->] (10.7,0.2) node[left=-1mm]{$\scriptstyle \Expect[\|\vct G\|]$} -- (11.02,0.2);
      \path (13.5,0.65) node[right=-1mm]{$\scriptstyle g(\lambda)$};
      \path (13.5,0.55) node[right=-1mm]{$\scriptstyle g(C,D;\lambda)$};
      \path (13.5,0.45) node[right=-1mm]{$\scriptstyle g(\IR^{20},D;\lambda)$};
      \path (13.5,0.35) node[right=-1mm]{$\scriptstyle g(C,\IR^{50};\lambda)$};
      \path (13.5,0.25) node[right=-1mm]{$\scriptstyle g(\IR^{20},\IR^{50};\lambda)$};
      \draw[->] (13.5,0.65) -- (12.7,0.51);
      \draw[->,cyan] (13.5,0.55) -- (12.475,0.372);
      \draw[->,green] (13.5,0.45) -- (12.620,0.289);
      \draw[->,red] (13.5,0.35) -- (12.8,0.22);
      \draw[->,magenta] (13.5,0.25) -- (12.956,0.136);
   \end{tikzpicture}}
   \end{center}
   \caption{Bounds for the distributions of the restricted norm and the restricted singular value; $C:=\Circ_{40}\big(\frac{\pi}{4}\big)$ and $D:=\Circ_{100}\big(\frac{\pi}{4}\big)$, $\vct G\in\IR^{20\times 50}$ Gaussian.}
   \label{fig:norm-sval-bounds}
\end{figure}
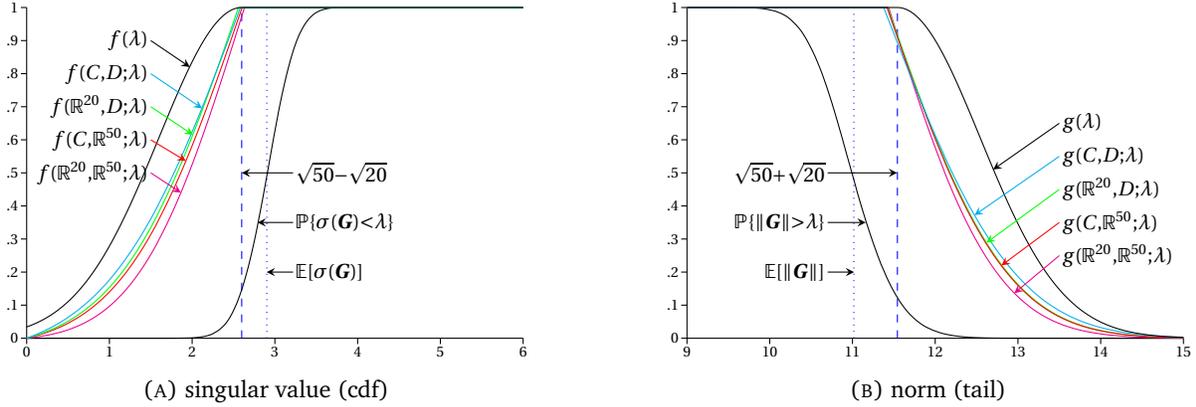
\end{example}

\subsection{Bounds from tube formulas}\label{sec:tube-forms}

We finally return to the other line of research, introduced in Section~\ref{sec:2-viewpoints}, which stems from the analysis of condition numbers, and which lays its focus on the tubular neighborhood of a set of ill-posed inputs. Specifically, we consider the (generalized) Renegar condition, cf.~Section~\ref{sec:gen-feas-prob}, which is given by
  \[ \RCD{A}{C}{D}=\frac{\|\vct A\|}{\dist(\vct A,\Sigma(C,D))} ,\qquad \dist(\vct A,\Sigma(C,D)) = \max\big\{\sres{A}{C}{D},\srestm{A}{D}{C}\big\} , \]
where, as usual $C\subseteq\IR^m$, $D\subseteq\IR^n$ are closed convex cones, $\vct A\in\IR^{n\times m}$. Interesting parameters for convex optimization are, for example, $C=\IR_+^m$, $D=\IR^n$, $n\leq m$, which corresponds to classical linear programming. Choosing for~$C$ the cone of nonnegative semidefinite matrices yields semidefinite programming.

The condition number $\RCD{A}{C}{D}$ measures the computational hardness of the input~$\vct A$ (with respect to the cones~$C$ and~$D$), so the goal is to derive upper bounds for, say, $\Prob\{\RCD{G}{C}{D}\geq t\}$ where $\vct G\in\IR^{n\times m}$ Gaussian. Here, the role of the norm is of secondary importance, so that the main interest lies in the cdf of the distance to ill-posedness,
  \[ \Prob\big\{ \dist(\vct G,\Sigma(C,D)) \leq \tfrac{1}{t}\big\} . \]
In order to get this into a more familiar form, recall that $\P(C,D)\cup\D(C,D)=\IR^{n\times m}$ unless $C=D=\IR^m$. In particular, if $(C,D)\neq(\IR^m,\IR^m)$ then
$\Prob\big\{ \sres{G}{C}{D}\leq \frac{1}{t} \text{ or } \srestm{G}{D}{C}\leq \frac{1}{t}\big\} = 1$, and thus
\begin{align*}
   \Prob\{ \dist(\vct G,\Sigma(C,D))\leq\lambda\} & = \Prob\{ \sres{G}{C}{D}\leq \lambda \text{ and } \srestm{G}{D}{C}\leq \lambda\}
\\ & = \Prob\{ \sres{G}{C}{D}\leq \lambda\} + \Prob\{ \srestm{G}{D}{C}\leq \lambda\} - 1 .
\end{align*}
Assuming $\sdim(C)<\sdim(D)$ (without loss of generality by symmetry) and using the asymptotics derived from~\eqref{eq:prob{sres(G)=0}}, we see that one is basically interested in bounds for the cdf of the restricted singular value~$\Prob\{\sres{G}{C}{D}\leq\lambda\}$ for small~$\lambda$.
The bounds obtained from Gordon's comparison theorem are not strong enough to even show that $\Expect\big[\log \RCD{G}{C}{D}\big] < \infty$. 
In fact, applying the bounds~\eqref{eq:bounds-step_fct-nres}/\eqref{eq:bounds-step_fct-sres} to the above expression for the distance to ill-posedness yields a right-hand side that fails to converge to zero as $\lambda\to 0$.

On the other hand, the bounds obtained from tube formulas show that the expectation of the logarithm of the condition number is not only finite but polynomial in the dimensions. To be more precise, for $n\leq m$ let $\St_{n,m}:=\{\vct U\in\IR^{n\times m}\mid \vct U\vct U^T=\Id_n\}$ denote the Stiefel manifold, which is compact and admits a unique orthogonal invariant probability measure. If $\vct U\in\St_{n,m}$ is uniformly at random, then in~\cite{AB:13} a bound of the form
\begin{align*}
   \Prob\{\Ren_C(\vct U)\geq t\} & = \Prob\big\{\sres{U}{C}{\IR^n}\leq \tfrac{1}{t}\big\} \leq \sum_k v_k(C)\, f_{m,n,k}(t)
\intertext{was given, where $f_{n,m,k}(t)\to0$ for $t\to\infty$ (and estimated to derive polynomial bounds for the expected logarithm of the condition). In fact, it can be shown~\cite{A:14} that for the general Renegar condition one obtains}
   \Prob\{\Ren_{C,D}(\vct U)\geq t\} & = \Prob\big\{\sres{U}{C}{D}\leq \tfrac{1}{t}\big\} \leq \sum_{k,\ell} v_k(C)\, v_\ell(D)\, f_{m,n,k,\ell}(t) ,
\end{align*}
with similar coefficient functions satisfying $f_{m,n,k,n}(t)=f_{m,n,k}(t)$. Estimating these expressions is a nontrivial matter on its own, but maybe the framework provided in this paper can help with this task.

\section{Conclusion}

In this paper we have provided a unifying framework to establish and clarify connections between random matrix theory, integral geometry, and convex optimization. We hope that these connections will be developed further and encourage people from each area to adapt tools and techniques from the other. We finish with some open questions that we find particularly attractive to deserve further study.

{\bf Asymptotics.} 
In the unrestricted case it is known, cf.~Section~\ref{sec:condnumbs-randommatr}, that the estimates for the largest and the smallest singular values, which are derived from Slepian's and Gordon's inequalities, are asymptotically exact. Is this also the case for the restricted versions? More precisely, let $(C_m)$, $(D_n)$ be sequences of closed convex cones, $C_m\subseteq\IR^m$, $D_n\subseteq\IR^n$, such that the relative statistical dimensions converge to $c,d\in[0,1]$, respectively, $\lim_{m\to\infty}\frac{\sdim(C_m)}{m}=c$ and $\lim_{n\to\infty}\frac{\sdim(D_n)}{n}=d$. For a Gaussian $(n\times m)$-matrix, $n>m$, do we have asymptotically (with $\frac{m}{n}$ converging to a limit in~$(0,1)$) $\nres{G}{C}{D}\sim\sqrt{d n}+\sqrt{c m}$ and $\sres{G}{C}{D}\sim\sqrt{d n}-\sqrt{c m}$? Moreover, for a square Gaussian $(m\times m)$-matrix $\mtx{G}$, do we have the asymptotics $\nres{G}{C}{C}\sim\sqrt{c m}$ and $\sres{G}{C}{C}\sim\frac{1}{\sqrt{c m}}$?

If these questions can be answered in the affirmative, can they as well be lifted to general subgaussian distributions?\footnote{It seems that at least the bounds derived from Slepian's and Gordon's inequalities can be extended to the subgaussian case rather straightforwardly~\cite{Tr:14}.}

{\bf Renegar's condition number.} In Section~\ref{sec:comp-bounds} we have seen that Gordon's inequality, applied to the restricted singular values in the standard ways described, does not 
yield bounds on the distribution of Renegar's condition number (or even for the rectangular matrix condition number) that are good enough to recreate the bounds as obtained, for example, 
in~\cite{CD:05} or in~\cite{AB:13} in terms of tube formulas and integral geometry. Is it possible to formulate a contraction map for a convex bundle underlying the distance to ill-posedness, rather than
just the individual restricted singular values, in a way that would lead to bounds
\begin{equation*}
 \Prob\{\dist(\mtx{G},\Sigma)\leq \lambda\} \leq f(\lambda)
\end{equation*}
such that $f(\lambda)\to 0$ as $\lambda \to 0$, using the methods of Gaussian comparison inequalities?

{\bf Kinematic formula for random maps.} The generalized Steiner formula~\eqref{eq:steiner-mike} in the case of polyhedral cones~$C$ is based on a decomposition of~$\R^m$ induced by the facial
decomposition of~$C$ and its
polar~$C^\polar$. In a similar way, a map~$\mtx{A}_{C\to D}$ between polyhedral cones induces a decomposition of~$\R^{n\times m}$. Is it possible to develop a ``Steiner formula'' that would express the restricted
norms and singular values of Gaussian matrices in terms of ``matricial intrinsic volumes'', and would contain the Steiner formula as a special case for $\norm{\Proj_C(\vct{g})}$? Answering such a question could allow to transfer Edelman's exact formulas~\cite{Edelman88} to the cone-restricted case.

\bibliographystyle{myalpha}
\bibliography{statdim}

\def\cprime{$'$} \def\cprime{$'$}
\begin{thebibliography}{CRPW12}

\bibitem[AB13]{AB:13}
D.~Amelunxen and P.~B\"urgisser.
\newblock Probabilistic analysis of the {G}rassmann condition number.
\newblock {\em Found. Comput. Math.}, 2013.

\bibitem[ALMT14]{edge}
D.~Amelunxen, M.~Lotz, M.~B. McCoy, and J.~A. Tropp.
\newblock Living on the edge: phase transitions in convex programs with random
  data.
\newblock {\em Information and Inference}, 2014.

\bibitem[Ame11]{am:thesis}
D.~Amelunxen.
\newblock Geometric analysis of the condition of the convex feasibility
  problem.
\newblock PhD Thesis, Univ. Paderborn, 2011.

\bibitem[Ame14]{A:14}
D.~Amelunxen.
\newblock A bound for the probability that two convex cones define a small
  angle.
\newblock In preparation, 2014.

\bibitem[AW04]{azais2004upper}
J.-M. Aza{\"\i}s and M.~Wschebor.
\newblock Upper and lower bounds for the tails of the distribution of the
  condition number of a gaussian matrix.
\newblock {\em SIAM Journal on Matrix Analysis and Applications},
  26(2):426--440, 2004.

\bibitem[Bar02]{Barv}
A.~Barvinok.
\newblock {\em A course in convexity}, volume~54 of {\em Graduate Studies in
  Mathematics}.
\newblock American Mathematical Society, Providence, RI, 2002.

\bibitem[BC10]{BC:10}
P.~B\"{u}rgisser and F.~Cucker.
\newblock Smoothed analysis of {M}oore--{P}enrose {I}nversion.
\newblock {\em SIAM Journal on Matrix Analysis and Applications},
  31(5):2769--2783, 2010.

\bibitem[BC13]{Condition}
P.~B\"urgisser and F.~Cucker.
\newblock {\em Condition: The geometry of numerical algorithms}.
\newblock Number 349 in Grundlehren der {M}athematischen {W}issenschaften.
  Springer Verlag, 2013.

\bibitem[BCL06]{BCL:06a}
P.~B\"urgisser, F.~Cucker, and M.~Lotz.
\newblock Smoothed analysis of complex conic condition numbers.
\newblock {\em J. Math. Pures et Appl.}, 86:293--309, 2006.

\bibitem[BCL08]{BCL:08}
P.~B{\"u}rgisser, F.~Cucker, and M.~Lotz.
\newblock The probability that a slightly perturbed numerical analysis problem
  is difficult.
\newblock {\em Math. Comp.}, 77(263):1559--1583, 2008.

\bibitem[BF09]{BF:09}
A.~Belloni and R.~M. Freund.
\newblock A geometric analysis of {R}enegar's condition number, and its
  interplay with conic curvature.
\newblock {\em Math. Program.}, 119(1, Ser. A):95--107, 2009.

\bibitem[Bog98]{Bog:98}
V.~I. Bogachev.
\newblock {\em Gaussian measures}, volume~62 of {\em Mathematical Surveys and
  Monographs}.
\newblock American Mathematical Society, Providence, RI, 1998.

\bibitem[B{\"u}r10]{BuergICM}
P.~B{\"u}rgisser.
\newblock Smoothed analysis of condition numbers.
\newblock In {\em Proceedings of the {I}nternational {C}ongress of
  {M}athematicians, Hyderabad}, volume~IV, pages 2609--2633. World Scientific,
  2010.

\bibitem[BV04]{BV:04}
S.~Boyd and L.~Vandenberghe.
\newblock {\em Convex optimization}.
\newblock Cambridge University Press, Cambridge, 2004.

\bibitem[BY93]{bai1993}
Z.~D. Bai and Y.~Q. Yin.
\newblock Limit of the smallest eigenvalue of a large dimensional sample
  covariance matrix.
\newblock {\em The Annals of Probability}, pages 1275--1294, 1993.

\bibitem[CD05]{CD:05}
Z.~Chen and J.~J. Dongarra.
\newblock Condition numbers of {G}aussian random matrices.
\newblock {\em SIAM J. Matrix Anal. Appl.}, 27(3):603--620 (electronic), 2005.

\bibitem[CJ12]{chan:12}
V.~Chandrasekaran and M.~Jordan.
\newblock Computational and statistical tradeoffs via convex relaxation.
\newblock {\em arXiv preprint arXiv:1211.1073}, 2012.

\bibitem[CRPW12]{CRPW:12}
V.~Chandrasekaran, B.~Recht, P.~A. Parrilo, and A.~S. Willsky.
\newblock The convex geometry of linear inverse problems.
\newblock {\em Found. Comput. Math.}, 12(6):805--849, 2012.

\bibitem[Dem87]{Demmel87}
J.~Demmel.
\newblock On condition numbers and the distance to the nearest ill-posed
  problem.
\newblock {\em Numer. Math.}, 51:251--289, 1987.

\bibitem[Dem88]{Demmel88}
J.~Demmel.
\newblock The probability that a numerical analysis problem is difficult.
\newblock {\em Math. Comp.}, 50:449--480, 1988.

\bibitem[DS01]{davidson2001local}
K.~R. Davidson and S.~J. Szarek.
\newblock Local operator theory, random matrices and banach spaces.
\newblock {\em Handbook of the geometry of Banach spaces}, 1:317--366, 2001.

\bibitem[Ede88]{Edelman88}
A.~Edelman.
\newblock Eigenvalues and condition numbers of random matrices.
\newblock {\em SIAM J. of Matrix Anal. and Applic.}, 9:543--556, 1988.

\bibitem[ES05]{Edelman2005tails}
A.~Edelman and B.~D. Sutton.
\newblock Tails of condition number distributions.
\newblock {\em SIAM journal on matrix analysis and applications},
  27(2):547--560, 2005.

\bibitem[FR13]{FR:13}
S.~Foucart and H.~Rauhut.
\newblock {\em A mathematical introduction to compressive sensing}, volume 336
  of {\em Applied and Numerical Harmonic Analysis}.
\newblock Birkh\"auser, Basel, 2013.

\bibitem[Gor85]{G:85}
Y.~Gordon.
\newblock Some inequalities for {G}aussian processes and applications.
\newblock {\em Israel J. Math.}, 50(4):265--289, 1985.

\bibitem[Gor87]{G:87}
Y.~Gordon.
\newblock Elliptically contoured distributions.
\newblock {\em Probability theory and related fields}, 76(4):429--438, 1987.

\bibitem[KR97]{KR:97}
D.~A. Klain and G.-C. Rota.
\newblock {\em Introduction to geometric probability}.
\newblock Lezioni Lincee. [Lincei Lectures]. Cambridge University Press,
  Cambridge, 1997.

\bibitem[KW12]{KW:12}
N.~Krislock and H.~Wolkowicz.
\newblock Euclidean distance matrices and applications.
\newblock In {\em Handbook on semidefinite, conic and polynomial optimization},
  volume 166 of {\em Internat. Ser. Oper. Res. Management Sci.}, pages
  879--914. Springer, New York, 2012.

\bibitem[Led01]{ledo:01}
M.~Ledoux.
\newblock {\em The concentration of measure phenomenon}, volume~89 of {\em
  Mathematical Surveys and Monographs}.
\newblock American Mathematical Society, Providence, RI, 2001.

\bibitem[LT91]{LT:91}
M.~Ledoux and M.~Talagrand.
\newblock {\em Probability in Banach Spaces: Isoperimetry and Processes}.
\newblock A Series of Modern Surveys in Mathematics Series. Springer-Verlag,
  Berlin, 1991.

\bibitem[Mau11]{M:11}
A.~Maurer.
\newblock A proof of {S}lepian's inequality.
\newblock
  \href{http://www.andreas-maurer.eu/Slepian3.pdf}{www.andreas-maurer.eu/Slepian3.pdf},
  2011.

\bibitem[MT13a]{MT:13b}
M.~B. McCoy and J.~A. Tropp.
\newblock The achievable performance of convex demixing.
\newblock \href{http://arxiv.org/abs/1309.7478}{arXiv:1309.7478v1 [cs.IT]},
  2013.

\bibitem[MT13b]{MT:13}
M.~B. McCoy and J.~A. Tropp.
\newblock From {S}teiner formulas for cones to concentration of intrinsic
  volumes.
\newblock \href{http://arxiv.org/abs/1308.5265}{arXiv:1308.5265v1 [math.MG]},
  2013.

\bibitem[OTH13]{OTH:13}
S.~Oymak, C.~Thrampoulidis, and B.~Hassibi.
\newblock The squared error of generalized {L}{A}{S}{S}{O}: a precise analysis.
\newblock \href{http://arxiv.org/abs/1311.0830}{arXiv:1311.0830v2 [cs.IT]},
  2013.

\bibitem[Ren94]{rene:94}
J.~Renegar.
\newblock Some perturbation theory for linear programming.
\newblock {\em Math. Programming}, 65(1, Ser. A):73--91, 1994.

\bibitem[Ren95a]{rene:95b}
J.~Renegar.
\newblock Incorporating condition measures into the complexity theory of linear
  programming.
\newblock {\em SIAM J. Optim.}, 5(3):506--524, 1995.

\bibitem[Ren95b]{rene:95a}
J.~Renegar.
\newblock Linear programming, complexity theory and elementary functional
  analysis.
\newblock {\em Math. Programming}, 70(3, Ser. A):279--351, 1995.

\bibitem[RK13]{RK:13}
H.~Rauhut and M.~Kabanava.
\newblock Analysis $\ell_1$-recovery with frames and {G}aussian measurements.
\newblock \href{http://arxiv.org/abs/1306.1356}{arXiv:1306.1356v2 [cs.IT]},
  2013.

\bibitem[RV08]{RV:08}
M.~Rudelson and R.~Vershynin.
\newblock On sparse reconstruction from {F}ourier and {G}aussian measurements.
\newblock {\em Comm. Pure Appl. Math.}, 61(8):1025--1045, 2008.

\bibitem[RV09]{RV:09rect}
M.~Rudelson and R.~Vershynin.
\newblock Smallest singular value of a random rectangular matrix.
\newblock {\em Comm. Pure Appl. Math.}, 62(12):1707--1739, 2009.

\bibitem[RV10]{rudelson2010non}
M.~Rudelson and R.~Vershynin.
\newblock Non-asymptotic theory of random matrices: extreme singular values.
\newblock In {\em Proceedings of the {I}nternational {C}ongress of
  {M}athematicians, Hyderabad}, volume III, pages 1576--1602. World Scientific,
  2010.

\bibitem[Sch93]{Schn:book}
R.~Schneider.
\newblock {\em Convex bodies: the {B}runn-{M}inkowski theory}, volume~44 of
  {\em Encyclopedia of Mathematics and its Applications}.
\newblock Cambridge University Press, Cambridge, 1993.

\bibitem[Sma81]{Smale81}
S.~Smale.
\newblock The fundamental theorem of algebra and complexity theory.
\newblock {\em Bulletin of the AMS}, 4:1--36, 1981.

\bibitem[Sma97]{Sm:97}
S.~Smale.
\newblock Complexity theory and numerical analysis.
\newblock In {\em Acta numerica, 1997}, volume~6 of {\em Acta Numer.}, pages
  523--551. Cambridge Univ. Press, Cambridge, 1997.

\bibitem[ST02]{ST:02}
D.~A. Spielman and S.-H. Teng.
\newblock Smoothed analysis of algorithms.
\newblock In {\em Proceedings of the {I}nternational {C}ongress of
  {M}athematicians, {V}ol. {I} ({B}eijing, 2002)}, pages 597--606, Beijing,
  2002. Higher Ed. Press.

\bibitem[Sto09]{stojnic10}
M.~Stojnic.
\newblock Various thresholds for $\ell_1$-optimization in compressed sensing.
\newblock {\em preprint}, 2009.
\newblock arXiv:0907.3666.

\bibitem[SW08]{SW:08}
R.~Schneider and W.~Weil.
\newblock {\em Stochastic and integral geometry}.
\newblock Probability and its Applications (New York). Springer-Verlag, Berlin,
  2008.

\bibitem[Tro14]{Tr:14}
J.~A. Tropp.
\newblock Convex recovery of a structured signal from independent random linear
  measurements.
\newblock \href{http://arxiv.org/abs/1405.1102}{arXiv:1405.1102v1 [cs.IT]},
  2014.

\bibitem[TV07]{TV:07}
T.~Tao and V.~Vu.
\newblock The condition number of a randomly perturbed matrix.
\newblock In {\em S{TOC}'07---{P}roceedings of the 39th {A}nnual {ACM}
  {S}ymposium on {T}heory of {C}omputing}, pages 248--255. ACM, New York, 2007.

\bibitem[Ver12]{Vershynin2012}
R.~Vershynin.
\newblock Introduction to the non-asymptotic analysis of random matrices.
\newblock In Y.~C. Eldar and G.~Kutyniok, editors, {\em Compressed sensing},
  pages xii+544. Cambridge University Press, Cambridge, 2012.
\newblock Theory and applications.

\bibitem[VRPH07]{VRP:07}
J.~C. Vera, J.~C. Rivera, J.~Pe{\~n}a, and Y.~Hui.
\newblock A primal-dual symmetric relaxation for homogeneous conic systems.
\newblock {\em J. Complexity}, 23(2):245--261, 2007.

\bibitem[Wsc04]{Wsch:04}
M.~Wschebor.
\newblock Smoothed analysis of $\kappa(a)$.
\newblock {\em J. of Complexity}, 20:97--107, 2004.

\end{thebibliography}

\appendix

\section{The biconic feasibility problem - proofs}\label{sec:appendix-convex}

In this appendix we 
provide the proofs for Section~\ref{sec:gen-feas-prob}.
Recall that for $C\subseteq\IR^m$, $D\subseteq\IR^n$ closed convex cones, the biconic feasibility problem is given by
\\\def\tmpX{3mm}
\begin{minipage}{0.46\textwidth}
\begin{align}
   \exists \vct x & \in C\setminus\{\vct0\} \quad\text{s.t.} \hspace{\tmpX} \vct{Ax} \in D^\polar ,
\tag{P}
\label{eq:(P1)}
\end{align}
\end{minipage}
\rule{0.07\textwidth}{0mm}
\begin{minipage}{0.46\textwidth}
\begin{align}
   \exists \vct y & \in D\setminus\{\vct0\} \quad\text{s.t.} \hspace{\tmpX} -\vct A^T\vct y\in C^\polar ,
\tag{D}
\label{eq:(D1)}
\end{align}
\end{minipage}
\\[2mm] and the sets of primal feasible and dual feasible instances can be characterized by
\begin{align*}
   \P(C,D) & = \big\{\vct A\in\IR^{n\times m}\mid C\cap \big(\vct A^T D\big)^\polar\neq\{\vct0\}\big\} = \{\vct A\in\IR^{n\times m}\mid \sres{A}{C}{D}=0 \} ,
\\ \D(C,D) & = \{\vct A\in\IR^{n\times m}\mid D\cap (-\vct A C)^\polar\neq\{\vct0\}\} =  \{\vct A\in\IR^{n\times m}\mid \srestm{A}{D}{C}=0 \} ,
\end{align*}
respectively, cf.~\eqref{eq:P(C,D)}/\eqref{eq:D(C,D)}. The proof of Proposition~\ref{prop:primaldual} uses the following generalization of Farkas' Lemma.

\begin{lemma}\label{lem:farkas-gen}
Let $C,\tilde C\subseteq\IR^m$ be closed convex cones with $\inter(C)\neq\emptyset$. Then
\begin{equation}\label{eq:Farkas-gen}
  \inter(C)\cap \tilde C=\emptyset \iff C^\polar\cap (-\tilde C^\polar)\neq\{\vct0\} .
\end{equation}
\end{lemma}

\begin{proof}
If $\inter(C)\cap \tilde C=\emptyset$, then there exists a separating hyperplane $H=\vct v^\bot$, $\vct v\neq\vct0$, so that $\langle \vct v,\vct x\rangle \leq 0$ for all $\vct x\in C$ and $\langle \vct v,\vct y\rangle \geq 0$ for all $\vct y\in \tilde C$. 
But this means $\vct v\in C^\polar\cap (-\tilde C^\polar)$.
On the other hand, if $\vct x\in \inter(C)\cap \tilde C$ then only in the case $C=\IR^m$, for which the claim is trivial, can $\vct x=\vct0$. If $\vct x\neq\vct0$, then~$C^\polar\setminus\{\vct0\}$ lies in the open half-space~$\{\vct v\mid \langle \vct v,\vct x\rangle<0\}$ and $-\tilde C^\polar$ lies in the closed half-space $\{\vct v\mid \langle \vct v,\vct x\rangle\geq0\}$, and thus $C^\polar\cap(-\tilde C^\polar)=\{\vct0\}$.
\end{proof}

For the proof of the third claim in Proposition~\ref{prop:primaldual} we also need the following well-known convex geometric lemma; a proof can be found, for example, in~\cite[proof of Thm.~6.5.6]{SW:08}. We say that two cones $C,D\subseteq\IR^m$, with $\inter(C)\neq\emptyset$, \emph{touch} if $C\cap D\neq\{\vct0\}$ but $\inter(C)\cap D=\emptyset$.

\begin{lemma}\label{lem:touch}
Let $C,D\subseteq\IR^m$ closed convex cones with $\inter(C)\neq\emptyset$. If $\vct Q\in O(m)$ uniformly at random, then the randomly rotated cone $\vct QD$ almost surely does not touch~$C$.
\end{lemma}

\begin{proof}[Proof of Proposition~\ref{prop:primaldual}]
(1) The sets $\P(C,D)$ and $\D(C,D)$ are closed as they are preimages of the closed set $\{0\}$ under continuous functions,
c.f.~\eqref{eq:P(C,D)}/\eqref{eq:D(C,D)}. Indeed, for any $\vct{x}$, the function $\mtx{A}\mapsto \norm{\Proj_D(\mtx{A}\vct{x})}$ is continuous, and as a minimum of such functions over the compact set $C\cap S^{m-1}$, it
follows that $\sigma_{C\to D}(\mtx{A})$ is continuous. Hence, $\P(C,D)=\{\vct A\in\IR^{n\times m}\mid \sres{A}{C}{D}=0 \}$ is closed.
The same argument applies to $\D(C,D)$.

(2) For the claim about the union of the sets $\P(C,D)$ and $\D(C,D)$ we first consider the case $C\neq\IR^m$, so that $\vct 0\not\in\inter(C)$. Using the generalized Farkas' Lemma~\ref{lem:farkas-gen}, we obtain
\begin{align*}
   \vct A\not\in\P(C,D) & \iff C\cap \big(\vct A^T D\big)^\polar = \{\vct0\} \;\Rightarrow\;  \inter(C)\cap \big(\vct A^T D\big)^\polar = \emptyset \stackrel{\eqref{eq:Farkas-gen}}{\Longrightarrow} C^\polar\cap (-\vct A^T D)\neq\{\vct0\} \;\Rightarrow\; \vct A\in\D(C,D) .
\end{align*}
This shows $\P(C,D)\cup\D(C,D)=\IR^{n\times m}$. For $D\neq\IR^n$ the argument is the same.
For $C=\IR^m$ and $D=\IR^n$:
\begin{align*}
   \P(\IR^m,\IR^n) & = \big\{\vct A\in\IR^{n\times m}\mid \ker\vct A\neq\{\vct0\}\big\} = \begin{cases} \{\text{rank deficient matrices}\} & \text{if } m\leq n \\ \IR^{n\times m} & \text{if } m>n , \end{cases}
\\ \D(\IR^m,\IR^n) & = \big\{\vct A\in\IR^{n\times m}\mid \ker\vct A^T\neq\{\vct0\}\big\} = \begin{cases} \IR^{n\times m} & \text{if } m<n \\ \{\text{rank deficient matrices}\} & \text{if } m\geq n . \end{cases}
\end{align*}
In particular,this shows $\P(\IR^m,\IR^m)\cup \D(\IR^m,\IR^m)=\{\text{rank deficient matrices}\}$.

(3) If $(C,D)=(\IR^m,\IR^n)$ then by the characterization above $\Sigma(\IR^m,\IR^n)$ consists of the rank deficient matrices, which is a nonempty set. If $(C,D)\neq(\IR^m,\IR^m)$, then the union of the closed sets $\P(C,D)$ and $\D(C,D)$ equals~$\IR^{n\times m}$, which is an irreducible topological space, so that their intersection $\Sigma(C,D)=\P(C,D)\cap\D(C,D)$ must be nonempty.

As for the claim about the Lebesgue measure of $\Sigma(C,D)$, we may use the symmetry between~\eqref{eq:(P1)} and~\eqref{eq:(D1)} to assume without loss of generality $n\leq m$. If $\vct A\in\IR^{n\times m}$ has full rank, then $\vct AC$ has nonempty interior and from Proposition~\ref{prop:nres=0,sres=0} and Farkas' Lemma,
\begin{align*}
   \sres{A}{C}{D}=0 & \iff C\cap (\vct A^TD)^\polar\neq\{\vct0\} \iff \vct AC\cap D^\polar\neq\{\vct0\} \;\text{or}\; \ker\vct A\cap C\neq\{\vct0\} ,
\\ \srestm{A}{D}{C}=0 & \iff D\cap (-\vct AC)^\polar\neq\{\vct0\} \stackrel{\eqref{eq:Farkas-gen}}{\iff} D^\polar\cap \inter(\vct AC)=\emptyset .
\end{align*}
Note that if $\vct{Ax}=\vct0$ for some $\vct x\in \inter(C)$, then~$\vct A$, being a continuous surjection, maps an open neighborhood of~$\vct x$ to an open neighborhood of the origin, so that~$\vct AC=\IR^n$. Hence, $D\cap (-\vct AC)^\polar\neq\{\vct0\}$ implies $\ker\vct A\cap \inter(C)=\emptyset$, since otherwise $\vct AC=\IR^n$, i.e., $(\vct AC)^\polar=\{\vct0\}$.

If $\vct A\in\Sigma(C,D)$, i.e., $\sres{A}{C}{D}=\srestm{A}{D}{C}=0$, and if $\vct A$ has full rank, then $\vct AC\cap D^\polar\neq\{\vct0\}$ implies that~$D^\polar$ touches~$\vct AC$, while $\ker\vct A\cap C\neq\{\vct0\}$ implies that~$\ker \vct A$ touches~$C$. Hence, if $\vct A=\vct G$ Gaussian, then $\vct G$ has almost surely full rank, and Lemma~\ref{lem:touch} implies that both touching events have zero probability, so that almost surely $\vct G\not\in\Sigma(C,D)$.
\end{proof}

We next provide the proof for the characterization of the restricted singular values as distances to the primal and dual feasible sets. From now on we use again the short-hand notation $\P:=\P(C,D)$ and $\D:=\D(C,D)$.

\begin{proof}[Proof of Proposition~\ref{prop:sing-dist}]
By symmetry, it suffices to show that $\dist(\vct A,\P)=\sres{A}{C}{D}$. If $\vct A\in\P$ then $\dist(\vct A,\P)=0=\sres{A}{C}{D}$, so assume that $\vct A\not\in\P$. Let $\DA\in\IR^{n\times m}$ such that $\vct A+\DA\in\P$ and $\dist(\vct A,\P)=\|\DA\|$. 
Since $\vct A+\DA\in\P$, there exists
$\vct x_0\in C\cap S^{m-1}$ such that $\vct w_0:=(\vct A+\DA)\vct x_0\in 
D^\polar$. For all $\vct y\in D$
  \[ 0\geq \langle\vct w_0,\vct y\rangle = \langle (\vct A+\DA)\vct x_0,\vct y\rangle = \langle \vct{Ax}_0,\vct y\rangle - \langle -\DA\vct x_0,\vct y\rangle . \]
If $\vct y_0\in B^n\cap D$ is such that $\|\Pi_D(\vct{Ax}_0)\|=\langle \vct{Ax}_0,\vct y_0\rangle$, then
\begin{align*}
   \dist(\vct A,\P) & = \|\DA\| \geq \|\DA\vct x_0\| 
   \geq \|\Pi_D(-\DA\vct x_0)\| = \max_{\vct y\in B^n\cap D}\langle -\DA\vct x_0,\vct y\rangle
\\ & \geq \langle -\DA\vct x_0,\vct y_0\rangle \geq \langle \vct{Ax}_0,\vct y_0\rangle = \|\Pi_D(\vct{Ax}_0)\| \geq \min_{\vct x\in C\cap S^{m-1}} \|\Pi_D(\vct{Ax})\| = \sres{A}{C}{D} .
\end{align*}

For the reverse inequality $\dist(\vct A,\P)\leq \sres{A}{C}{D}$ we need to construct a perturbation~$\DA$ such that $\vct A+\DA\in\P$ and $\|\DA\|\leq\sres{A}{C}{D}$. Let $\vct x_0\in C\cap S^{m-1}$ and $\vct y_0\in D\cap B^n$ such that
  \[ \sres{A}{C}{D} = \min_{\vct x\in C\cap S^{m-1}} \max_{\vct y\in D\cap B^n} \langle \vct{Ax},\vct y\rangle = \langle \vct{Ax}_0,\vct y_0\rangle . \]
Since $\vct A\not\in\P$ we have $\sres{A}{C}{D}>0$, which implies $\|\vct y_0\|=1$, i.e., $\vct y_0\in D\cap S^{n-1}$. We define
  \[ \DA := -\vct y_0\vct y_0^T\vct A . \]
Note that
  \[ \|\DA\| = \|\vct A^T\vct y_0\| \leq \langle \vct A^T\vct y_0,\vct x_0\rangle = \sres{A}{C}{D} . \]
Furthermore,
\begin{align*}
   (\vct A+\DA)\vct x_0 & = \vct{Ax}_0 - \vct y_0\vct y_0^T\vct{Ax}_0 = \vct{Ax}_0 - \langle \vct{Ax}_0,\vct y_0\rangle\vct y_0 = \vct{Ax}_0 - \Pi_D(\vct{Ax}_0) = \Pi_{D^\polar}(\vct{Ax}_0) .
\end{align*}
So $\vct x_0\in C\setminus\{\vct0\}$ and $(\vct A+\DA)\vct x_0\in D^\polar$, which shows that $\vct A+\DA\in\P$, and hence $\dist(\vct A,\P)\leq \|\DA\| \leq \sres{A}{C}{D}$.
\end{proof}

\section{A new variant of Gordon's comparison theorem}\label{sec:proof-comp-thm}

Underlying some of our analysis is a new variant of Gordon's comparison theorem. For completeness we first recall the familiar version of
Gordon's inequality~\cite{G:85}, see also~\cite{G:87} and~\cite[Chapter 8]{FR:13} for a simplified derivation.

\begin{theorem}[Gordon]\label{thm:gordon}
Let $X_{ij},Y_{ij}$, $1\leq i\leq m$, $1\leq j\leq n$, be centered Gaussian random variables, and assume that
\begin{align*}
   \Expect |X_{ij}-X_{k\ell}|^2 & \leq \Expect |Y_{ij}-Y_{k\ell}|^2 , & & \hspace{-2cm} \text{for all } i\neq k \text{ and } j,\ell ,
\\ \Expect |X_{ij}-X_{i\ell}|^2 & \geq \Expect |Y_{ij}-Y_{i\ell}|^2 , & & \hspace{-2cm} \text{for all } i,j,\ell .
\intertext{Then $\Expect \min_i\max_j X_{ij} \geq \Expect \min_i\max_j Y_{ij}$. If additionally}
   \Expect X_{ij}^2 & = \Expect Y_{ij}^2 , & & \hspace{-2cm} \text{for all } i,j ,
\end{align*}
then for any monotonically increasing function $f\colon\IR\to\IR$,
  \[ \Expect \min_i\max_j f(X_{ij}) \geq \Expect \min_i\max_j f(Y_{ij}) . \]
\end{theorem}

Slepian's lemma is obtained by setting $m=1$ in Gordon's theorem. The following theorem (in the degenerate case of $X_0=Y_0=0$) lies somewhere in the middle between the two cases treated by Gordon's theorem.

\begin{theorem}\label{prop:gordon-with-0}
Let $X_0,Y_0,X_{ij},Y_{ij}$, $1\leq i\leq m$, $1\leq j\leq n$, be centered Gaussian random variables, and assume that
\begin{align*}
   \Expect |X_{ij}-X_{k\ell}|^2 & \leq \Expect |Y_{ij}-Y_{k\ell}|^2 , & & \hspace{-2cm} \text{for all } i\neq k \text{ and } j,\ell ,
\\ \Expect |X_{ij}-X_{i\ell}|^2 & \geq \Expect |Y_{ij}-Y_{i\ell}|^2 , & & \hspace{-2cm} \text{for all } i,j,\ell ,
\\ \Expect |X_{ij}-X_0|^2 & \geq \Expect |Y_{ij}-Y_0|^2 , & & \hspace{-2cm} \text{for all } i,j .
\end{align*}
Then for any monotonically increasing convex function $f\colon\IR_+\to\IR$,
\begin{equation}\label{eq:E[f(minmax(0,X_(ij)-X0))]>=...}
  \Expect \min_i\max_j f_+(X_{ij}-X_0) \geq \Expect \min_i\max_j f_+(Y_{ij}-Y_0) ,
\end{equation}
where $f_+(x):=f(x)$, if $x\geq0$, and $f_+(x):=f(0)$, if $x\leq0$.
\end{theorem}

In Section~\ref{sec:lin_imag} we provide an example, which shows that~\eqref{eq:E[f(minmax(0,X_(ij)-X0))]>=...} may fail if~$f$ is not convex. 
The proof we present is based on a geometric reduction from Maurer~\cite{M:11}, cf.~Lemma~\ref{lem:geom_maurer}.

In the following we fix a monotonically increasing convex function $f\colon\IR_+\to\IR$, which is differentiable on $(0,\infty)$ and satisfies $\lim_{x\to0+}f'(x)=0$. The extension $f_+\colon\IR\to\IR$, with $f_+(x):=f(x)$, if $x\geq0$, and $f_+(x):=f(0)$, if $x\leq0$, is thus monotonically increasing, convex, and differentiable on~$\IR$. On the Euclidean space $\IR\times\IR^{m\times n}$, whose elements we denote by $\vct x=(x_0,x_{11},\ldots,x_{mn})$, we define $F\colon\IR\times\IR^{m\times n}\to\IR$ by
\begin{equation}\label{eq:def-F}
  F(\vct x) := \min_i\max_j f_+(x_{ij}-x_0).
\end{equation}
This function is differentiable almost everywhere. More precisely, it is differentiable if
  \[ \min_i\max_j x_{ij}<x_0 \quad\text{or}\quad \big|\big\{(k,\ell)\mid x_{k\ell} = \min_i\,\max_j\, \max\{x_{ij},x_0\}>x_0\big\}\big|=1 . \]
In the first case $\nabla F(\vct x)=0$. In the second case $\nabla F(\vct x)$ is zero except for the $(k,\ell)$th entry, $x_{k\ell} = \min_i\,\max_j\, \max\{x_{ij},x_0\}$ ($>x_0$), which is given by~$f'(x_{k\ell}-x_0)$. So, if $\vct x(t)$ is a differentiable curve through~$\vct x$ with $\vct x(0)=\vct x$, $\dot{\vct{x}}:=\dot{\vct{x}}(0)$, then
\begin{equation}\label{eq:deriv_F}
  \tfrac{d}{dt} F(\vct x(t))\big|_{t=0} = \ip{\nabla F(\vct{x})}{\dot{\vct{x}}} = \dot x_{k\ell} f'(x_{k\ell}-x_0) .
\end{equation}

\begin{lemma}\label{lem:part-deriv}
Let $X_0$ and $X_{ij}$, $1\leq i\leq m$, $1\leq j\leq n$, be centered Gaussian random variables such that their joint covariance matrix has full rank. Fix $1\leq k_0,k\leq m$ and $1\leq \ell_0,\ell\leq n$ with $(k_0,\ell_0)\neq(k,\ell)$, and let $Y,Z$ be Gaussians, defined in one of the two following ways:
\begin{enumerate}
  \item $X_{k_0\ell_0}=Y+Z$ with $Z$ independent of $Y,X_0,X_{ij}$, for all $(i,j)\neq (k_0,\ell_0)$,
  \item $X_0=Y+Z$ with $Z$ independent of $Y,X_{ij}$, for all $(i,j)$.
\end{enumerate}
If $\vct X(t)$ is 
defined by
\begin{equation}\label{eq:def-X_(ij)(t)-nonzero}
  X_0(t) := X_0 ,\qquad X_{ij}(t):=X_{ij} ,\quad\text{for } (i,j)\neq (k,\ell) ,\qquad X_{k\ell}(t) := X_{k\ell}+tZ ,
\end{equation}
then
  \[ \tfrac{d}{dt}\Expect\big[ F(\vct X(t))\big]\big|_{t=0} \begin{cases} \leq 0 & \text{if~$Y,Z$ defined as in (1) and $k=k_0$, or $Y,Z$ defined as in (2)} \\ \geq 0 & \text{if~$Y,Z$ defined as in (1) and $k\neq k_0$} . \end{cases} \]
\end{lemma}

\begin{proof}
We distinguish between the cases (1) and (2).

(1) Let $X_{k_0\ell_0}=Y+Z$ as described above. We define $\vct X^+(t),\vct X^-(t)$ by 
\begin{align*}
   X_0^+(t) & := X_0^-(t) := X_0, & X_{ij}^+(t) & := X_{ij}^-(t) := X_{ij} ,\quad \text{if } (i,j)\not\in\{(k,\ell),(k_0,\ell_0)\} ,\hspace{-6cm}
\\[1mm] X_{k_0\ell_0}^+(t) & := Y + |Z| , & X_{k_0\ell_0}^-(t) & := Y - |Z| , & X_{k\ell}^+(t) & := X_{k\ell}+t|Z| , & X_{k\ell}^-(t) & := X_{k\ell}-t|Z| ,
\end{align*}
and denote $X_{ij}^+:=X_{ij}^+(0)$ and $X_{ij}^-:=X_{ij}^-(0)$. Since $Z$ is independent of $Y,X_0,X_{ij}$, for $(i,j)\neq (k_0,\ell_0)$, we have
  \[ \Expect\big[ F(\vct X(t))\big] = \tfrac{1}{2} \Expect\big[ F(\vct X^+(t)) + F(\vct X^-(t)) \big] .  \]
To simplify the notation, we set
  \[ \dot{\vct X}^+ = \tfrac{d}{dt} \vct X^+(t)\big|_{t=0} ,\qquad \dot{\vct X}^- = \tfrac{d}{dt} \vct X^-(t)\big|_{t=0} . \]
A standard argument involving Lebesgue's dominated convergence theorem shows that
\begin{align*}
 \tfrac{d}{dt}\Expect\big[F(\vct X^+(t))\big] & =\Expect\big[\tfrac{d}{dt}F(\vct X^+(t))\big], &
 \tfrac{d}{dt}\Expect\big[F(\vct X^-(t))\big] & =\Expect\big[\tfrac{d}{dt}F(\vct X^-(t))\big],
\end{align*}
so that it is enough to show that almost surely
\begin{equation}\label{eq:expect<0_reduced}
  \tfrac{d}{dt}F(\vct X^+(t))+\tfrac{d}{dt}F(\vct X^-(t)) =
   \ip{\nabla F(\mtx{X}^+)}{\dot{\vct X}^+}+\ip{\nabla F(\mtx{X}^-)}{\dot{\vct X}^-} \begin{cases} \leq 0 & \text{if } k=k_0 \\ \geq 0 & \text{if } k\neq k_0 . \end{cases}
\end{equation}
Note that $X_{k\ell}^+=X_{k\ell}$ and $X_0^+=X_0$. By~\eqref{eq:deriv_F}, almost surely
\begin{align*}
   \ip{\nabla F(\mtx{X}^+)}{\dot{\vct X}^+}  & = \begin{cases}
      \dot X_{k\ell}^+ \, f'(X_{k\ell}-X_0) & \text{if } X_{k\ell} = \min_i\max_j\max\{X_{ij}^+,X_0\} > X_0 \text{ and}
   \\ & \hspace{2mm} X_{k'\ell'}^+ \neq \min_i\max_j\max\{X_{ij}^+,X_0\} \text{ for all } (k',\ell')\neq (k,\ell)
   \\ 0 & \text{else (almost surely)} ,
   \end{cases}
\end{align*}
and similarly for 
$\ip{\nabla F(\mtx{X}^-)}{\dot{\vct X}^-}$.
Note that if $X_{k\ell} = \min_i\max_j\max\{X_{ij}^+,X_0\} > X_0$, then almost surely $X_{k'\ell'}^+ \neq \min_i\max_j\max\{X_{ij}^+,X_0\}$ for all $(k',\ell')\neq (k,\ell)$, so we may skip this additional condition. Since $\dot X_{k\ell}^+=|Z|$ and $\dot X_{k\ell}^-=-|Z|$ and by the monotonicity of~$f$, we have 
$\ip{\nabla F(\mtx{X}^+)}{\dot{\vct X}^+}\geq0$ and 
$\ip{\nabla F(\mtx{X}^-)}{\dot{\vct X}^-}\leq 0$.

If $k=k_0$ then $X_{k\ell} = \min_i\max_j\max\{X_{ij}^+,X_0\} > X_0$ implies $X_{k\ell} = \min_i\max_j\max\{X_{ij}^-,X_0\} > X_0$, and in this case 
$\ip{\nabla F(\mtx{X}^+)}{\dot{\vct X}^+}+\ip{\nabla F(\mtx{X}^-)}{\dot{\vct X}^-}=0$. Since this is the only case in which $\ip{\nabla F(\mtx{X}^+)}{\dot{\vct X}^+}$
is nonzero (with positive probability), we have almost surely 
$\ip{\nabla F(\mtx{X}^+)}{\dot{\vct X}^+}+\ip{\nabla F(\mtx{X}^-)}{\dot{\vct X}^-}\leq 0$.

If $k\neq k_0$ then $X_{k\ell} = \min_i\max_j\max\{X_{ij}^-,X_0\} > X_0$ implies $X_{k\ell} = \min_i\max_j\max\{X_{ij}^+,X_0\} > X_0$, and in this case 
$\ip{\nabla F(\mtx{X}^+)}{\dot{\vct X}^+}+\ip{\nabla F(\mtx{X}^-)}{\dot{\vct X}^-}=0$. Since this is the only case in which $\ip{\nabla F(\mtx{X}^-)}{\dot{\vct X}^-}$
is nonzero (with positive probability), we have almost surely 
$\ip{\nabla F(\mtx{X}^+)}{\dot{\vct X}^+}+\ip{\nabla F(\mtx{X}^-)}{\dot{\vct X}^-}=0$.

This settles the first case.

\smallskip

(2) Let $X_0=Y+Z$ as described above. We define $\vct X^+(t),\vct X^-(t)$ by
\begin{align*}
   & & X_{ij}^+(t) & := X_{ij}^-(t) := X_{ij} ,\quad \text{if } (i,j)\neq (k,\ell) ,\hspace{-6cm}
\\[1mm] X_0^+(t) & := Y + |Z| , & X_0^-(t) & := Y - |Z| , & X_{k\ell}^+(t) & := X_{k\ell}+t|Z| , & X_{k\ell}^-(t) & := X_{k\ell}-t|Z| .
\end{align*}
Again, from the independence assumption on $Z$ we obtain $\Expect[ F(\vct X(t))] = \tfrac{1}{2} \Expect[ F(\vct X^+(t)) + F(\vct X^-(t))]$, and it suffices to show that almost surely
\begin{equation}\label{eq:expect<0_reduced_zero}
\ip{\nabla F(\mtx{X}^+)}{\dot{\vct X}^+}+\ip{\nabla F(\mtx{X}^-)}{\dot{\vct X}^-} \leq 0
\end{equation}
Note that $X_{ij}^+=X_{ij}$ for all $(i,j)$. By~\eqref{eq:deriv_F},
\begin{align*}
   \ip{\nabla F(\mtx{X}^+)}{\dot{\vct X}^+} & = \begin{cases}
      \dot X_{k\ell}^+ \, f'(X_{k\ell}-X_0^+) & \text{if } X_{k\ell} = \min_i\max_j\max\{X_{ij},X_0^+\} > X_0^+ \text{ and}
   \\ & \hspace{2mm} X_{k'\ell'} \neq \min_i\max_j\max\{X_{ij},X_0^+\} \text{ for all } (k',\ell')\neq (k,\ell)
   \\ 0 & \text{else (almost surely)} ,
   \end{cases}
\end{align*}
and similarly for $\ip{\nabla F(\mtx{X}^-)}{\dot{\vct X}^-}$. As in the first case, we can skip the additional uniqueness condition, which is almost surely satisfied, and again, $\dot X_{k\ell}^+=|Z|$ and $\dot X_{k\ell}^-=-|Z|$ and the monotonicity of~$f$ imply $\ip{\nabla F(\mtx{X}^+)}{\dot{\vct X}^+}\geq0$ and $\ip{\nabla F(\mtx{X}^-)}{\dot{\vct X}^-}\leq 0$.

Now, $X_{k\ell} = \min_i\max_j\max\{X_{ij},X_0^+\} > X_0^+$ implies $X_{k\ell} = \min_i\max_j\max\{X_{ij},X_0^-\} > X_0^-$, and in this case $X_{k\ell}-X_0^+\leq X_{k\ell}-X_0^-$. By convexity of~$f$ it follows that $f'(X_{k\ell}-X_0^+)\leq f'(X_{k\ell}-X_0^-)$, and thus $\ip{\nabla F(\mtx{X}^+)}{\dot{\vct X}^+}+\ip{\nabla F(\mtx{X}^-)}{\dot{\vct X}^-}\leq 0$. Since this is the only case in which $\ip{\nabla F(\mtx{X}^+)}{\dot{\vct X}^+}$ is nonzero, we have almost surely $\ip{\nabla F(\mtx{X}^+)}{\dot{\vct X}^+} + \ip{\nabla F(\mtx{X}^-)}{\dot{\vct X}^-} \leq 0$. This settles the second case.
\end{proof}

It remains to show that Lemma~\ref{lem:part-deriv} indeed implies Theorem~\ref{prop:gordon-with-0}. We deduce this from general geometric arguments as used in~\cite{M:11}. We reproduce these arguments in the following for convenience of the reader, except for Lemma~\ref{lem:factor_D}, which is a copy of~\cite[Lem.~4]{M:11} and follows from well-known properties of Euclidean distance matrices, cf.~\cite{KW:12} for a recent survey on this theory.

In the following let $\mE$ denote $d$-dimensional Euclidean space~$\R^d$. Recall that a function $\Phi\colon\mE^k\to \IR$ is Euclidean motion invariant if
  \[ \Phi(\vct x_1+\vct y,\ldots,\vct x_k+\vct y) = \Phi(\vct{Qx}_1,\ldots,\vct{Qx}_k) = \Phi(\vct x_1,\ldots,\vct x_k) \]
for all $\vct x_1,\ldots,\vct x_k,\vct y\in\mE$ and $\vct Q\in O(\mE)$. We identify the general linear group on~$\mE$ with the set of bases of~$\mE$:
  \[ \GL(\mE) = \{(\vct x_1,\ldots,\vct x_d)\mid \vct x_i \text{ linear independent}\} \subset\mE^d . \]
Let $D\colon\mE^d\to\IR^{\binom{d}{2}}$ be defined by
\begin{equation}\label{eq:def-D}
  D(\vct x_1,\ldots,\vct x_d) := \big(\|\vct x_i-\vct x_j\|^2\big)_{i<j} ,
\end{equation}
and denote $\Delta:=D(\mE^d)$ and $\Delta_0:=D(\GL(\mE))$.

\begin{lemma}\label{lem:factor_D}
Let $\mE,D,\Delta,\Delta_0$ be defined as above. The sets $\Delta,\Delta_0$ are convex, $\Delta_0$ is open, and $\Delta$ is the closure of $\Delta_0$. Furthermore, any Euclidean motion invariant function $\Phi\colon\mE^d\to \IR$ factorizes uniquely over $D$,
  \[ \Phi = \vp\circ D ,\qquad \vp\colon \IR^{\binom{d}{2}}\to\IR . \]
Additionally, if $\Phi$ is continuous, then so is~$\vp$, and if $\Phi$ is differentiable on~$\GL(\mE)$, then $\phi$ is differentiable on $\Delta_0$.
\end{lemma}

The following lemma describes the proof strategy as demonstrated in~\cite{M:11}.

\begin{lemma}\label{lem:geom_maurer}
Let $\mE,D,\Delta,\Delta_0$ be defined as above and let $\Phi\colon\mE^d\to \IR$ be a continuous Euclidean motion invariant function, which is differentiable on~$\GL(\mE)$. Assume that for some symmetric sign matrix $\vct S\in\{1,-1\}^{d\times d}$ the following holds: for every basis $(\vct x_1,\ldots,\vct x_d)\in\GL(\mE)$ and all $i_0,j_0$ there exists a curve~$(\vct x_1(t),\ldots,\vct x_d(t))$, $\vct x_i(0)=\vct x_i$, such that
  \[ s_{ij}\, \tfrac{d}{dt} \|\vct x_i(t)-\vct x_j(t)\|^2\big|_{t=0} \begin{cases} < 0 & \text{if } (i,j)=(i_0,j_0) \\ = 0 & \text{else} , \end{cases} \qquad\text{and}\qquad \tfrac{d}{dt} \Phi(\vct x_1(t),\ldots,\vct x_d(t))\big|_{t=0} \leq 0 . \]
Then for all $(\vct x_1,\ldots,\vct x_d),(\vct y_1,\ldots,\vct y_d)\in\mE^d$ satisfying $s_{ij}\|\vct x_i-\vct x_j\|^2 \geq s_{ij}\|\vct y_i-\vct y_j\|^2$ for all~$i<j$, we have
\begin{equation}\label{eq:claim-lem-Maur}
  \Phi(\vct x_1,\ldots,\vct x_d) \geq \Phi(\vct y_1,\ldots,\vct y_d) .
\end{equation}
\end{lemma}

\begin{proof}
Using the decomposition $\Phi=\vp\circ D$, we can paraphrase the claim in terms of~$\vp$. For this, we define
  \[ \big\{(a_{ij})_{i<j}\in\IR^{\binom{d}{2}}\mid s_{ij} a_{ij}\leq 0 \text{ for all } i<j\big\} =: C_{\vct S} , \]
which is an isometric image of the nonnegative orthant $\IR_+^{\binom{d}{2}}$. The claim of the lemma is that for all $(a_{ij}),(b_{ij})\in \Delta$ with $(b_{ij}-a_{ij})\in C_{\vct S}$ we have $\vp(a_{ij}) \geq \vp(b_{ij})$.

By continuity of~$\Phi$ it suffices to show the claim~\eqref{eq:claim-lem-Maur} for bases $(\vct x_1,\ldots,\vct x_d),(\vct y_1,\ldots,\vct y_d)\in\GL(\mE)$. In terms of~$\vp$ the claim can then be restated by saying that for any point $(a_{ij})\in \Delta_0$ the derivative of~$\vp$ is nonpositive in any direction~$(v_{ij})\in C_{\vct S}$. By linearity of the derivative of~$\vp$ and by convexity of~$C_{\vct S}$, it suffices to show the monotonicity of~$\vp$ in the extreme directions of the cone~$C_{\vct S}$. Choosing such an extreme direction $(v_{ij})$ with $v_{ij} < 0$ if $(i,j)=(i_0,j_0)$ and $v_{ij}=0$ if $(i,j)\neq(k,\ell)$, and letting the curve $\vct X(t) = (\vct x_1(t),\ldots,\vct x_d(t))$ be such that $\frac{d}{dt} D(\vct X(t))\big|_{t=0}=(v_{ij})$, we obtain
  \[ \nabla_{(a_{ij})} \vp(v_{ij}) = \tfrac{d}{dt} \Phi(\vct x_1(t),\ldots,\vct x_d(t))\big|_{t=0} \leq 0 \]
by assumption. This shows the monotonicity of~$\vp$ in direction~$C_{\vct S}$ and thus proves the claim.
\end{proof}

\begin{proof}[Proof of Theorem~\ref{prop:gordon-with-0}]
By continuity we may assume that~$f$ is differentiable on $(0,\infty)$ and satisfies $\lim_{x\to0+}f'(x)=0$. We consider the Euclidean space $\mE=\IR\times\IR^{m\times n}$, and define
  \[ \Phi\colon\mE^{1+mn}\to \IR ,\qquad \Phi(\vct x_0,\vct x_{11},\ldots,\vct x_{mn}) := \Expect\big[ \min_i\max_j f_+(\langle \vct x_{ij}-\vct x_0,\vct g\rangle)\big] , \]
where $\vct g$ is a standard Gaussian vector in~$\mE$. The map~$\Phi$ is Euclidean motion invariant, continuous, and differentiable on~$\GL(\mE)$.
Setting $X_0 = \langle \vct x_0,\vct g\rangle$ and $X_{ij} = \langle \vct x_{ij},\vct g\rangle$, we have
  \[ \Expect |X_{ij}-X_{k\ell}|^2 = \|\vct x_{ij}-\vct x_{k\ell}\|^2 ,\qquad \Expect |X_{ij}-X_0|^2 = \|\vct x_{ij}-\vct x_0\|^2 , \]
and we can reformulate the claim of Theorem~\ref{prop:gordon-with-0} in terms of~$\Phi$:

If $(\vct x_0,\vct x_{11},\ldots,\vct x_{mn}),(\vct y_0,\vct y_{11},\ldots,\vct y_{mn})\in\mE^{1+mn}$ satisfy
\begin{align*}
   \|\vct x_{ij}-\vct x_{k\ell}\|^2 & \leq \|\vct y_{ij}-\vct y_{k\ell}\|^2 \;\; \text{if } i\neq k , & \|\vct x_{ij}-\vct x_{i\ell}\|^2 & \geq \|\vct y_{ij}-\vct y_{i\ell}\|^2 , & \|\vct x_{ij}-\vct x_0\|^2 & \geq \|\vct y_{ij}-\vct y_0\|^2 ,
\end{align*}
then $\Phi(\vct x_0,\vct x_{11},\ldots,\vct x_{mn})\geq \Phi(\vct y_0,\vct y_{11},\ldots,\vct y_{mn})$.

By Lemma~\ref{lem:geom_maurer} we obtain a different condition that we need to verify, and in the remainder of the proof we will show that Lemma~\ref{lem:part-deriv} is exactly this condition. We restrict to the presentation of case~(1), the second case follows analogously.

The decomposition $X_{k_0\ell_0}=Y+Z$ corresponds to the decomposition $\vct x_{k_0\ell_0} = \vct y+\vct z$ with $\vct y$ the orthogonal projection of $\vct x_{k_0\ell_0}$ on the linear span of $\vct x_0$ and $\vct x_{ij}$, $(i,j)\neq(k_0,\ell_0)$. Note that $\vct z\neq\vct0$. The curve $\vct X(t)$ defined in~\eqref{eq:def-X_(ij)(t)-nonzero} corresponds to the curve $(\vct x_0(t),\vct x_{11}(t),\ldots,\vct x_{mn}(t))$ in~$\mE^{1+mn}$ given by
  \[ \vct x_0(t) = \vct x_0 ,\qquad \vct x_{ij}(t) = \vct x_{ij} ,\quad\text{for } (i,j)\neq (k,\ell) ,\qquad x_{k\ell}(t) = x_{k\ell}+t\vct z . \]
We obtain
\begin{align*}
   \|\vct x_{k\ell}(t)-\vct x_0(t)\|^2 & = \|\vct x_{k\ell} + t\vct z-\vct x_0\|^2 = \|\vct x_{k\ell}-\vct x_0\|^2 + t^2 \|\vct z\|^2 ,
\\ \|\vct x_{k\ell}(t)-\vct x_{ij}(t)\|^2 & = \|\vct x_{k\ell} + t\vct z-\vct x_{ij}\|^2 = \|\vct x_{k\ell}-\vct x_{ij}\|^2 + t^2 \|\vct z\|^2 ,\quad \text{if } (i,j)\not\in \{(k,\ell),(k_0,\ell_0)\} ,
\\ \|\vct x_{k\ell}(t)-\vct x_{k_0\ell_0}(t)\|^2 & = \|\vct x_{k\ell} + t\vct z-\vct y-\vct z\|^2 = \|\vct x_{k\ell}-\vct y\|^2 + (t-1)^2 \|\vct z\|^2 ,
\end{align*}
and thus
\begin{align*}
   \tfrac{d}{dt} \|\vct x_{ij}(t)-\vct x_0(t)\|^2\big|_{t=0} & = 0 ,
\\ \tfrac{d}{dt} \|\vct x_{ij}(t)-\vct x_{i'j'}(t)\|^2\big|_{t=0} & = 0 ,\quad \text{if } \{(i,j),(i',j')\}\neq \{(k,\ell),(k_0,\ell_0)\} ,
\\ \tfrac{d}{dt} \|\vct x_{k\ell}(t)-\vct x_{k_0\ell_0}(t)\|^2\big|_{t=0} & = -2\|\vct z\|^2 .
\end{align*}
Hence, Lemma~\ref{lem:part-deriv} shows exactly the condition described in Lemma~\ref{lem:geom_maurer}, which finishes the proof.
\end{proof}

\end{document}